\documentclass[11pt]{article}
\pdfoutput=1

\usepackage{geometry}
\usepackage{amsmath}
\usepackage{amssymb}
\usepackage{amsthm}
\usepackage{mathtools}
\usepackage{mathrsfs}  

\usepackage{graphicx}
\usepackage{tikz}
\usepackage{tikz-cd}   
\usetikzlibrary{arrows.meta, positioning, decorations.markings, calc}
\usepackage[x11names,svgnames]{xcolor}

\usepackage{microtype}
\usepackage[T1]{fontenc}
\usepackage{lmodern}

\usepackage[numbers,sort&compress]{natbib}

\usepackage{hyperref}
\hypersetup{colorlinks=true,
            linkcolor=blue!70!black,
            citecolor=blue!70!black,
            urlcolor=blue!70!black}

\usepackage{doi}
\usepackage{aliascnt}
\usepackage[capitalise, nameinlink]{cleveref}

\usepackage{enumitem}
\usepackage{booktabs}  
\usepackage{appendix}  
\usepackage{caption}

\usepackage{amsfonts}
\usepackage{mathdots}

\usepackage{subcaption}   
\usepackage{float}         
\usepackage{stackengine}   
\usepackage{scalerel}      

\usepackage{comment}
\usepackage[colorinlistoftodos]{todonotes}

\usepackage{nicematrix}

\newtheoremstyle{ghriststyle}
    {12pt}                    
    {12pt}                    
    {\itshape}                
    {}                        
    {\scshape}                
    {.}                       
    {.5em}                    
    {}                        

\newtheoremstyle{ghristdefstyle}
    {12pt}                    
    {12pt}                    
    {\normalfont}             
    {}                        
    {\scshape}                
    {.}                       
    {.5em}                    
    {}                        

\newtheoremstyle{ghristremarkstyle}
    {8pt}                     
    {8pt}                     
    {\normalfont}             
    {}                        
    {\itshape}                
    {.}                       
    {.5em}                    
    {}                        

\theoremstyle{ghriststyle}
\newtheorem{theorem}{Theorem}[section]
\newaliascnt{lemma}{theorem}
\newtheorem{lemma}[lemma]{Lemma}
\aliascntresetthe{lemma}
\crefname{lemma}{Lemma}{lemmas}
\newaliascnt{proposition}{theorem}
\newtheorem{proposition}[proposition]{Proposition}
\aliascntresetthe{proposition}
\crefname{proposition}{Proposition}{propositions}
\newaliascnt{corollary}{theorem}

\aliascntresetthe{corollary}
\crefname{corollary}{Corollary}{corollaries}

\theoremstyle{ghristdefstyle}
\newtheorem{definition}[theorem]{Definition}

\theoremstyle{ghristremarkstyle}
\newaliascnt{remark}{theorem}
\newtheorem{remark}[remark]{Remark}
\aliascntresetthe{remark}
\crefname{remark}{Remark}{remarks}
\newaliascnt{note}{theorem}
\newtheorem{note}[note]{Note}
\aliascntresetthe{note}
\crefname{note}{Note}{notes}

\newcommand{\id}{\mathrm{id}}

\newcommand{\R}{\mathbb{R}}

\newcommand{\relu}{\mathrm{ReLU}}

\stackMath
\def\trianglelefteqslant{\ThisStyle{\mathrel{%
  \stackinset{r}{.75pt+.15\LMpt}{t}{.1\LMpt}%
    {\rule{.3pt}{1.1\LMex+.2ex}}{\SavedStyle\leqslant}%
}}}

\makeatletter

\newcommand{\thanksblock}[1]{\gdef\@thanksblock{#1}}
\gdef\@thanksblock{} 

\renewcommand{\maketitle}{%
  \begin{center}%
    {\LARGE\bfseries \@title \par}%
    \vskip 1em%
    {\large \lineskip .5em%
      \begin{tabular}[t]{c}%
        \@author
      \end{tabular}\par}%
    \ifx\@thanksblock\@empty\else
      \vskip .75em%
      {\small \@thanksblock\par}%
    \fi
  \end{center}%
  \vskip 1.5em%
}

\makeatother

\date{} 

\usepackage{titlesec}

\titleformat{\section}
    {\normalfont\large\bfseries}
    {\thesection.}
    {0.5em}
    {}

\titleformat{\subsection}
    {\normalfont\normalsize\bfseries}
    {\thesubsection.}
    {0.5em}
    {}

\titleformat{\subsubsection}
    {\normalfont\normalsize\itshape}
    {\thesubsubsection.}
    {0.5em}
    {}

\renewenvironment{abstract}{%
    \begin{center}
        \textsc{Abstract}
    \end{center}
    \begin{quotation}
    \noindent
    \small
}{%
    \end{quotation}
    \vspace{1em}
    \normalsize
}

\title{Neural Networks as Local-to-Global Computations}

\author{Vicente Bosca \and Robert Ghrist}

\thanksblock{%
Department of Mathematics, University of Pennsylvania\\
\texttt{vicenteg@sas.upenn.edu}\quad
\texttt{ghrist@math.upenn.edu}%
}

\begin{document}
\maketitle


\begin{abstract}
We construct a cellular sheaf from any feedforward ReLU neural network
by placing one vertex for each intermediate quantity in the forward
pass and encoding each computational step -- affine transformation,
activation, output -- as a restriction map on an edge.  The restricted
coboundary operator on the free coordinates is unitriangular, so its
determinant is~$1$ and the restricted Laplacian is positive definite
for every activation pattern.  It follows that the relative
cohomology vanishes and the forward pass output is the unique
harmonic extension of the boundary data.  The sheaf heat equation
converges exponentially to this output despite the state-dependent
switching introduced by piecewise linear activations.  Unlike the
forward pass, the heat equation propagates information
bidirectionally across layers, enabling pinned neurons that impose
constraints in both directions, training through local discrepancy
minimization without a backward pass, and per-edge diagnostics that
decompose network behavior by layer and operation type.  We validate
the framework experimentally on small synthetic tasks, confirming the
convergence theorems and demonstrating that sheaf-based training,
while not yet competitive with stochastic gradient descent, obeys
quantitative scaling laws predicted by the theory.
\end{abstract}



\section{Introduction}\label{sec:introduction}

A feedforward neural network computes by propagating data in one
direction: input enters at the first layer, passes through alternating
affine transformations and nonlinearities, and exits at the last.  In contrast,
a cellular sheaf on a graph assigns vector spaces to vertices and linear
maps to edges, encoding local consistency constraints whose global
resolution is governed bidirectionally by the sheaf Laplacian.  We connect the two:
the forward pass of a feedforward ReLU network is the unique harmonic
extension of boundary data on a cellular sheaf constructed from the
network's weights, biases, and activation functions.

The construction places one vertex for each intermediate quantity in
the forward pass -- pre-activations, post-activations, output -- and
one edge for each computational step.  Weight matrices and biases are
absorbed into extended restriction maps on the weight edges; the ReLU
nonlinearity becomes a diagonal projection on the activation edges.
Input data and bias coordinates are held fixed as boundary conditions.
The mathematical core of the embedding is a single structural
observation: the restricted coboundary operator $\delta_\Omega$ on the
free coordinates is a square, block lower-unitriangular matrix
(\cref{lemma:det1}).  The identity blocks on the diagonal come from the
downstream endpoint of each edge; the sub-diagonal blocks are the
restriction maps themselves -- ReLU projections, weight matrices, and
identity maps -- none of which contribute to the determinant.
Unitriangularity gives $\det\delta_\Omega = 1$, hence the restricted
Laplacian $\delta_\Omega^T\delta_\Omega$ is positive definite, the
relative cohomology vanishes, and the harmonic extension of the
boundary data is unique.  Forward substitution in the unitriangular
system reproduces the forward pass equations layer by layer
(\cref{prop:forward-pass}).

The sheaf heat equation converges exponentially to this harmonic
extension (\cref{thm:convergence}), but unlike the forward pass, it
propagates information bidirectionally across the graph.  Convergence
is not immediate: the ReLU creates a state-dependent sheaf whose
restriction maps switch discontinuously as pre-activations cross zero,
producing a continuous piecewise affine (CPWA) dynamical system.  The
convergence proof combines a Lyapunov argument -- the total discrepancy
$\|\delta\,x\|^2$ decreases along trajectories, including during
Filippov sliding at switching surfaces -- with a structural observation:
the feedforward cascade excludes spurious critical points on switching
surfaces for generic weights.  A second convergence theorem
(\cref{thm:final}) extends the result to nonlinear output activations
(sigmoid, softmax) via local adjoint techniques, drawing on the CPWA
convergence theory developed by Gould~\cite{gould2025thesis}.

This identification supports three constructions that exploit the
bidirectional and local character of sheaf diffusion.
\begin{enumerate}
\item 
\emph{Pinned neurons.}
Constraining a hidden neuron to a target value adds a penalty edge to
the sheaf graph.  The heat equation then propagates the constraint
upstream and downstream simultaneously, producing an equilibrium that
reflects a global compromise across the entire network rather than a
local downstream override.  The construction generalizes the weighted
reluctance mechanism of Hansen and Ghrist~\cite{hansen2021opinion} to
the neural setting (\cref{subsec:pinning}).
More generally, for any frozen activation pattern, clamping an
arbitrary mixture of inputs, hidden coordinates, and outputs yields a
local-minimizer completion of the rest
(\cref{note:partial-clamping}), turning a pretrained feedforward network
into a boundary-value problem solver.

\item 
\emph{Sheaf-based training.}
Evolving the trainable restriction maps (weight matrices) alongside the
$0$-cochain through local discrepancy minimization yields a training
algorithm in which each weight update depends only on the data at
adjacent vertices -- no backward pass is required in the usual
reverse-mode sense.  This places the construction near other local or
relaxation-based alternatives to backpropagation, including equilibrium
propagation and predictive-coding-style learning
\cite{scellier2017equilibrium,whittington2017predictive,millidge2022predictive},
while remaining distinct in that the computation is organized as sheaf
diffusion minimizing a global discrepancy functional.  Seely~\cite{seely2025sheaf} organizes linear predictive coding as sheaf
diffusion in a similar spirit, though the learning setting differs
(bilevel alternation of inference and weight updates rather than coupled
state-parameter dynamics).

The resulting joint
dynamics on coupled state and parameter sheaves extend the
opinion-expression framework of~\cite{bosca2026selective}.  A
stagnation bound from that framework predicts the scaling
$\beta \sim 1/n_{\mathrm{train}}$ for the weight learning rate, which
we confirm experimentally.  On small synthetic tasks, sheaf-based
training converges on all configurations tested but does not match the
final loss achieved by stochastic gradient descent
(\cref{subsec:training,subsec:timescale}).

\item 
\emph{Structural diagnostics.}
The sheaf Laplacian's spectrum and the per-edge discord decomposition
provide diagnostic tools with no feedforward counterpart.  The Fiedler
eigenvector consistently concentrates on the output pre-activation
stalk, identifying it as the information-flow bottleneck.  The spectral
gap stabilizes before the training loss does, suggesting that diffusion
geometry is set early in optimization.  These diagnostics apply to any
trained network embedded post hoc as a sheaf, not only to
sheaf-trained networks (\cref{subsec:diagnostics}).
\end{enumerate}

Our work is distinct from sheaf neural
networks~\cite{hansen2020sheaf,bodnar2022neural,barbero2022connection,hernandez2024joint}, which design new
graph neural network architectures by running sheaf diffusion as a
message-passing layer, and from quiver
representations~\cite{seigal2023principal,parada2020quiver}, which
describe linear network structure algebraically.  We do not propose a
new architecture: we embed an existing one in a sheaf-theoretic
framework and develop the consequences.  

The closest parallel is the concurrent work of
Seely~\cite{seely2025sheaf}, who formalizes linear predictive coding
networks as cellular sheaves and uses sheaf cohomology and Hodge
decompositions to diagnose learning failures in recurrent topologies.
Our construction handles nonlinear activations and affine maps (bias
terms); Seely's linear setting analyzes recurrent architectures and
exact cohomological phenomena that the piecewise linear case has not
yet reached.

Our construction draws on
machinery from the opinion dynamics framework
of~\cite{bosca2026selective} -- stubborn opinions become pinned neurons,
joint dynamics on state and structure sheaves become simultaneous
training of activations and weights, and timescale separation bounds
become learning rate schedules -- as well as the CPWA convergence theory
of Gould~\cite{gould2025thesis}.

The paper is organized as follows.
\Cref{sec:background} reviews feedforward networks and cellular sheaf
diffusion.  \Cref{sec:sheafyNN} presents the construction and proves
the harmonic extension result.  \Cref{sec:convergence} establishes
convergence.  \Cref{sec:extensions} develops pinned neurons, training,
and batch processing.  \Cref{sec:experiments} reports experimental
results.  \Cref{sec:discussion} discusses limitations, extensions to
other architectures, and open problems.


\section{Background}\label{sec:background}

Our construction bridges two mathematical worlds: feedforward ReLU networks
(\S\ref{subsec:NNs}) and cellular sheaves with their associated diffusion
dynamics (\S\ref{subsec:sheaves}). We review each in turn, establishing the
notation and results that \S\ref{sec:sheafyNN} will bring together. Readers
comfortable with either topic may skim the corresponding subsection; the
essential points to carry forward are the piecewise linear structure of ReLU
networks (Remark~\ref{rem:CPWA}) and the convergence of sheaf diffusion to
harmonic extensions (Theorem~\ref{thm:HG-restricted}).

\subsection{Feedforward ReLU Networks}
\label{subsec:NNs}

Consider a feedforward neural network with $k$ hidden layers and input
$\mathbf{x} \in \R^{n_0}$. Layer $\ell \in \{1, \ldots, k{+}1\}$ has $n_\ell$
neurons, weight matrix $W^{(\ell)} \in \R^{n_\ell \times n_{\ell-1}}$, and bias
vector $b^{(\ell)} \in \R^{n_\ell}$. Writing $\mathbf{a}^{(0)} \equiv \mathbf{x}$
for the input, each hidden layer computes
\begin{equation}\label{eq:forward-pass}
  \mathbf{z}^{(\ell)} = W^{(\ell)} \mathbf{a}^{(\ell-1)} + b^{(\ell)}, \qquad
  \mathbf{a}^{(\ell)} = \relu(\mathbf{z}^{(\ell)}), \qquad 1 \leq \ell \leq k,
\end{equation}
where $\relu(\cdot) = \max(\cdot, 0)$ is applied element-wise and
$\mathbf{z}^{(\ell)}$ denotes the pre-activation vector. The output layer
applies a final activation~$\phi$ chosen according to the task (softmax for
classification, identity for regression):
\begin{equation}\label{eq:output-layer}
  \hat{\mathbf{y}} = \phi\bigl(W^{(k+1)} \mathbf{a}^{(k)} + b^{(k+1)}\bigr).
\end{equation}

Each layer performs an affine transformation followed by a nonlinearity, and
the full network is their composition. In
\cref{sec:sheafyNN}, we rewrite each of these steps as a
\emph{linear} map between appropriately extended vector spaces, so that the
entire forward pass can be encoded in the restriction maps of a cellular
sheaf.

\begin{remark}[Piecewise linear structure]\label{rem:CPWA}
  The ReLU activation partitions $\R^{n_0}$ into finitely many convex
  polyhedral regions determined by the sign patterns of the pre-activations
  $\mathbf{z}^{(\ell)}$. Within each region the network implements an affine
  map, but the global function is nonlinear due to the region-dependent
  activation pattern. This is the \emph{continuous piecewise affine} (CPWA) structure studied
systematically in~\cite{montufar2014linearregions,gould2025thesis}; see
also~\cite{liu2023relu, bosca2025topological}. We return to this structure in
  \cref{sec:convergence}, where the state-dependence of the activation
  pattern on the evolving dynamics creates a switched system.
\end{remark}

\subsection{Cellular Sheaves and Sheaf Diffusion}
\label{subsec:sheaves}

Cellular sheaves assign vector spaces and linear maps to a graph in a way that
encodes local consistency constraints. The theory was initiated by
Curry~\cite{curry2014sheaves}; spectral and Laplacian aspects were developed by
Hansen and Ghrist~\cite{hansen2019toward}. We follow the
conventions of ~\cite{bosca2026selective}, which in turn
builds on the opinion dynamics framework
of~\cite{hansen2021opinion}. Readers familiar with sheaf Laplacians and
harmonic extension may skip to \cref{sec:sheafyNN}.

Let $G = (V, E)$ be a finite connected graph with a chosen orientation on each
edge.

\begin{definition}\label{def:cellular-sheaf}
A \emph{cellular sheaf} $\mathcal{F}$ on $G$ assigns:
\begin{enumerate}[label=(\roman*)]
  \item a finite-dimensional real inner product space $\mathcal{F}(v)$, called
    the \emph{stalk}, to each vertex $v \in V$,
  \item a finite-dimensional real inner product space $\mathcal{F}(e)$ to each
    edge $e \in E$, and
  \item a linear map
    $\mathcal{F}_{v \trianglelefteqslant e} \colon \mathcal{F}(v) \to
    \mathcal{F}(e)$, called a \emph{restriction map}, for each incident pair
    $v \trianglelefteqslant e$.
\end{enumerate}
\end{definition}

We identify each stalk with $\R^n$ for the appropriate~$n$, each inner product
with the standard one, and each restriction map with its representing matrix.
The spaces of \emph{$0$-cochains} and \emph{$1$-cochains} bundle all vertex and
edge data:
\begin{equation}\label{eq:cochains}
  C^0(G; \mathcal{F}) = \bigoplus_{v \in V} \mathcal{F}(v), \qquad
  C^1(G; \mathcal{F}) = \bigoplus_{e \in E} \mathcal{F}(e),
\end{equation}
each equipped with the direct sum inner product. A $0$-cochain
$x \in C^0(G; \mathcal{F})$ assigns a vector $x_v \in \mathcal{F}(v)$ to every
vertex. The \emph{coboundary operator}
$\delta \colon C^0(G; \mathcal{F}) \to C^1(G; \mathcal{F})$ is defined on each
oriented edge $e = u \to v$ by
\begin{equation}\label{eq:coboundary}
  (\delta x)_e
    = \mathcal{F}_{v \trianglelefteqslant e}\, x_v
    - \mathcal{F}_{u \trianglelefteqslant e}\, x_u.
\end{equation}
The value $(\delta x)_e$ measures the \emph{discrepancy} on edge $e$: it
vanishes precisely when the restriction maps applied to the data at both
endpoints produce the same vector in the edge stalk. The
\emph{sheaf Laplacian}
\begin{equation}\label{eq:laplacian}
  L_\mathcal{F} = \delta^T \delta \colon C^0(G; \mathcal{F}) \to C^0(G; \mathcal{F})
\end{equation}
is positive semidefinite, and the quadratic form
$\langle x, L_\mathcal{F} x \rangle = \|\delta x\|^2$ measures the total
discrepancy in the network. Its local action at vertex $v$ is
\begin{equation}\label{eq:laplacian-local}
  (L_\mathcal{F}\, x)_v
    = \sum_{e:\, v \trianglelefteqslant e}
      \mathcal{F}_{v \trianglelefteqslant e}^T
      \bigl(
        \mathcal{F}_{v \trianglelefteqslant e}\, x_v
        - \mathcal{F}_{w \trianglelefteqslant e}\, x_w
      \bigr),
\end{equation}
where the sum runs over edges incident to $v$ and $w$ denotes the other
endpoint of~$e$. Each vertex updates its data using only information from its
immediate neighbors, making the dynamics inherently local.

\begin{remark}[Opinion dynamics interpretation]\label{rem:opinion-dynamics}
  In the opinion dynamics framework
  of~\cite{hansen2021opinion,bosca2026selective}, the stalk $\mathcal{F}(v)$
  is an agent's private opinion space, the edge stalk $\mathcal{F}(e)$ is a
  shared discourse space, and restriction maps encode how agents express
  private opinions to neighbors. The coboundary measures disagreement, and
  the sheaf Laplacian drives agents toward consensus. This vocabulary will
  occasionally resurface: we refer to the quadratic form
  $\|\delta x\|^2$ as the \emph{discord} of the $0$-cochain $x$, and to
  vertices with fixed data as \emph{stubborn}.
\end{remark}

The \emph{space of global sections} is
\begin{equation}\label{eq:global-sections}
  H^0(G; \mathcal{F}) = \ker \delta = \ker L_\mathcal{F}.
\end{equation}
The equality $\ker \delta = \ker L_\mathcal{F}$ is the sheaf Hodge
theorem~\cite{hansen2019toward}. A global section is a $0$-cochain with
zero discrepancy on every edge: a consistent assignment of data across the
entire network.

The sheaf heat equation drives an arbitrary $0$-cochain toward such a
consistent state.

\begin{theorem}[Convergence of sheaf diffusion
  {\cite[Theorem~4.1]{hansen2021opinion}}]
\label{thm:HG-heat}
  Solutions to the sheaf heat equation
  \begin{equation}\label{eq:heat}
    \frac{dx}{dt} = -\alpha\, L_\mathcal{F}\, x, \qquad \alpha > 0,
  \end{equation}
  converge exponentially to the orthogonal projection of $x(0)$ onto
  $H^0(G; \mathcal{F})$. The rate of convergence is governed by the smallest
  positive eigenvalue of $L_\mathcal{F}$.
\end{theorem}

When a subset $U \subseteq V$ of vertices maintains fixed data, the relevant
dynamics act only on the complement $\Omega = V \setminus U$. Writing
$\omega = x|_\Omega$ for the restriction of the $0$-cochain to free vertices
and $u = x|_U$ for the fixed boundary data, the $U$-restricted dynamics take
the reduced form
\begin{equation}\label{eq:U-restricted}
  \frac{d\omega}{dt}
    = -\alpha\bigl(
        L_\mathcal{F}[\Omega,\Omega]\,\omega
        + L_\mathcal{F}[\Omega,U]\,u
      \bigr),
\end{equation}
where $L_\mathcal{F}[\cdot,\cdot]$ denotes the corresponding block of the
sheaf Laplacian.

\begin{theorem}[$U$-restricted dynamics
  {\cite[Theorem~5.1]{hansen2021opinion}}]
\label{thm:HG-restricted}
  The dynamics~\eqref{eq:U-restricted} converge exponentially to the
  \emph{harmonic extension} of $u$: the unique $0$-cochain $x^*$ satisfying
  $x^*|_U = u$ and $(L_\mathcal{F}\, x^*)_v = 0$ for all $v \in \Omega$.
  Uniqueness holds whenever the relative cohomology
  $H^0(G, U; \mathcal{F}) = 0$.
\end{theorem}

The harmonic extension minimizes the total discrepancy
$\|\delta x\|^2$ subject to the boundary conditions imposed by $U$. In the
context of our neural network construction (\cref{sec:sheafyNN}), the input
vertex plays the role of a stubborn vertex, and the forward pass
output emerges as the harmonic extension of their fixed values.

\medskip

\Cref{thm:HG-heat,thm:HG-restricted} apply to sheaves with fixed, linear
restriction maps. Two generalizations are needed for the neural network
setting.

The first concerns \emph{evolving restriction maps}. When both vertex data
and restriction maps adapt simultaneously to reduce discrepancy, one obtains
a joint dynamics on the state and the sheaf structure. This framework was
introduced by Hansen and Ghrist~\cite{hansen2021opinion} and extended in ~\cite{bosca2026selective}, which shows that selective
constraints on both vertex data and restriction maps can be treated
uniformly as constrained diffusion on appropriate sheaves, following the
interpretation of such constraints as affine sheaf structures developed by
Gould~\cite{gould2025thesis}. We draw on this framework in
\cref{sec:extensions}, where evolving restriction maps correspond to weight
training.

The second concerns \emph{nonlinear restriction maps}. The definitions above
require each $\mathcal{F}_{v \trianglelefteqslant e}$ to be a linear map,
which restricts the class of operations that a sheaf can encode.
Gould~\cite{gould2025thesis} extends the convergence theory to sheaves
whose restriction maps are continuous piecewise affine (CPWA), using
Filippov theory and generalized gradient techniques. This extension opens
the door to encoding the components of feedforward ReLU networks within the
cellular sheaf framework, as we develop in \cref{sec:sheafyNN}.


\section{The Neural Sheaf Construction}\label{sec:sheafyNN}

We now embed a feedforward ReLU network into a cellular sheaf so that the
forward pass computation becomes a question of local-to-global consistency.
The construction assigns a vertex to each intermediate quantity in the forward
pass and encodes each computational step as a restriction map on an edge. The
input values are held fixed as boundary conditions, and the network output
emerges as the unique harmonic extension of these boundary data.

\subsection{From Networks to Sheaves}\label{subsec:networks-to-sheaves}

The forward pass~\eqref{eq:forward-pass} alternates affine transformations
with ReLU nonlinearities. To encode each step as a linear map between stalks,
we introduce extended representations that absorb biases into the weight
matrices.

For each layer $\ell \in \{1, \ldots, k{+}1\}$, define the
\emph{extended weight matrix}
\begin{equation}\label{eq:extended-weight}
  \overline{W}^{(\ell)}
    \coloneqq \begin{pmatrix} W^{(\ell)} & B^{(\ell)} \end{pmatrix}
    \in \R^{n_\ell \times (n_{\ell-1} + n_\ell)},
  \qquad B^{(\ell)} = \mathrm{diag}(b^{(\ell)}).
\end{equation}
For each hidden layer $\ell \in \{1, \ldots, k\}$, define the
\emph{extended activation vector}
\begin{equation}\label{eq:extended-activation}
  \overline{\mathbf{a}}^{(\ell)}
    \coloneqq \begin{pmatrix} \mathbf{a}^{(\ell)} \\ \mathbf{1}_{n_{\ell+1}} \end{pmatrix}
    \in \R^{n_\ell + n_{\ell+1}},
\end{equation}
and set
$\overline{\mathbf{a}}^{(0)} \coloneqq (\mathbf{x}^T,\,
\mathbf{1}_{n_1}^T)^T$ for the input. The ones block in each extended vector
carries the bias information for the next layer: combined with the diagonal
matrix $B^{(\ell)}$ in~$\overline{W}^{(\ell)}$, it produces the bias term
$b^{(\ell)}$ through ordinary matrix multiplication. This ones block is held
fixed throughout the dynamics, making it a \emph{stubborn direction} in the
sense of~\cite{bosca2026selective}; see also the affine sheaf interpretation
in~\cite{gould2025thesis}.

The ReLU activation at layer~$\ell$ is encoded by the diagonal
\emph{ReLU matrix} $R^{\mathbf{z}^{(\ell)}} \in \R^{n_\ell \times n_\ell}$
with entries
\begin{equation}\label{eq:relu-matrix}
  R^{\mathbf{z}^{(\ell)}}_{jj}
    = \begin{cases}
        1 & \text{if } \mathbf{z}^{(\ell)}_j \geq 0,\\
        0 & \text{if } \mathbf{z}^{(\ell)}_j < 0.
      \end{cases}
\end{equation}
Since $R^{\mathbf{z}^{(\ell)}}$ is diagonal with entries in $\{0,1\}$, it is
an orthogonal projection:
$(R^{\mathbf{z}^{(\ell)}})^2 = R^{\mathbf{z}^{(\ell)}}$ and
$(R^{\mathbf{z}^{(\ell)}})^T = R^{\mathbf{z}^{(\ell)}}$. Finally, let
$P_{n_\ell} = \begin{pmatrix} I_{n_\ell} & 0 \end{pmatrix}$ denote the
projection onto the first $n_\ell$ coordinates. With these definitions, the
forward pass equations~\eqref{eq:forward-pass} become
$\mathbf{z}^{(\ell)} = \overline{W}^{(\ell)}\overline{\mathbf{a}}^{(\ell-1)}$
and
$\mathbf{a}^{(\ell)} = R^{\mathbf{z}^{(\ell)}} \mathbf{z}^{(\ell)}$:
each step is now a matrix multiplication.

We construct a cellular sheaf $\mathcal{F}$ on a graph $G = (V, E)$ encoding
the full $k$-layer forward pass. The underlying graph is a path with one
vertex for each intermediate quantity:
\begin{equation}\label{eq:graph-path}
  v_x \,\text{ -- }\, v_{z^{(1)}} \,\text{ -- }\, v_{a^{(1)}}
    \,\text{ -- }\, v_{z^{(2)}} \,\text{ -- }\, \cdots
    \,\text{ -- }\, v_{a^{(k)}} \,\text{ -- }\, v_{z^{(k+1)}}
    \,\text{ -- }\, v_{y}.
\end{equation}

\medskip
\noindent\textbf{Stalks.}
The stalk dimensions match the quantities they carry:
\begin{equation}\label{eq:stalks}
\begin{aligned}
  \mathcal{F}(v_x) &= \R^{n_0 + n_1},
    &\qquad \mathcal{F}(v_{z^{(\ell)}}) &= \R^{n_\ell},
    &\qquad &1 \leq \ell \leq k{+}1,\\
  \mathcal{F}(v_{a^{(\ell)}}) &= \R^{n_\ell + n_{\ell+1}},
    &\qquad \mathcal{F}(v_{y}) &= \R^{n_{k+1}}.
    &\qquad &1 \leq \ell \leq k.
\end{aligned}
\end{equation}
The extended stalks at $v_x$ and $v_{a^{(\ell)}}$ have two blocks. The first
block carries neural data (the input $\mathbf{x}$ at $v_x$, or the
activation $a^{(\ell)}$ at $v_{a^{(\ell)}}$). The second block holds
$\mathbf{1}_{n_{\ell+1}}$, which is fixed and propagates bias to the next
layer.

\medskip
\noindent\textbf{Restriction maps.}
Edges connect consecutive vertices, and restriction maps encode the forward
pass operations. Three types of edges appear.

A \emph{weight edge} $e_{z^{(\ell)}}$ connects $v_{a^{(\ell - 1)}}$ to $v_{z^{(\ell)}}$
for each $\ell \in \{1, \ldots, k{+}1\}$ (identifying $v_{a^{(0)}} \equiv v_x$),
with edge stalk $\R^{n_\ell}$ and restriction maps
\begin{equation}\label{eq:weight-edge}
  \mathcal{F}_{v_{a^{(\ell - 1)}} \trianglelefteqslant e_{z^{(\ell)}}}
    = \overline{W}^{(\ell)}, \qquad
  \mathcal{F}_{v_{z^{(\ell)}} \trianglelefteqslant e_{z^{(\ell)}}}
    = I_{n_\ell}.
\end{equation}
An \emph{activation edge} $e_{a^{(\ell)}}$ connects $v_{z^{(\ell)}}$ to $v_{a^{(\ell)}}$
for each $\ell \in \{1, \ldots, k\}$, with edge stalk $\R^{n_\ell}$ and
restriction maps
\begin{equation}\label{eq:activation-edge}
  \mathcal{F}_{v_{z^{(\ell)}} \trianglelefteqslant e_{a^{(\ell)}}}
    = R^{\mathbf{z}^{(\ell)}}, \qquad
  \mathcal{F}_{v_{a^{(\ell)}} \trianglelefteqslant e_{a^{(\ell)}}}
    = P_{n_\ell}.
\end{equation}
An \emph{output edge} $e_{y}$ connects $v_{z^{(k+1)}}$ to $v_{y}$,
with edge stalk $\R^{n_{k+1}}$ and restriction maps
\begin{equation}\label{eq:output-edge}
  \mathcal{F}_{v_{z^{(k+1)}} \trianglelefteqslant e_{y}} = \phi, \qquad
  \mathcal{F}_{v_{y} \trianglelefteqslant e_{y}} = I_{n_{k+1}}.
\end{equation}
When $\phi$ is the identity, this is a standard linear restriction map. When
$\phi$ is nonlinear (sigmoid, softmax), it falls outside the linear sheaf
framework; we address convergence in this case in
\cref{sec:convergence}.

\Cref{fig:k-layers} illustrates the correspondence for the general $k$-hidden layer
construction, and \cref{fig:nn-sheaf-correspondence} displays a single
hidden layer network.

\medskip

The input vertex $v_x$ is pinned to the extended input
$\overline{\mathbf{x}} = (\mathbf{x}^T,\, \mathbf{1}_{n_1}^T)^T$.  At each
post-activation vertex $v_{a^{(\ell)}}$, the second block of coordinates is held
fixed at $\mathbf{1}_{n_{\ell+1}}$ while the first block (the activation
$a^{(\ell)}$) is free to evolve.  All other vertices are fully
dynamic.  These fixed coordinates are the boundary data of the sheaf; the
remaining coordinates evolve under sheaf diffusion
(\cref{subsec:harmonic-extension}).

\begin{figure}[ht]
\centering
\begin{tikzpicture}[
    vertex/.style={circle, fill=white, draw=black, thick, minimum size=5pt, inner sep=0pt},
    edge stalk h/.style={rectangle, rounded corners=1pt, fill=white, draw=black, minimum width=16pt, minimum height=4pt, inner sep=0pt},
    edge stalk v/.style={rectangle, rounded corners=1pt, fill=white, draw=black, minimum width=4pt, minimum height=16pt, inner sep=0pt},
    arrow/.style={-{Stealth[length=4pt]}, thick},
    every node/.style={font=\small},
]

\node[vertex, draw=DarkRed] (vx) at (0, 2) {};

\node[vertex, draw=DarkGreen] (vz1) at (0, 0) {};

\node[vertex, draw=DarkOrange1] (va1) at (3, 0) {};

\node[vertex, draw=DarkGreen] (vz2) at (6, 0) {};

\node[vertex, draw=DarkOrange1] (vak) at (9.5, 0) {};

\node[vertex, draw=DarkGreen] (vzk1) at (12.5, 0) {};

\node[vertex, draw=DarkGreen] (vy) at (12.5, 2) {};

\node[above left] at (vx) {$\overline{\mathbf{x}}$};
\node[below left] at (vz1) {$z^{\scriptscriptstyle(1)}$};
\node[below] at (va1) {$\overline{a}^{\scriptscriptstyle(1)}$};
\node[below] at (vz2) {$z^{\scriptscriptstyle(2)}$};
\node[below] at (vak) {$\overline{a}^{\scriptscriptstyle(k)}$};
\node[below right] at (vzk1) {$z^{\scriptscriptstyle(k+1)}$};
\node[above right] at (vy) {$\hat{y}$};

\draw[thick] (vx) -- (vz1);

\draw[thick] (vz1) -- (va1);
\draw[thick] (va1) -- (vz2);
\draw[thick] (vz2) -- (7, 0);
\draw[thick] (8.5, 0) -- (vak);
\draw[thick] (vak) -- (vzk1);

\draw[thick] (vzk1) -- (vy);

\node at (7.75, 0) {$\cdots$};

\node[edge stalk v] (e1) at (0, 1) {};
\node[edge stalk h] (e2) at (1.5, 0) {};
\node[edge stalk h] (e3) at (4.5, 0) {};
\node[edge stalk h] (e4) at (11, 0) {};
\node[edge stalk v] (e5) at (12.5, 1) {};


\draw[arrow] (-0.3, 1.85) -- (-0.3, 1.15);
\draw[arrow] (-0.3, 0.15) -- (-0.3, 0.85);
\node[left, font=\scriptsize] at (-0.35, 1.5) {$\overline{W}^{\scriptscriptstyle(1)}$};
\node[left, font=\scriptsize] at (-0.35, 0.5) {$I_{n_1}$};

\draw[arrow] (0.35, -0.3) -- (1.3, -0.3);
\draw[arrow] (2.65, -0.3) -- (1.7, -0.3);
\node[below, font=\scriptsize] at (0.75, -0.35) {$R^{z^{(1)}}$};
\node[below, font=\scriptsize] at (2.25, -0.35) {$P_{n_1}$};

\draw[arrow] (3.35, -0.3) -- (4.3, -0.3);
\draw[arrow] (5.65, -0.3) -- (4.7, -0.3);
\node[below, font=\scriptsize] at (3.75, -0.35) {$\overline{W}^{\scriptscriptstyle(2)}$};
\node[below, font=\scriptsize] at (5.25, -0.35) {$I_{n_2}$};

\draw[arrow] (6.35, -0.3) -- (7.3, -0.3);
\node[below, font=\scriptsize] at (6.75, -0.35) {$R^{z^{(2)}}$};

\draw[arrow] (9.15, -0.3) -- (8.2, -0.3);
\node[below, font=\scriptsize] at (8.75, -0.35) {$P_{n_k}$};

\draw[arrow] (9.85, -0.3) -- (10.8, -0.3);
\draw[arrow] (12.15, -0.3) -- (11.2, -0.3);
\node[below, font=\scriptsize] at (10.25, -0.35) {$\overline{W}^{\scriptscriptstyle(k+1)}$};
\node[below, font=\scriptsize] at (11.75, -0.35) {$I_{n_{k+1}}$};

\draw[arrow] (12.8, 1.85) -- (12.8, 1.15);
\draw[arrow] (12.8, 0.15) -- (12.8, 0.85);
\node[right, font=\scriptsize] at (12.85, 1.5) {$I_{n_{k+1}}$};
\node[right, font=\scriptsize] at (12.85, 0.5) {$\phi$};

\end{tikzpicture}
\caption{Neural sheaf encoding a $k$-hidden layer ReLU network. Red nodes are stubborn (boundary conditions); green and yellow nodes are dynamic, with yellow nodes having a fixed component. Restriction maps $\overline{W}^{(\ell)}$ encode weights and biases, $R^{z^{(\ell)}}$ encodes ReLU activation, $P_{n_\ell}$ projects onto the first $n_\ell$ coordinates, and $\phi$ is the final activation function.}
\label{fig:k-layers}
\end{figure}
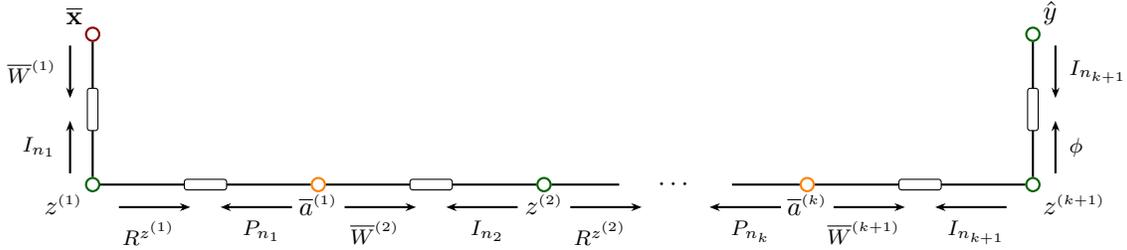

The restriction maps are designed so that local consistency on each edge
reproduces one step of the forward pass. On a weight edge, vanishing
discrepancy requires
$\overline{W}^{(\ell)}\overline{a}^{(\ell-1)} =
z^{(\ell)}$; on an activation edge, it requires
$R^{z^{(\ell)}} z^{(\ell)} = a^{(\ell)}$. A
$0$-cochain that is locally consistent on every edge therefore encodes the
complete forward pass. In \cref{subsec:harmonic-extension} we show that sheaf
diffusion with the fixed boundary data converges to precisely this consistent
state.

\subsection{The Forward Pass as Harmonic Extension}
\label{subsec:harmonic-extension}

The sheaf $\mathcal{F}$ from \cref{subsec:networks-to-sheaves} comes equipped
with boundary data: the input stalk at $v_x$ is fixed to
$\overline{\mathbf{a}}^{(0)}$, and the ones block in each extended activation
$\overline{a}^{(\ell)}$ is held at $\mathbf{1}_{n_{\ell+1}}$. The
remaining coordinates are free. We collect the fixed coordinates into a
vector~$u$ and the free coordinates into~$\omega$, so that the $0$-cochain
decomposes as $\overline{\omega} = u \oplus \omega$. This is analogous to the
$U$-restricted dynamics of \cref{thm:HG-restricted}, though here the
fixed data consists of individual coordinates within stalks rather than
entire vertex stalks. ~\cite{bosca2026selective} formalizes
this coordinate-level decomposition through constrained opinion dynamics on
sheaves with stubborn directions; for our purposes, the informal
description suffices.

Applying the sheaf heat equation to the free coordinates gives the restricted
dynamics
\begin{equation}\label{eq:neural-restricted}
  \frac{d\omega}{dt}
    = -\alpha\bigl(
        L_{\mathcal{F}_t}[\Omega,\Omega]\,\omega
        + L_{\mathcal{F}_t}[\Omega,U]\,u
      \bigr),
\end{equation}
where $L_{\mathcal{F}_t}[\Omega,\Omega]$ and $L_{\mathcal{F}_t}[\Omega,U]$ denote
the blocks of the sheaf Laplacian restricted to the free and fixed coordinate
subspaces, respectively. Each free coordinate is driven by discrepancies on
its incident edges: the pre-activation $z^{(\ell)}$ is pulled toward
agreement with $\overline{W}^{(\ell)}\overline{a}^{(\ell-1)}$ via
its weight edge, while the activation $a^{(\ell)}$ is pulled toward
$R^{z^{(\ell)}} z^{(\ell)}$ via its activation edge and toward
$z^{(\ell+1)}$ via the subsequent weight edge.

\begin{remark}[State-dependent restriction maps]\label{rem:state-dependent}
  The diagonal ReLU matrix $R^{z^{(\ell)}}$ depends on the sign pattern of
  the current pre-activation $z^{(\ell)}$. Since $z^{(\ell)}$ is part of
  the evolving $0$-cochain, the restriction maps on the activation edges
  change with the dynamics, creating a \emph{state-dependent sheaf}
  $\mathcal{F}_t$: the Laplacian $L_{\mathcal{F}_t}$ and the
  dynamics~\eqref{eq:neural-restricted} are piecewise constant, switching
  when a pre-activation coordinate crosses zero. For each fixed activation
  pattern the sheaf is an ordinary cellular sheaf and the classical theory
  of \cref{thm:HG-heat,thm:HG-restricted} applies. Convergence
  \emph{despite} the switching is the subject of \cref{sec:convergence}.
\end{remark}

%

The existence and uniqueness of a harmonic extension, and its
identification with the forward pass, rest on a single structural
property of the coboundary operator.

For any frozen activation pattern, write $\mathcal{F}_R$ for the
cellular sheaf whose activation-edge restriction maps use the
corresponding diagonal ReLU matrices $R^{(\ell)}$. Decompose each
$0$-cochain as $\overline{\omega} = u \oplus \omega$, where $u$ collects the fixed
boundary data (input, bias directions) and $\omega$ the free
coordinates.  The coboundary decomposes accordingly:
$\delta \overline{\omega} = \delta_\Omega\, \omega + \delta_U\, u$, where
$\delta_\Omega \colon C^0_\Omega \to C^1$ is the restriction of the
coboundary to the free coordinate subspace.

\begin{lemma}[Unitriangular factorization]\label{lemma:det1}
  For any activation pattern and identity output activation
  $\phi = \id$, the restricted coboundary $\delta_\Omega$ is a square
  matrix with $\det \delta_\Omega = 1$.  Consequently the restricted
  Laplacian satisfies
  \[
    L_{\mathcal{F}_R}[\Omega,\Omega]
      = \delta_\Omega^T\, \delta_\Omega,
    \qquad
    \det L_{\mathcal{F}_R}[\Omega,\Omega] = 1,
  \]
  and in particular is positive definite.
\end{lemma}

\begin{proof}
  Order the free coordinates along the path as
  \[
    \omega
      = \bigl(
          z^{(1)},\; a^{(1)},\;
          z^{(2)},\; a^{(2)},\;
          \ldots,\;
          a^{(k)},\; z^{(k+1)},\; \hat{y}
        \bigr),
  \]
  and the edge stalks in the same path order:
  $e_{z^{(1)}},\, e_{a^{(1)}},\, e_{z^{(2)}},\, \ldots,\,
  e_{z^{(k+1)}},\, e_y$.  Each free-coordinate block has the same
  dimension as the edge stalk it is paired with -- $\R^{n_\ell}$ for
  both $z^{(\ell)}$ and $e_{z^{(\ell)}}$, and likewise for
  activation edges and the output edge -- so $\delta_\Omega$ is a square
  matrix of size $\bigl(2\!\sum_{\ell=1}^{k} n_\ell + 2\,n_{k+1}\bigr)$.

  Each edge contributes one block row to $\delta_\Omega$:  the
  downstream endpoint enters with an identity block on the diagonal,
  and the upstream endpoint enters with
  $-R^{(\ell)}$, $-W^{(\ell)}$, or $-I$ on the first block
  sub-diagonal.  Concretely, for a single hidden layer
  ($k = 1$) with free coordinates
  $\omega = (z^{(1)},\, a^{(1)},\,
  z^{(2)},\, \hat{y})$:
  \begin{equation}\label{eq:delta-omega-k1}
    \delta_\Omega
      = \begin{pmatrix}
          I_{n_1} &         &         &         \\
          -R^{(1)} & I_{n_1} &         &         \\
                  & -W^{(2)} & I_{n_2} &         \\
                  &         & -I_{n_2} & I_{n_2}
        \end{pmatrix}.
  \end{equation}
  For general $k$, the pattern extends along the path: sub-diagonal
  blocks alternate between $-R^{(\ell)}$ (activation edges) and
  $-W^{(\ell+1)}$ (weight edges), terminating with $-I_{n_{k+1}}$
  (output edge).  In every case $\delta_\Omega$ is block
  lower-triangular with identity blocks on the diagonal, so
  $\det\delta_\Omega = 1$.  Therefore
  \[
    \det L_{\mathcal{F}_R}[\Omega,\Omega]
      = \bigl(\det\delta_\Omega\bigr)^2
      = 1.
  \]
  Invertibility of $\delta_\Omega$ implies that
  $\delta_\Omega^T\delta_\Omega$ is positive definite.
\end{proof}

\begin{remark}\label{rem:reduced-system}
  For identity output, the output edge forces
  $\hat{y} = z^{(k+1)}$ at equilibrium.
  Eliminating $\hat{y}$ by Schur complement gives the
  \emph{reduced} restricted Laplacian on the free coordinates
  $(z^{(1)}, a^{(1)}, \ldots,
  a^{(k)}, z^{(k+1)})$.  Since the
  $\hat{y}$-diagonal block is $I_{n_{k+1}}$, the Schur
  complement preserves the unit determinant.  The reduced system
  has the block-tridiagonal form derived in
  \cref{ap:block-structure} and is used for the dynamics throughout
  \cref{sec:convergence,sec:extensions}.
\end{remark}

The unit determinant forces $L_{\mathcal{F}_R}[\Omega,\Omega]$ to be
positive definite for every activation pattern.  In the language of
\cref{subsec:sheaves}, for a fixed $t$ the relative cohomology
$H^0(G, U; \mathcal{F}_R) = 0$ vanishes, so the harmonic extension of
the boundary data exists and is unique.  Since there are finitely many
activation patterns, the smallest eigenvalue across all patterns
provides a uniform lower bound
$\lambda^* = \alpha\,\min_{R}\lambda_{\min}(L_{\mathcal{F}_R}[\Omega,\Omega])
> 0$ on the convergence rate.

Since the restricted Laplacian is positive definite for every activation
pattern, each pattern yields a unique harmonic extension of the boundary
data.  The following result identifies all of these harmonic extensions
at once and shows that forward substitution in the unitriangular system
is exactly the forward pass.

\begin{proposition}[Forward pass as harmonic extension]
\label{prop:forward-pass}
  For any frozen activation pattern $R$, the unique harmonic extension
  of the boundary data satisfies
  $z^{(\ell)*} = \overline{W}^{(\ell)}\overline{a}^{(\ell-1)*}$
  and
  $a^{(\ell)*} = R^{(\ell)}\,z^{(\ell)*}$
  for all layers, with $\hat{y}^* = z^{(k+1)*}$.
  In particular, when $\overline{\mathbf{a}}^{(0)}$ is fixed to
  $\overline{\mathbf{x}}$, only the forward-pass activation
  pattern~\eqref{eq:forward-pass} produces a harmonic extension
  that is also a global section.
\end{proposition}

\begin{proof}
  The harmonic extension minimizes the total discrepancy
  $\frac{1}{2}\|\delta\, \overline{\omega}\|^2$ subject to the boundary data, so it
  satisfies $\delta_\Omega^T(\delta_\Omega\,\omega^* +
  \delta_U\,u) = 0$.  Since $\delta_\Omega$ is invertible
  (\cref{lemma:det1}), this simplifies to
  \[
    \delta_\Omega\,\omega^* = -\delta_U\,u,
  \]
  which is the zero-discrepancy condition: the total coboundary
  $\delta\,\overline{\omega}^* = 0$ vanishes on every edge.

  Because $\delta_\Omega$ is unitriangular, this system is solved by
  forward substitution along the path.  The first block row (weight
  edge~$1$) gives
  $z^{(1)*} = \overline{W}^{(1)}\overline{\mathbf{x}}$.  The
  second row (activation edge~$1$) gives
  $a^{(1)*} = R^{(1)}z^{(1)*}$.  Continuing
  alternately through weight and activation edges:
  $z^{(\ell)*} =
  \overline{W}^{(\ell)}\overline{a}^{(\ell-1)*}$ and
  $a^{(\ell)*} = R^{(\ell)}z^{(\ell)*}$ for
  $1 \leq \ell \leq k$, and finally
  $z^{(k+1)*} =
  \overline{W}^{(k+1)}\overline{a}^{(k)*}$ and
  $\hat{y}^* = z^{(k+1)*}$.  These are exactly the
  forward-pass equations~\eqref{eq:forward-pass} with activation
  pattern $R$.

  Since $\delta\,\overline{\omega}^* = 0$, the extension is a global section of the
  frozen sheaf $\mathcal{F}_R$.  It is a global section of the
  \emph{actual} state-dependent sheaf $\mathcal{F}_t$ only when $R$
  is self-consistent: the sign pattern of the computed
  $z^{(\ell)*}$ must agree with $R^{(\ell)}$.  Fixing
  $\overline{\mathbf{a}}^{(0)} = \overline{\mathbf{x}}$ determines
  $z^{(1)*}$, hence $R^{(1)}$, hence $a^{(1)*}$,
  and so on layer by layer: the self-consistent pattern is unique,
  provided no pre-activation coordinate vanishes (a generic condition
  in the weights; see \cref{rem:zero-preactivation}).
\end{proof}

\begin{remark}[Zero pre-activations]\label{rem:zero-preactivation}
  When $\mathbf{z}^{(\ell)}_j = 0$ at the forward-pass solution,
  both $R^{(\ell)}_{jj} = 0$ and $R^{(\ell)}_{jj} = 1$ produce
  $\mathbf{a}^{(\ell)}_j = 0$, so the harmonic extension is the same
  either way.  Self-consistency is unambiguous; the non-uniqueness is
  only in the labeling of the activation pattern, not in the solution
  itself.  This is a codimension-one condition in weight space and is
  absent for generic weights, but it can arise for trained networks
  with dead neurons.  Our convention $R^{(\ell)}_{jj} = 1$ for
  $\mathbf{z}^{(\ell)}_j \geq 0$ resolves the labeling; the forward pass and all
  convergence results are unaffected.
\end{remark}

\begin{figure}[ht]
\centering

\begin{subfigure}[b]{0.48\textwidth}
\centering
\begin{tikzpicture}[
    neuron/.style={circle, draw=black, minimum size=0.6cm, inner sep=0pt, font=\scriptsize},
    arrow/.style={->, >=Stealth, thick},
    weight/.style={font=\scriptsize, inner sep=1pt},
]

\node[neuron, draw=DarkRed] (x1) at (0, 1.5) {$\mathbf{x}_1$};
\node[neuron, draw=DarkRed] (x2) at (0, 0) {$\mathbf{x}_2$};

\node[neuron, draw=DarkGreen] (a1) at (3, 2.25) {$\mathbf{a}_1^{\scriptscriptstyle(1)}$};
\node[neuron, draw=DarkGreen] (a2) at (3, 0.75) {$\mathbf{a}_2^{\scriptscriptstyle(1)}$};
\node[neuron, draw=DarkGreen] (a3) at (3, -0.75) {$\mathbf{a}_3^{\scriptscriptstyle(1)}$};

\node[neuron, draw=DarkGreen] (y1) at (6, 0.75) {$\hat{\mathbf{y}}$};

\draw[arrow] (x1) -- (a1);
\draw[arrow] (x1) -- (a2);
\draw[arrow] (x1) -- (a3);
\draw[arrow] (x2) -- (a1);
\draw[arrow] (x2) -- (a2);
\draw[arrow] (x2) -- (a3);

\draw[arrow] (a1) -- (y1);
\draw[arrow] (a2) -- (y1);
\draw[arrow] (a3) -- (y1);

\node[weight] at ($(x1)!0.5!(a1) + (0, 0.3)$) {$W^{\scriptscriptstyle(1)}_{11}$};
\node[weight] at ($(x2)!0.5!(a3) + (0, -0.3)$) {$W^{\scriptscriptstyle(1)}_{32}$};

\node[weight] at ($(a1)!0.5!(y1) + (0, 0.3)$) {$W^{\scriptscriptstyle(2)}_{11}$};
\node[weight] at ($(a3)!0.5!(y1) + (0, -0.3)$) {$W^{\scriptscriptstyle(2)}_{13}$};

\node[align=center, font=\scriptsize ] at (0, -1.8) {$\mathbf{x}$};
\node[align=center, font=\scriptsize ] at (2.7, -1.8) {$\mathbf{a}^{\scriptscriptstyle(1)} = \relu(W^{\scriptscriptstyle(1)}\mathbf{x} + b^{\scriptscriptstyle(1)})$};
\node[align=center, font=\scriptsize ] at (6.3, -1.8) {$\hat{\mathbf{y}} = \phi(W^{\scriptscriptstyle(2)}\mathbf{a}^{\scriptscriptstyle(1)} + b^{\scriptscriptstyle(2)})$};

\node[above=0.3cm of x1] {\footnotesize Input};
\node[above=0.3cm of a1] {\footnotesize Hidden};
\node[above=0.3cm of y1] {\footnotesize Output};

\end{tikzpicture}
\caption{Feedforward neural network}
\label{fig:nn-panel}
\end{subfigure}
\hfill
\begin{subfigure}[b]{0.48\textwidth}
\centering

\begin{tikzpicture}[
    vertex/.style={circle, fill=white, draw=black, thick, minimum size=5pt, inner sep=0pt},
    edge stalk h/.style={rectangle, rounded corners=1pt, fill=white, draw=black, minimum width=16pt, minimum height=4pt, inner sep=0pt},
    edge stalk v/.style={rectangle, rounded corners=1pt, fill=white, draw=black, minimum width=4pt, minimum height=16pt, inner sep=0pt},
    arrow/.style={-{Stealth[length=4pt]}, thick},
    every node/.style={font=\small},
]

\node[vertex, draw=DarkRed] (vx) at (0, 2) {};

\node[vertex, draw=DarkGreen] (vz1) at (0, 0) {};

\node[vertex, draw=DarkOrange1] (va1) at (2.5, 0) {};

\node[vertex, draw=DarkGreen] (vz2) at (5, 0) {};

\node[vertex, draw=DarkGreen] (vy) at (5, 2) {};

\node[above left] at (vx) {$\overline{\mathbf{x}}$};
\node[below left] at (vz1) {$z^{\scriptscriptstyle(1)}$};
\node[below] at (va1) {$\overline{a}^{\scriptscriptstyle(1)}$};
\node[below right] at (vz2) {$z^{\scriptscriptstyle(2)}$};
\node[above right] at (vy) {$\hat{y}$};
\node[above] at (2.5, 3) {$\dfrac{d\omega}{dt} = -\alpha(L_{\mathcal{F}_t}[\Omega,\Omega]\omega + L_{\mathcal{F}_t}[\Omega,U]u)$};

\draw[thick] (vx) -- (vz1);

\draw[thick] (vz1) -- (va1);

\draw[thick] (va1) -- (vz2);

\draw[thick] (vz2) -- (vy);

\node[edge stalk v] (e1) at (0, 1) {};
\node[edge stalk h] (e2) at (1.25, 0) {};
\node[edge stalk h] (e3) at (3.75, 0) {};
\node[edge stalk v] (e4) at (5, 1) {};


\draw[arrow] (-0.3, 1.85) -- (-0.3, 1.15);
\draw[arrow] (-0.3, 0.15) -- (-0.3, 0.85);
\node[left, font=\scriptsize] at (-0.35, 1.5) {$\overline{W}^{\scriptscriptstyle(1)}$};
\node[left, font=\scriptsize] at (-0.35, 0.5) {$I_3$};

\draw[arrow] (0.15, -0.3) -- (1.1, -0.3);
\draw[arrow] (2.2, -0.3) -- (1.4, -0.3);
\node[below, font=\scriptsize] at (0.55, -0.35) {$R^{z^{(1)}}$};
\node[below, font=\scriptsize] at (1.85, -0.35) {$P_3$};

\draw[arrow] (2.8, -0.3) -- (3.6, -0.3);
\draw[arrow] (4.85, -0.3) -- (3.9, -0.3);
\node[below, font=\scriptsize] at (3.15, -0.35) {$\overline{W}^{\scriptscriptstyle(2)}$};
\node[below, font=\scriptsize] at (4.45, -0.35) {$I_2$};

\draw[arrow] (5.3, 1.85) -- (5.3, 1.15);
\draw[arrow] (5.3, 0.15) -- (5.3, 0.85);
\node[right, font=\scriptsize] at (5.35, 1.5) {$I_1$};
\node[right, font=\scriptsize] at (5.35, 0.5) {$\phi$};

\end{tikzpicture}
\caption{Cellular sheaf structure}
\label{fig:sheaf-panel}
\end{subfigure}

\caption{The neural network--sheaf correspondence. Red nodes represent fixed data (boundary conditions); green and yellow nodes are computed, with yellow nodes having a fixed component. \textbf{(a)} The feedforward network computes postactivations $\mathbf{a}^{(1)} = \relu(W^{(1)}\mathbf{x} + b^{(1)})$ and output $\hat{\mathbf{y}} = \phi(W^{(2)}\mathbf{a}^{(1)} + b^{(2)})$ via the forward pass. \textbf{(b)} In the sheaf formulation, only $\overline{\mathbf{x}}$ (input extended with ones) is pinned. Restriction maps encode weights and biases ($\overline{W}^{(1)}, \overline{W}^{(2)}$), ReLU activations ($R^{z^{(1)}}$), and final activation ($\phi$). The heat equation drives any initialization to the correct forward pass output.}
\label{fig:nn-sheaf-correspondence}
\end{figure}
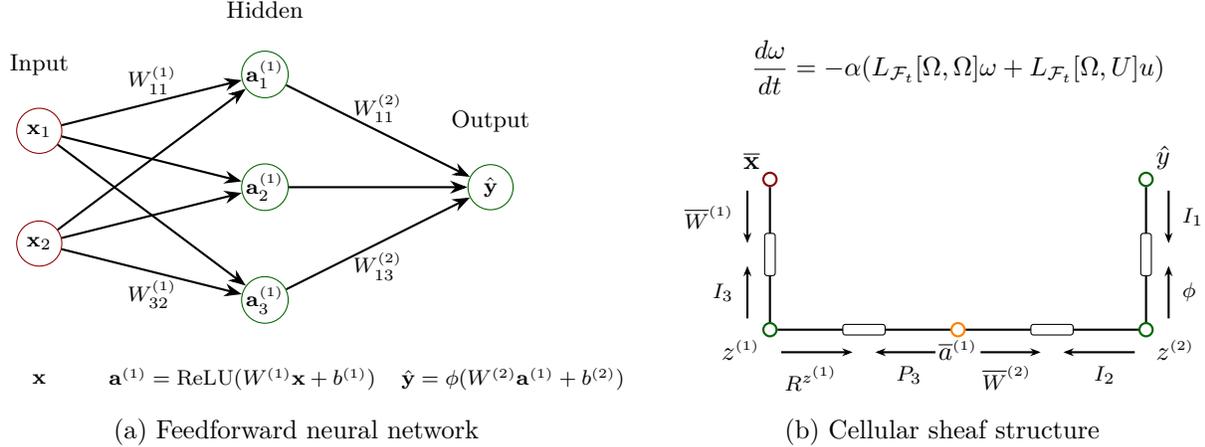

Thus, a standard feedforward ReLU network can be embedded in the space of global sections of a cellular sheaf in this way. The sheaf diffusion
dynamics~\eqref{eq:neural-restricted} provide an alternative approach to compute this
output. Although more convoluted, we argue that this formulation might offer benefits in terms of interpretability and privacy, stemming from the distributed nature of the algorithm. In \cref{sec:convergence} we show that the
dynamics converge to $\omega^*$ exponentially fast despite this switching.


\section{Convergence Analysis}\label{sec:convergence}

\Cref{prop:forward-pass} identifies the forward pass output as the unique
harmonic extension of the boundary data for each fixed activation pattern, but
the activation pattern itself evolves with the dynamics. This section
establishes that the dynamics nonetheless converge to the correct output. We
treat identity output activation first
(\cref{subsec:relu-convergence}), then extend to nonlinear final activations
in \cref{subsec:final-activation}.

\subsection{Convergence for ReLU Networks}
\label{subsec:relu-convergence}

The restricted dynamics~\eqref{eq:neural-restricted} form a continuous
piecewise affine (CPWA) system~\cite{gould2025thesis}: the state space is
partitioned into finitely many convex polyhedral regions
$\{\mathcal{R}_i\}_{i=1}^N$, one per ReLU activation
pattern~\cite{liu2023relu}, and within each region the vector field is affine.
Write $\mathcal{F}_i$ for the sheaf whose restriction maps correspond to the
activation pattern of region~$\mathcal{R}_i$. At region boundaries where some
$z^{(\ell)}_j = 0$, the restriction maps $R^{z^{(\ell)}}$ switch
discontinuously. Under our convention that $R^{z^{(\ell)}}_{jj} = 1$ for
$z^{(\ell)}_j \geq 0$, the dynamics select the velocity of maximal norm among Filippov
selections at switching surfaces; this is Gould's \emph{fast selection
rule}~\cite[Def.~7.4.17]{gould2025thesis}, which guarantees existence and
uniqueness of solutions and prevents spurious
sliding~\cite[Prop.~7.4.14, Rem.~7.4.20]{gould2025thesis}.

\begin{theorem}[Convergence to neural network output]
\label{thm:convergence}
  Consider a $k$-hidden-layer ReLU network with identity output activation
  $\phi = \mathrm{id}$, and assume the forward-pass equilibrium $\omega^*$
  lies in the interior of an activation region (a generic condition in the
  weights; see \cref{rem:zero-preactivation}). For any initial condition
  $\omega_0$, the solution of the restricted
  dynamics~\eqref{eq:neural-restricted} converges exponentially to the unique
  cochain $\omega^*$ encoding the standard forward pass output.
\end{theorem}

\begin{proof}
  Within each region~$\mathcal{R}_i$, the dynamics are gradient descent on the
  potential
  \begin{equation}\label{eq:lyapunov}
    V(\omega)
      = \tfrac{1}{2}\,\omega^T L_{\mathcal{F}_i}[\Omega,\Omega]\,\omega
        + \omega^T L_{\mathcal{F}_i}[\Omega,U]\,u,
  \end{equation}
  which equals the total discord
  $\frac{1}{2}\|\delta_i\overline{\omega}\|^2$ up to an additive constant
  depending only on the boundary data~$u$, where $\delta_i$ is the coboundary
  of $\mathcal{F}_i$ and $\overline{\omega} = u \oplus \omega$ is the full
  $0$-cochain. By \cref{lemma:det1}, $L_{\mathcal{F}_i}[\Omega,\Omega]$ is
  positive definite in every region, so $V$ has a unique minimizer in each
  region.

  The forward pass cochain $\omega^*$ from \cref{prop:forward-pass} is the
  unique \emph{self-consistent} critical point: the activation pattern it
  induces is the same one used to compute it. Any interior critical point of a
  region $\mathcal{R}_i$ must satisfy the forward pass equations with the ReLU
  pattern of~$\mathcal{R}_i$, and for a feedforward network the pinned input
  determines the activation pattern layer by layer, so only $\omega^*$
  qualifies. Boundary critical points, where $V > 0$ and the discord does not
  vanish, are not local minima of $V$; the fast selection rule ensures
  trajectories exit such points rather than remaining
  there~\cite[Lem.~7.4.23, Rem.~7.4.20]{gould2025thesis}.

  All fast solutions are
  bounded~\cite[Thm.~7.4.21]{gould2025thesis}, and bounded solutions of CPWA
  gradient systems converge to generalized critical
  points~\cite[Lem.~7.4.16]{gould2025thesis}. Since $\omega^*$ is the only
  such point, all trajectories converge to $\omega^*$.

  For the rate, the hypothesis guarantees that $\omega^*$ lies in the
  interior of some region~$\mathcal{R}_*$.  There exists $T$ such that the
  trajectory remains in~$\mathcal{R}_*$ for $t > T$, after which the
  dynamics reduce to the linear system
  $\frac{d}{dt}\omega = -\alpha\,L_{\mathcal{F}_*}[\Omega,\Omega](\omega -
  \omega^*)$, giving exponential convergence at rate
  $\alpha\,\lambda_{\min}(L_{\mathcal{F}_*}[\Omega,\Omega])$.  During the
  transient $t \leq T$, the Lyapunov function $V$ decreases at rate at
  least $\alpha\,\lambda_{\min}(L_{\mathcal{F}_i}[\Omega,\Omega])$ within
  each visited region~$\mathcal{R}_i$, so the overall convergence rate is
  bounded below by
  $\lambda^* = \alpha\,\min_i\lambda_{\min}(L_{\mathcal{F}_i}[\Omega,\Omega])
  > 0$, which is positive since there are finitely many regions and each
  restricted Laplacian is positive definite by \cref{lemma:det1}.
\end{proof}

Note that when the final activation is the identity, $z^{(k+1)} = \hat{y}$, so the output edge is redundant and in practice we can eliminate it. For all the other cases the output edge enforces additional nontrivial constraints, as we emphasize in the following section.

\subsection{Convergence with Final Activations}
\label{subsec:final-activation}

\Cref{thm:convergence} covers networks with identity output, sufficient for
regression. Classification requires nonlinear final activations such as
sigmoid, tanh, or softmax, and these are not piecewise linear: they cannot be
encoded as linear restriction maps in a cellular sheaf.
Gould~\cite{gould2025thesis} resolves this through the notion of a
\emph{local adjoint}, which extends the sheaf framework to nonlinear maps. To do this Gould considers that edge $e_{y}$ carries an edge
potential
\begin{equation}\label{eq:output-potential}
  U_{\mathrm{out}}(z^{(k+1)}, \hat{y})
    = \tfrac{1}{2}\|\phi(z^{(k+1)}) - \hat{y}\|^2
\end{equation}
Gradient descent on this potential drives
both vertices toward agreement through~$\phi$. When $\phi$ is a linear map this reduces to standard sheaf diffusion, but in the case of nonlinear $\phi$ the chain rule introduces the
Jacobian~$J_\phi(z^{(k+1)})$, which acts as a local
adjoint~\cite[Def.~7.2.9]{gould2025thesis}: the best linear approximation to
the adjoint of~$\phi$ at the current operating point. The resulting dynamics
on the output pair are
\begin{align}\label{eq:final-dynamics}
  \frac{d}{dt} z^{(k+1)}
    &= -\alpha\bigl[
        (z^{(k+1)} - \overline{W}^{(k+1)}\overline{a}^{(k)})
        + J_\phi(z^{(k+1)})^T\bigl(\phi(z^{(k+1)}) - \hat{y}\bigr)
      \bigr],\nonumber\\
  \frac{d}{dt} \hat{y}
    &= -\alpha\bigl(\hat{y} - \phi(z^{(k+1)})\bigr),
\end{align}
where the first term in the $z^{(k+1)}$ equation is the contribution from the
weight edge to the previous layer. The dynamics for all other coordinates
remain as in~\eqref{eq:neural-restricted}. At any fixed point,
$\hat{y} = \phi(z^{(k+1)})$ and the Jacobian term vanishes, so the hidden
layer equations reduce to the harmonic extension of
\cref{prop:forward-pass}. The full system therefore recovers the standard
forward pass including the final activation.

\begin{theorem}[Convergence with final activation]
\label{thm:final}
  Consider a $k$-hidden-layer ReLU network with $C^1$ final activation
  $\phi\colon \mathbb{R}^{n_{k+1}} \to \mathbb{R}^{n_{k+1}}$ that is either
  bounded (sigmoid, tanh, softmax) or has at most linear growth
  ($\|\phi(z)\| \leq C(1 + \|z\|)$ for some~$C > 0$).  Assume the
  forward-pass equilibrium $\omega^*$ lies in the interior of an activation
  region (a generic condition; see \cref{rem:zero-preactivation}).  For any
  initial condition $(\omega_0, \hat{y}_0)$, the solution of the restricted
  dynamics~\eqref{eq:neural-restricted} augmented
  by~\eqref{eq:final-dynamics} converges to the unique fixed point
  $(\omega^*, \phi(z^{(k+1)*}))$ encoding the standard neural network output.
\end{theorem}

\begin{proof}[Proof sketch]
  Append the output potential~\eqref{eq:output-potential} to the Lyapunov
  function~\eqref{eq:lyapunov} to obtain
  $E(\omega, \hat{y}) = V(\omega) +
  \frac{1}{2}\|\phi(z^{(k+1)}) - \hat{y}\|^2$.
  The positive-definite quadratic term $V(\omega)$ dominates at large
  $\|\omega\|$, and the growth hypothesis on~$\phi$ ensures the output
  potential does not destroy this dominance, so $E$ is coercive and
  trajectories are bounded. In each ReLU region interior,
  $\frac{d}{dt}E = -\alpha\,\|\nabla E\|^2 \leq 0$. Since
  $\operatorname{ReLU}(0) = 0$ regardless of convention, the one-sided
  energies agree on every switching surface, and $E$ decreases during
  Filippov sliding at rate $-\alpha\,\|\nabla_\tau E\|^2$.
  LaSalle's invariance principle~\cite{shevitz1994lyapunov} yields
  convergence to the set where $0 \in \partial_C E$. The feedforward
  cascade structure excludes Clarke critical points on switching surfaces
  for generic weights, so $\omega^*$ is the unique element of this set.
  The full proof appears in~\cref{ap:proofs}.
\end{proof}

The edge potential~\eqref{eq:output-potential} uses squared error as its
notion of discrepancy. Hansen and Ghrist~\cite{hansen2021opinion} observe that
replacing quadratic potentials with other convex functions yields nonlinear
sheaf Laplacians that retain the local update structure. In
\cref{sec:extensions}, we exploit this flexibility to pair $\phi$ with other loss functions in the training setting.


\section{Extensions}\label{sec:extensions}

The sheaf framework developed in
\cref{sec:sheafyNN,sec:convergence} identifies the forward pass as a global
section of a cellular sheaf, computed via heat equation dynamics. The
extensions below build on the distributed, bidirectional nature of this
formulation. Each has a feedforward counterpart: pinning parallels dropout,
restriction map evolution parallels backpropagation, and batch processing is
standard minibatch evaluation. The sheaf perspective restructures these
operations as local consistency problems on a graph, which changes the
geometry of information flow and opens the door to hybrid strategies that mix
feedforward and diffusion-based reasoning.

\subsection{Pinned Neurons}\label{subsec:pinning}

In the feedforward view, fixing a hidden neuron to a target value affects only
downstream layers. The sheaf formulation replaces this asymmetry with a global
consistency problem. When we constrain a neuron, the heat equation propagates
that constraint in both directions along the graph, and the entire network
seeks an equilibrium that accommodates the imposed value while minimizing
total discord. This is the neural analogue of the weighted reluctance
construction of Hansen and Ghrist~\cite{hansen2021opinion}, where a stubborn
``parent'' vertex exerts continuous influence on an agent's expressed opinion.

To pin a subset $J_\ell$ of neurons at layer~$\ell$ to target values
$\mathbf{p} \in \mathbb{R}^{|J_\ell|}$, we augment the sheaf with a new
stubborn vertex $v_{p}$ with stalk $\mathbb{R}^{|J_\ell|}$ holding
the target value~$\mathbf{p}$, connected to
$v_{a^{(\ell)}}$ by a penalty edge $e_{p}$ with
restriction maps
\begin{equation}\label{eq:pinning-maps}
  \mathcal{F}_{v_{a^{(\ell)}} \trianglelefteqslant e_{p}}
    = \sqrt{\gamma}\, P_{J_\ell},
  \qquad
  \mathcal{F}_{v_{p} \trianglelefteqslant e_{p}}
    = \sqrt{\gamma}\, I_{|J_\ell|},
\end{equation}
where $P_{J_\ell}$ projects onto the coordinates in $J_\ell$ and $\gamma > 0$
controls the pinning strength. The penalty edge contributes a term
$\gamma\, P_{J_\ell}^T\bigl(P_{J_\ell}\,\overline{a}^{(\ell)} -
\mathbf{p}\bigr)$ to the Laplacian action at
$v_{a^{(\ell)}}$, pulling the selected neurons
toward~$\mathbf{p}$. As $\gamma \to \infty$ this approaches a hard
constraint; for finite~$\gamma$ the network may deviate from the target if
other edges create sufficient tension. The input vertex $v_x$ in our
construction is itself an instance of this mechanism with
$\gamma = \infty$, and the ones blocks in the extended activations are
another.

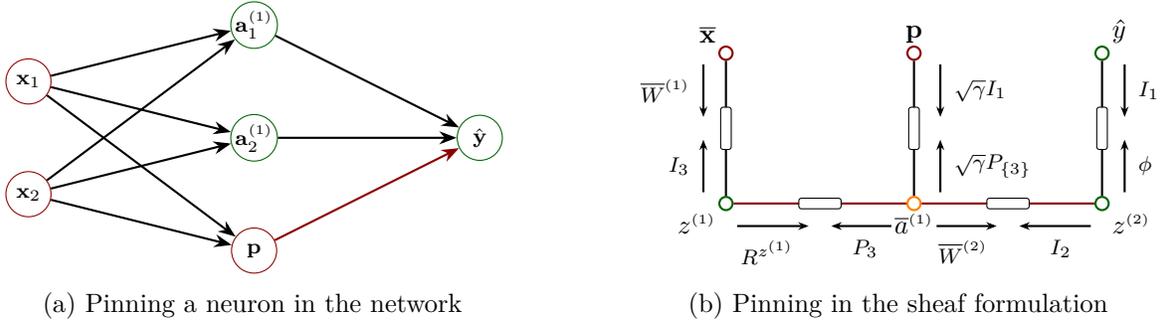
\begin{figure}[ht]
\centering

\begin{subfigure}[b]{0.48\textwidth}
\centering
\begin{tikzpicture}[
    neuron/.style={circle, draw=black, minimum size=0.6cm, inner sep=0pt, font=\scriptsize},
    arrow/.style={->, >=Stealth, thick},
]

\node[neuron, draw=DarkRed] (x1) at (0, 1.5) {$\mathbf{x}_1$};
\node[neuron, draw=DarkRed] (x2) at (0, 0) {$\mathbf{x}_2$};

\node[neuron, draw=DarkGreen] (a1) at (3, 2.25) {$\mathbf{a}_1^{\scriptscriptstyle(1)}$};
\node[neuron, draw=DarkGreen] (a2) at (3, 0.75) {$\mathbf{a}_2^{\scriptscriptstyle(1)}$};
\node[neuron, draw=DarkRed] (a3) at (3, -0.75) {$\mathbf{p}$};

\node[neuron, draw=DarkGreen] (y1) at (6, 0.75) {$\hat{\mathbf{y}}$};

\draw[arrow] (x1) -- (a1);
\draw[arrow] (x1) -- (a2);
\draw[arrow] (x1) -- (a3);
\draw[arrow] (x2) -- (a1);
\draw[arrow] (x2) -- (a2);
\draw[arrow] (x2) -- (a3);

\draw[arrow] (a1) -- (y1);
\draw[arrow] (a2) -- (y1);

\draw[arrow, DarkRed] (a3) -- (y1);

\end{tikzpicture}
\caption{Pinning a neuron in the network}
\label{fig:pin-nn}
\end{subfigure}
\hfill
\begin{subfigure}[b]{0.48\textwidth}
\centering
\begin{tikzpicture}[
    vertex/.style={circle, fill=white, draw=black, thick, minimum size=5pt, inner sep=0pt},
    edge stalk h/.style={rectangle, rounded corners=1pt, fill=white, draw=black, minimum width=16pt, minimum height=4pt, inner sep=0pt},
    edge stalk v/.style={rectangle, rounded corners=1pt, fill=white, draw=black, minimum width=4pt, minimum height=16pt, inner sep=0pt},
    edge stalk h DarkRed/.style={rectangle, rounded corners=1pt, fill=white, draw=DarkRed, minimum width=16pt, minimum height=4pt, inner sep=0pt},
    arrow/.style={-{Stealth[length=4pt]}, thick},
    arrowDarkRed/.style={-{Stealth[length=4pt]}, thick, DarkRed},
    every node/.style={font=\small},
]

\node[vertex, draw=DarkRed] (vx) at (0, 2) {};
\node[vertex, draw=DarkGreen] (vz1) at (0, 0) {};
\node[vertex, draw=DarkOrange1] (va1) at (2.5, 0) {};
\node[vertex, draw=DarkGreen] (vz2) at (5, 0) {};
\node[vertex, draw=DarkGreen] (vy) at (5, 2) {};
\node[vertex, draw=DarkRed] (vp) at (2.5, 2) {};

\node[above left] at (vx) {$\overline{\mathbf{x}}$};
\node[below left] at (vz1) {$z^{\scriptscriptstyle(1)}$};
\node[below] at (va1) {$\overline{a}^{\scriptscriptstyle(1)}$};
\node[below right] at (vz2) {$z^{\scriptscriptstyle(2)}$};
\node[above right] at (vy) {$\hat{y}$};
\node[above] at (vp) {$\mathbf{p}$};

\draw[thick] (vx) -- (vz1);
\draw[thick, DarkRed] (vz1) -- (va1);
\draw[thick, DarkRed] (va1) -- (vz2);
\draw[thick] (vz2) -- (vy);
\draw[thick] (vp) -- (va1);

\node[edge stalk v] (e1) at (0, 1) {};
\node[edge stalk h] (e2) at (1.25, 0) {};
\node[edge stalk h] (e3) at (3.75, 0) {};
\node[edge stalk v] (e4) at (5, 1) {};
\node[edge stalk v] (e5) at (2.5, 1) {};


\draw[arrow] (-0.3, 1.85) -- (-0.3, 1.15);
\draw[arrow] (-0.3, 0.15) -- (-0.3, 0.85);
\node[left, font=\scriptsize] at (-0.35, 1.5) {$\overline{W}^{\scriptscriptstyle(1)}$};
\node[left, font=\scriptsize] at (-0.35, 0.5) {$I_3$};

\draw[arrow] (0.15, -0.3) -- (1.1, -0.3);
\draw[arrow] (2.2, -0.3) -- (1.4, -0.3);
\node[below, font=\scriptsize] at (0.55, -0.35) {$R^{z^{(1)}}$};
\node[below, font=\scriptsize] at (1.85, -0.35) {$P_3$};

\draw[arrow] (2.8, -0.3) -- (3.6, -0.3);
\draw[arrow] (4.85, -0.3) -- (3.9, -0.3);
\node[below, font=\scriptsize] at (3.15, -0.35) {$\overline{W}^{\scriptscriptstyle(2)}$};
\node[below, font=\scriptsize] at (4.45, -0.35) {$I_2$};

\draw[arrow] (5.3, 1.85) -- (5.3, 1.15);
\draw[arrow] (5.3, 0.15) -- (5.3, 0.85);
\node[right, font=\scriptsize] at (5.35, 1.5) {$I_1$};
\node[right, font=\scriptsize] at (5.35, 0.5) {$\phi$};

\draw[arrow] (2.85, 1.85) -- (2.85, 1.15);
\draw[arrow] (2.85, 0.15) -- (2.85, 0.85);
\node[right, font=\scriptsize] at (2.9, 1.5) {$\sqrt{\gamma}I_1$};
\node[right, font=\scriptsize] at (2.9, 0.5) {$\sqrt{\gamma}P_{\{3\}}$};

\end{tikzpicture}
\caption{Pinning in the sheaf formulation}
\label{fig:pin-sheaf}
\end{subfigure}

\caption{\textbf{Pinning a neuron.} (a) Fixing the third hidden neuron to the value $\mathbf{p}$ in the feedforward network affects only the output neuron through the red edge. (b) In the sheaf formulation, the pinned value $\mathbf{p}$ is imposed through the restriction map $P_{\{3\}}$, which projects onto the third component of $a^{(1)}$. The disturbance propagates to both adjacent stalks $z^{(1)}$ and $z^{(2)}$ (red edges).}
\label{fig:pin}
\end{figure}

The modified dynamics take the same form as~\eqref{eq:neural-restricted}:
pinning simply enlarges the boundary data. The pinned coordinates join the
fixed set~$U$, and the target values~$\mathbf{p}$ are appended to~$u$. 

The bidirectional nature of pinning distinguishes it from feedforward
interventions such as dropout~\cite{srivastava2014dropout}.  It is closer
in spirit to relaxation-based inference schemes in which constraints may
be imposed on internal variables of a computation graph
\cite{scellier2017equilibrium,millidge2022predictive}, though the
mechanism here is sheaf diffusion minimizing edgewise discrepancy. In the
feedforward view, fixing a neuron at layer~$\ell$ leaves
layers~$1, \ldots, \ell{-}1$ untouched. In the sheaf dynamics, discrepancies
created by the penalty edge propagate upstream through the weight and
activation edges, forcing earlier layers to adjust. The equilibrium
of the pinned dynamics reflects a global compromise across the entire
network, not a local override.

\begin{note}[Partial clamping]\label{note:partial-clamping}
  The unpinned dynamics converge to a global section of the sheaf, where
  every activation satisfies
  $a^{(\ell)*} = \mathrm{ReLU}(z^{(\ell)*})$ and the total discrepancy
  vanishes (\cref{prop:forward-pass}). Pinning disrupts this guarantee.
  The modified inhomogeneous term
  in the pinned dynamics can shift the equilibrium so that the
  natural computation $\mathrm{ReLU}(z^{(\ell)*})$ conflicts with the
  pinning constraint, forcing the trajectory onto a ReLU boundary where
  some $z^{(\ell)*}_j = 0$. At such a boundary the steady state may
  satisfy $a^{(\ell)*} \neq \mathrm{ReLU}(z^{(\ell)*})$, so the
  equilibrium is no longer a global section in general. Gould's CPWA
  framework~\cite{gould2025thesis} guarantees, via Clarke's generalized
  gradient~\cite{clarke1983optimization}, that the pinned system still
  converges to a generalized critical point. However, this critical point
  need not have zero discrepancy; it is a local minimum of the Lyapunov
  function under the fast selection
  rule~\cite[Lem.~7.4.23]{gould2025thesis}, not necessarily the global
  one.
\end{note}

\subsection{Training via Joint Dynamics}\label{subsec:training}

\Cref{sec:convergence} and \cref{subsec:pinning} treated the restriction
maps as fixed: the network's weights were given, and only the 0-cochain
$\omega$ evolved under the heat equation. Training a neural
network means adjusting the weights themselves. In the sheaf framework,
this amounts to evolving the restriction maps
$\mathcal{F}_{v \trianglelefteqslant e}$ alongside the cochain, so that
the sheaf structure adapts to reduce discrepancy. This is the joint
opinion-expression dynamics
of~\cite{hansen2021opinion,bosca2026selective,gould2025thesis},
transplanted to the neural setting. From the perspective of
learning algorithms this resembles other local learning schemes on
general computation graphs, such as predictive coding networks
\cite{whittington2017predictive,millidge2022predictive,seely2025sheaf}, but here the
update rule arises directly from gradient descent on a global sheaf
discrepancy functional, with weights and activations evolving as a
coupled system rather than in alternating phases.

\subsubsection{Evolving restriction maps}\label{sssec:evolving-maps}

Consider a supervised training pair $(\mathbf{x}, \mathbf{y})$. The input
vertex is already stubborn at $\mathbf{x}$, and following
\cref{subsec:pinning} we fix the output vertex $\hat{y}$ to the
target~$\mathbf{y}$. In general, the current weights cannot produce the
target from the given input, so the equilibrium may retain nonzero discrepancy. The key
observation is that we can also evolve the restriction maps, allowing the
sheaf structure itself to adapt until the 0-cochain becomes a global
section. This is the ``learning to lie'' mechanism from~\cite{hansen2021opinion}. Hernandez Caralt et al.~\cite{hernandez2024joint} apply joint diffusion of node features and restriction maps to learn the structure in a sheaf neural network setting, discretizing the coupled system of~\cite{hansen2021opinion} as a neural network layer. In~\cite{bosca2026selective} we explore a constrained version where only a
designated subset of restriction maps evolves, which is the setting relevant
here, since not all restriction maps in our construction correspond to
trainable parameters.

The general update rule is gradient descent on the local edge discrepancy.
For an edge $e = u \to v$ with restriction maps
$\mathcal{F}_{v \trianglelefteqslant e}$ and
$\mathcal{F}_{u \trianglelefteqslant e}$, the edgewise discrepancy is
$\mathcal{F}_{v \trianglelefteqslant e}\, x_v -
\mathcal{F}_{u \trianglelefteqslant e}\, x_u$, and gradient descent on
$\frac{1}{2}\|\mathcal{F}_{v \trianglelefteqslant e}\, x_v -
\mathcal{F}_{u \trianglelefteqslant e}\, x_u\|^2$ with respect to
$\mathcal{F}_{v \trianglelefteqslant e}$ gives
\begin{equation}\label{eq:restriction-map-evolution}
  \frac{d}{dt}\mathcal{F}_{v \trianglelefteqslant e}
    = -\beta\bigl(\mathcal{F}_{v \trianglelefteqslant e}\, x_v
      - \mathcal{F}_{u \trianglelefteqslant e}\, x_u\bigr)\, x_v^T,
\end{equation}
where $\beta > 0$ is the learning rate for the structural dynamics. Each
restriction map updates independently, using only the discrepancy at its own
edge and the data at the adjacent vertex. This locality is a distinctive
feature of sheaf-based training.

In the neural sheaf, the trainable restriction maps are the extended weight
matrices $\overline{W}^{(\ell)}$ on the weight edges
$e_{z^{(\ell)}}$. Applying the same idea of gradient descent to
these edges, where
$\mathcal{F}_{v_{a^{(\ell-1)}} \trianglelefteqslant
e_{z^{(\ell)}}} = \overline{W}^{(\ell)}$ and
$\mathcal{F}_{v_{z^{(\ell)}} \trianglelefteqslant e_{z^{(\ell)}}} =
I_{n_\ell}$, yields updates for the weight matrix and bias vector
separately:
\begin{equation}\label{eq:weight-bias-updates}
  \frac{d}{dt}W^{(\ell)} = -\beta\, (
\overline{W}^{(\ell)}\,\overline{a}^{(\ell-1)} - z^{(\ell)})\, a^{(\ell-1)T},
  \qquad
  \frac{d}{dt}b^{(\ell)} = -\beta\, (\overline{W}^{(\ell)}\,\overline{a}^{(\ell-1)} - z^{(\ell)}).
\end{equation}
Note that the decomposition into $W^{(\ell)}$ and $b^{(\ell)}$
is enforced by projecting onto the subspace of matrices that preserve the
diagonal sparsity pattern of
$B^{(\ell)} = \mathrm{diag}(b^{(\ell)})$.

\subsubsection{The joint dynamics}\label{sssec:joint-dynamics}

Combining the cochain dynamics~\eqref{eq:neural-restricted} with the weight
updates~\eqref{eq:weight-bias-updates} across all layers yields a coupled
system. To write it compactly, we introduce two projections.

The \emph{free-cochain projection} $P_{\Omega}$ zeros out the
components of $\overline{\omega}$ that correspond to fixed boundary data,
namely the input and output vertices, and the ones blocks in the extended activations.
When the sheaf Laplacian $L_{\mathcal{F}}\,\overline{\omega}$ acts on the
full cochain, $P_{\Omega}$ extracts only the update to the free
coordinates $\omega$, leaving the stubborn coordinates untouched.

The \emph{trainable-map projection} $\Pi_{W}$ zeros out velocity
components for restriction maps that should not evolve, namely the identity
maps on weight edges, the ReLU diagonal matrices on activation edges, and
the projection maps $P_{n_\ell}$. It also enforces the diagonal sparsity of
$B^{(\ell)}$. What remains after projection are the updates to the
extended weight matrices $\overline{W}^{(\ell)}$, the only restriction maps
we allow to evolve.

With these projections, the \emph{joint dynamics} take the form
\begin{equation}\label{eq:joint-dynamics}
\begin{aligned}
  \frac{d\omega}{dt}
    &= -\alpha \, P_{\Omega}
      \bigl(\delta(t)^T \delta(t)\, \overline{\omega}(t)\bigr), \\[4pt]
  \frac{d\delta}{dt}
    &= -\beta \, \Pi_{W}
      \bigl(\delta(t)\, \overline{\omega}(t)\,
        \overline{\omega}(t)^T\bigr),
\end{aligned}
\end{equation}
where $\delta$ denotes the coboundary operator encoding all restriction maps.
The first equation is sheaf diffusion restricted to free coordinates, as
in~\eqref{eq:neural-restricted}, but now with time-varying restriction maps.
The second equation aggregates the edgewise
updates~\eqref{eq:restriction-map-evolution} across all weight edges. Both
equations are gradient descent on the total discrepancy energy
$\Psi(\overline{\omega}, \delta) = \frac{1}{2}\|\delta\,\overline{\omega}\|^2$
with respect to $\overline{\omega}$ and $\delta$ respectively. The rates $\alpha$ and
$\beta$ control the relative speeds of cochain and weight updates.

In the language of~\cite{bosca2026selective}, this evolution can also be
interpreted through an auxiliary \emph{parameter sheaf}
$\mathcal{H}^{W}$ whose vertex data are spaces of
admissible restriction maps, and whose own Laplacian depends on the current
cochain $\overline{\omega}$. We defer this interpretation, along with a
detailed figure of the coupled sheaves, to Appendix~\ref{ap:structure-sheaf}.

\subsubsection{Loss functions and the output edge}\label{sssec:loss-functions}

The edge potential framework from \cref{subsec:final-activation} extends
naturally to training. We identify the output edge potential
$U_{\mathrm{out}}$ with a loss function
$\mathcal{L}(\phi(z^{(k+1)}), \mathbf{y})$, where $\phi$ is the final
activation and $\hat{y} = \mathbf{y}$ is the target output, held fixed as
a stubborn vertex.

We say that a loss function fits the \emph{discrepancy framework} if it
depends on predictions and targets only through their difference
$d = \hat{y} - \mathbf{y}$, so that
$U_{\mathrm{out}}(\hat{y}, \mathbf{y}) = f(d)$ for some convex function
$f \colon \mathbb{R}^{n_{k+1}} \to \mathbb{R}$. In this case, the
gradient $\nabla f(d)$ produces a nonlinear sheaf Laplacian
$L_{\mathcal{F}}^{\nabla f} = \delta^T \circ \nabla f \circ \delta$ in the
sense of Hansen and Ghrist~\cite[\S10]{hansen2021opinion}. The squared
error $f(d) = \frac{1}{2}\|d\|^2$ recovers the standard sheaf Laplacian.
The L1 loss $f(d) = \|d\|_1$ has gradient
$(\nabla f(d))_i = \mathrm{sign}(d_i)$.
The formulas for other losses in this family (Huber, $p$-norm) can be found
in Appendix~\ref{ap:edge-potentials}.

Cross-entropy
$\mathcal{L}(\hat{y}, \mathbf{y}) = -\sum_i y_i \log \hat{y}_i$ does not
fit the discrepancy framework, since it depends on the predicted
probabilities $\hat{y}$ and the target $\mathbf{y}$ separately rather than
on their difference. However, paired with the local adjoint perspective from
\cref{subsec:final-activation}, the expression simplifies. For softmax
$\phi(z)_i = e^{z_i}/\sum_j e^{z_j}$, a short calculation gives
\begin{equation}\label{eq:ce-simplification}
  J_\phi(z^{(k+1)})^T \nabla_{\hat{y}}\, \mathcal{L}
    = \phi(z^{(k+1)}) - \mathbf{y},
\end{equation}
and the same identity holds for sigmoid with binary cross-entropy. The
Jacobian of the activation cancels with the gradient structure of the loss,
so the resulting dynamics on the output edge takes the discrepancy back via the identity. Classification training thus fits within the
same local update framework as regression.

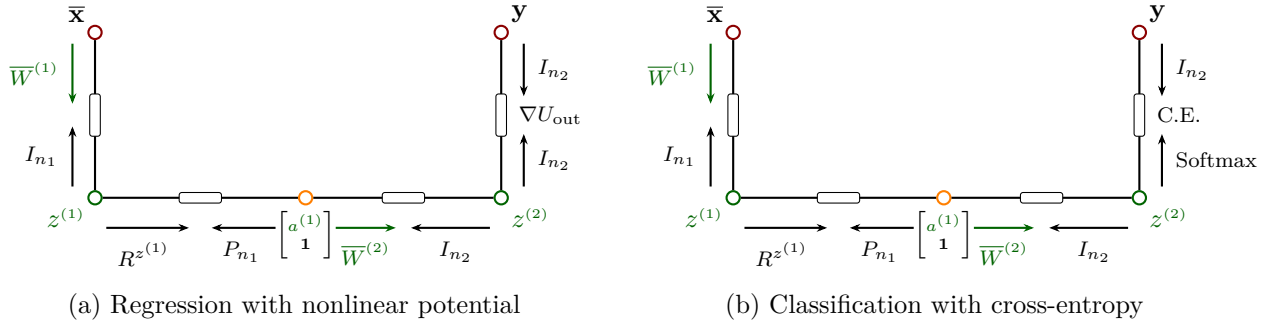
\begin{figure}[H]
\centering
\begin{subfigure}[b]{0.48\textwidth}
\centering
\begin{tikzpicture}[
    vertex/.style={circle, fill=white, draw=black, thick, minimum size=5pt, inner sep=0pt},
    edge stalk h/.style={rectangle, rounded corners=1pt, fill=white, draw=black, minimum width=16pt, minimum height=4pt, inner sep=0pt},
    edge stalk v/.style={rectangle, rounded corners=1pt, fill=white, draw=black, minimum width=4pt, minimum height=16pt, inner sep=0pt},
    arrow/.style={-{Stealth[length=4pt]}, thick},
    green arrow/.style={-{Stealth[length=4pt]}, thick, color=DarkGreen},
    every node/.style={font=\small},
]

\node[vertex, draw=DarkRed] (vx) at (0, 2.2) {};
\node[vertex, draw=DarkGreen] (vz1) at (0, 0) {};
\node[vertex, draw=DarkOrange1] (va1) at (2.8, 0) {};
\node[vertex, draw=DarkGreen] (vz2) at (5.4, 0) {};
\node[vertex, draw=DarkRed] (vy) at (5.4, 2.2) {};

\node[above left] at (vx) {$\overline{\mathbf{x}}$};
\node[below left, color=DarkGreen] at (vz1) {$z^{\scriptscriptstyle(1)}$};
\node[below=1pt, font=\tiny] at (va1) {$\begin{bmatrix} \textcolor{DarkGreen}{a^{\scriptscriptstyle(1)}} \\ \mathbf{1} \end{bmatrix}$};
\node[below right, color=DarkGreen] at (vz2) {$z^{\scriptscriptstyle(2)}$};
\node[above right] at (vy) {$\mathbf{y}$};

\draw[thick] (vx) -- (vz1);
\draw[thick] (vz1) -- (va1);
\draw[thick] (va1) -- (vz2);
\draw[thick] (vz2) -- (vy);

\node[edge stalk v] (e1) at (0, 1.1) {};
\node[edge stalk h] (e2) at (1.4, 0) {};
\node[edge stalk h] (e3) at (4.1, 0) {};
\node[edge stalk v] (e4) at (5.4, 1.1) {};

\draw[green arrow] (-0.3, 2.05) -- (-0.3, 1.25);
\draw[arrow] (-0.3, 0.15) -- (-0.3, 0.95);
\node[left, font=\scriptsize, color=DarkGreen] at (-0.35, 1.65) {$\overline{W}^{\scriptscriptstyle(1)}$};
\node[left, font=\scriptsize] at (-0.35, 0.55) {$I_{n_1}$};

\draw[arrow] (0.15, -0.4) -- (1.25, -0.4);
\draw[arrow] (2.4, -0.4) -- (1.55, -0.4);
\node[below, font=\scriptsize] at (0.65, -0.45) {$R^{z^{(1)}}$};
\node[below, font=\scriptsize] at (1.95, -0.45) {$P_{n_1}$};

\draw[green arrow] (3.2, -0.4) -- (4.0, -0.4);
\draw[arrow] (5.25, -0.4) -- (4.2, -0.4);
\node[below, font=\scriptsize, color=DarkGreen] at (3.6, -0.45) {$\overline{W}^{\scriptscriptstyle(2)}$};
\node[below, font=\scriptsize] at (4.8, -0.45) {$I_{n_2}$};

\draw[arrow] (5.7, 2.05) -- (5.7, 1.35);
\node[right, font=\scriptsize] at (5.75, 1.7) {$I_{n_2}$};
\node[right, font=\scriptsize] at (5.5, 1.1) {$\nabla U_{\text{out}}$};
\draw[arrow] (5.7, 0.15) -- (5.7, 0.85);
\node[right, font=\scriptsize] at (5.75, 0.5) {$I_{n_2}$};

\end{tikzpicture}
\caption{Regression with nonlinear potential}
\end{subfigure}
\hspace{0.02\textwidth}
\begin{subfigure}[b]{0.48\textwidth}
\centering
\begin{tikzpicture}[
    vertex/.style={circle, fill=white, draw=black, thick, minimum size=5pt, inner sep=0pt},
    edge stalk h/.style={rectangle, rounded corners=1pt, fill=white, draw=black, minimum width=16pt, minimum height=4pt, inner sep=0pt},
    edge stalk v/.style={rectangle, rounded corners=1pt, fill=white, draw=black, minimum width=4pt, minimum height=16pt, inner sep=0pt},
    arrow/.style={-{Stealth[length=4pt]}, thick},
    green arrow/.style={-{Stealth[length=4pt]}, thick, color=DarkGreen},
    every node/.style={font=\small},
]

\node[vertex, draw=DarkRed] (vx) at (0, 2.2) {};
\node[vertex, draw=DarkGreen] (vz1) at (0, 0) {};
\node[vertex, draw=DarkOrange1] (va1) at (2.8, 0) {};
\node[vertex, draw=DarkGreen] (vz2) at (5.4, 0) {};
\node[vertex, draw=DarkRed] (vy) at (5.4, 2.2) {};

\node[above left] at (vx) {$\overline{\mathbf{x}}$};
\node[below left, color=DarkGreen] at (vz1) {$z^{\scriptscriptstyle(1)}$};
\node[below=1pt, font=\tiny] at (va1) {$\begin{bmatrix} \textcolor{DarkGreen}{a^{\scriptscriptstyle(1)}} \\ \mathbf{1} \end{bmatrix}$};
\node[below right, color=DarkGreen] at (vz2) {$z^{\scriptscriptstyle(2)}$};
\node[above right] at (vy) {$\mathbf{y}$};

\draw[thick] (vx) -- (vz1);
\draw[thick] (vz1) -- (va1);
\draw[thick] (va1) -- (vz2);
\draw[thick] (vz2) -- (vy);

\node[edge stalk v] (e1) at (0, 1.1) {};
\node[edge stalk h] (e2) at (1.4, 0) {};
\node[edge stalk h] (e3) at (4.1, 0) {};
\node[edge stalk v] (e4) at (5.4, 1.1) {};

\draw[green arrow] (-0.3, 2.05) -- (-0.3, 1.25);
\draw[arrow] (-0.3, 0.15) -- (-0.3, 0.95);
\node[left, font=\scriptsize, color=DarkGreen] at (-0.35, 1.65) {$\overline{W}^{\scriptscriptstyle(1)}$};
\node[left, font=\scriptsize] at (-0.35, 0.55) {$I_{n_1}$};

\draw[arrow] (0.15, -0.4) -- (1.25, -0.4);
\draw[arrow] (2.4, -0.4) -- (1.55, -0.4);
\node[below, font=\scriptsize] at (0.65, -0.45) {$R^{z^{(1)}}$};
\node[below, font=\scriptsize] at (1.95, -0.45) {$P_{n_1}$};

\draw[green arrow] (3.2, -0.4) -- (4.0, -0.4);
\draw[arrow] (5.25, -0.4) -- (4.2, -0.4);
\node[below, font=\scriptsize, color=DarkGreen] at (3.6, -0.45) {$\overline{W}^{\scriptscriptstyle(2)}$};
\node[below, font=\scriptsize] at (4.8, -0.45) {$I_{n_2}$};

\draw[arrow] (5.7, 2.05) -- (5.7, 1.35);
\node[right, font=\scriptsize] at (5.75, 1.7) {$I_{n_2}$};
\node[right, font=\scriptsize] at (5.5, 1.1) {C.E.};
\draw[arrow] (5.7, 0.15) -- (5.7, 0.85);
\node[right, font=\scriptsize] at (5.75, 0.5) {$\text{Softmax}$};

\end{tikzpicture}
\caption{Classification with cross-entropy}
\end{subfigure}

\caption{Two training configurations for a 1-hidden-layer network. Red nodes are pinned boundary conditions, green nodes are dynamic pre-activation variables, and yellow nodes are dynamic post-activation variables with some part fixed. Green arrows represent trainable weight matrices. \textbf{Left:} Regression with identity activation composed with gradient of nonlinear potential $\nabla U$, enabling other losses beyond L2, like L1, Lp, or Huber. \textbf{Right:} Classification with Softmax/Sigmoid activation and cross-entropy loss, where dynamics simplify to allow a sheaf-like construction.}
\label{fig:training}
\end{figure}

\subsubsection{Regularization}\label{sssec:regularization}

The total discrepancy
$\Psi(\overline{\omega}, \delta) = \frac{1}{2}\|\delta\,\overline{\omega}\|^2$ is not
coercive. A priori, there is no guarantee that trajectories
of~\eqref{eq:joint-dynamics} remain bounded, since the cochain and the
coboundary operator can grow simultaneously while
$\delta\,\overline{\omega}$ stays small. Standard L2 regularization
resolves this. The \emph{regularized energy functional} is
\begin{equation}\label{eq:regularized-energy-training}
  \mathcal{L}_{\lambda,\mu}(\overline{\omega}, \delta)
    = \tfrac{1}{2}\|\delta\,\overline{\omega}\|^2
      + \tfrac{\lambda}{2}\|\omega\|^2
      + \tfrac{\mu}{2}\|\delta\|_F^2,
\end{equation}
where $\lambda, \mu > 0$ are regularization strengths. The gradient descent
dynamics of $\mathcal{L}_{\lambda,\mu}$ are
\begin{equation}\label{eq:regularized-joint}
\begin{aligned}
  \frac{d\omega}{dt}
    &= -\alpha P_{\Omega}\Bigl(
      \delta^T \delta\,\overline{\omega}
        + \lambda\, \overline{\omega} \Bigr), \\[4pt]
  \frac{d\delta}{dt}
    &= -\beta \, \Pi_{W} \Bigl(
      \delta\, \overline{\omega}\, \overline{\omega}^T
        + \mu\, \delta \Bigr).
\end{aligned}
\end{equation}
The cochain term $(\lambda/2)\|\omega\|^2$ penalizes large free
coordinates, acting as weight decay on the internal state. The weight term
$(\mu/2)\|\delta\|_F^2$ penalizes large restriction maps, corresponding to
standard weight decay on the network parameters. Both regularization terms
admit sheaf-theoretic interpretations as weighted reluctance, as shown
in~\cite{bosca2026selective,hansen2021opinion}. The general form with
nonzero anchoring, its connection to the framework
in~\cite{bosca2026selective}, and the diagrammatic realization within the
sheaf are developed in Appendix~\ref{ap:regularization}.

\begin{note}[Nonlinear regularization]
  The L2 penalties above correspond to quadratic edge potentials on the
  regularization edges. Just as in \cref{sssec:loss-functions}, these can
  be replaced by any convex function of the discrepancy: L1 regularization
  uses $f(d) = \|d\|_1$, elastic net combines L1 and L2, and so on.
  Each choice produces a different nonlinear Laplacian on the augmented
  graph, selecting for different sparsity or smoothness properties of
  the equilibrium.
\end{note}

\subsubsection{Batch Processing}\label{subsec:batch}

The development so far treats a single input $\mathbf{x} \in \mathbb{R}^{n_0}$.
Training and inference require processing multiple inputs simultaneously. The
extension is straightforward. Given a batch
$\mathbf{X} = (\mathbf{x}_1, \ldots, \mathbf{x}_M) \in
\mathbb{R}^{n_0 \times M}$, we replace each vertex stalk
$\mathbb{R}^{n_\ell}$ with $\mathbb{R}^{n_\ell \times M}$, so that
the 0-cochain stores $M$ columns rather than one. The restriction maps
$\overline{W}^{(\ell)}$ act identically on each column, since they encode
network structure rather than data.

The ReLU operation requires more care. For a single input, ReLU is encoded
by the diagonal matrix $R^{z^{(\ell)}}$ acting via matrix multiplication.
For a batch, each column $z^{(\ell)}_m$ activates a different subset of
neurons, so there is no single matrix that applies to the full batch. Instead,
we define a mask matrix
$\mathrm{MASK}^{Z^{(\ell)}} \in \{0,1\}^{n_\ell \times M}$ with entries
$\mathrm{MASK}^{Z^{(\ell)}}_{im} = 1$ if $Z^{(\ell)}_{im} \geq 0$ and $0$
otherwise, and replace matrix multiplication by coordinatewise (Hadamard)
multiplication: $A^{(\ell)} = \mathrm{MASK}^{Z^{(\ell)}}\odot Z^{(\ell)}$.
This is no longer a linear map between stalks in the classical sheaf sense. The batch dynamics are equivalent to $M$ independent sheaves sharing the same
restriction maps but with separate activation patterns.

\begin{figure}[htbp]
    \centering
    \begin{tikzpicture}[
    vertex/.style={circle, fill=white, draw=black, thick, minimum size=6pt, inner sep=0pt},
    edge stalk h/.style={rectangle, rounded corners=1pt, fill=white, draw=black, minimum width=16pt, minimum height=4pt, inner sep=0pt},
    edge stalk v/.style={rectangle, rounded corners=1pt, fill=white, draw=black, minimum width=4pt, minimum height=16pt, inner sep=0pt},
    arrow/.style={-{Stealth[length=4pt]}, thick},
    every node/.style={font=\small},
    bendy arrow/.style={-{Stealth[length=4pt]}, thick}
]


\coordinate (T_Z1) at (0, 0);
\coordinate (T_X)  at (0, 2);
\coordinate (T_A1) at (6.0, 0); 
\coordinate (T_Z2) at (11.0, 0); 
\coordinate (T_Y)  at (11.0, 2);

\node[vertex, draw=DarkRed] (tx) at (T_X) {};
\node[vertex, draw=DarkGreen] (tz1) at (T_Z1) {};
\node[vertex, draw=DarkOrange1] (ta1) at (T_A1) {};
\node[vertex, draw=DarkGreen] (tz2) at (T_Z2) {};
\node[vertex, draw=DarkRed] (ty) at (T_Y) {};

\node[above left] at (tx) {$\overline{\mathbf{X}}$};
\node[below left] at (tz1) {$Z^{\scriptscriptstyle(1)}$};
\node[below] at (ta1) {$\overline{A}^{\scriptscriptstyle(1)}$};
\node[below right] at (tz2) {$Z^{\scriptscriptstyle(2)}$};
\node[above right] at (ty) {$\mathbf{Y}$};

\draw[thick] (tx) -- (tz1);
\draw[thick] (tz1) -- (ta1);
\draw[thick] (ta1) -- (tz2);
\draw[thick] (tz2) -- (ty);

\node[edge stalk v] at (0, 1) {};
\node[edge stalk h] at (3.0, 0) {}; 
\node[edge stalk h] at (8.5, 0) {}; 
\node[edge stalk v] at (11.0, 1) {};


\draw[arrow] (-0.3, 1.85) -- (-0.3, 1.15);
\draw[arrow] (-0.3, 0.15) -- (-0.3, 0.85);
\node[left, font=\scriptsize] at (-0.35, 1.5) {$\overline{W}^{\scriptscriptstyle(1)}$};
\node[left, font=\scriptsize] at (-0.35, 0.5) {$I_{n_1\times M}$};

\draw[arrow] (0.2, -0.3) -- (2.5, -0.3);
\draw[arrow] (5.7, -0.3) -- (3.5, -0.3);
\node[below, font=\scriptsize] at (1.35, -0.35) {$\text{MASK}^{Z^{(1)}}\odot$};
\node[below, font=\scriptsize] at (4.65, -0.35) {$P_{n_1\times M}$};

\draw[arrow] (6.4, -0.3) -- (8.0, -0.3);
\draw[arrow] (10.8, -0.3) -- (9.0, -0.3);
\node[below, font=\scriptsize] at (7.2, -0.35) {$\overline{W}^{\scriptscriptstyle(2)}$};
\node[below, font=\scriptsize] at (9.9, -0.35) {$I_{n_2\times M}$};

\draw[arrow] (11.3, 1.85) -- (11.3, 1.15);
\draw[arrow] (11.3, 0.15) -- (11.3, 0.85);
\node[right, font=\scriptsize] at (11.35, 1.5) {$I_{{n_2}\times M}$};
\node[right, font=\scriptsize] at (11.35, 0.5) {$\phi$};

\node[font=\LARGE] at (5.5, -1.5) {$\equiv$};


\begin{scope}[shift={(0,-4.5)}]

    \coordinate (B_Z1) at (0, 0);
    \coordinate (B_X)  at (0, 2);
    \coordinate (B_A1) at (7.0, 0); 
    \coordinate (B_Z2) at (11.0, 0);
    \coordinate (B_Y)  at (11.0, 2);

    \node[vertex, draw=DarkRed] (bx) at (B_X) {};
    \node[vertex, draw=DarkGreen] (bz1) at (B_Z1) {};
    \node[vertex, draw=DarkOrange1] (ba1) at (B_A1) {};
    \node[vertex, draw=DarkGreen] (bz2) at (B_Z2) {};
    \node[vertex, draw=DarkRed] (by) at (B_Y) {};

    \node[above left] at (bx) {$\overline{\mathbf{X}}$};
    \node[below left] at (bz1) {$Z^{\scriptscriptstyle(1)}$};
    \node[below] at (ba1) {$\overline{A}^{\scriptscriptstyle(1)}$};
    \node[below right] at (bz2) {$Z^{\scriptscriptstyle(2)}$};
    \node[above right] at (by) {$\mathbf{Y}$};

    \draw[thick] (bx) -- (bz1);
    \draw[thick] (ba1) -- (bz2);
    \draw[thick] (bz2) -- (by);
    
    \node[edge stalk v] at (0, 1) {};
    \node[edge stalk h] at (9.0, 0) {}; 
    \node[edge stalk v] at (11.0, 1) {};

    
    \def\SplitStart{0.8}
    \def\SplitEnd{5.5}
    \def\CurveLen{1.0}
    \def\StraightStart{1.8} 
    \def\StraightEnd{4.5}   
    
    \def\YTop{0.9}
    \def\YInner{0.35}
    
    \draw[thick] (bz1) -- (\SplitStart, 0);
    \draw[thick] (\SplitEnd, 0) -- (ba1);
    
    \coordinate (S) at (\SplitStart, 0);
    \coordinate (E) at (\SplitEnd, 0);
    
    \draw[thick] (S) to[out=0, in=180] (\StraightStart, \YTop);
    \draw[thick] (\StraightStart, \YTop) -- (\StraightEnd, \YTop) node[edge stalk h, pos=0.5] {};
    \draw[thick] (\StraightEnd, \YTop) to[out=0, in=180] (E);

    \draw[thick] (S) to[out=0, in=180] (\StraightStart, \YInner);
    \draw[thick] (\StraightStart, \YInner) -- (\StraightStart + 0.5, \YInner); 
    \draw[thick] (\StraightEnd - 0.5, \YInner) -- (\StraightEnd, \YInner);   
    \draw[thick] (\StraightEnd, \YInner) to[out=0, in=180] (E);

    \draw[thick] (S) to[out=0, in=180] (\StraightStart, -\YInner);
    \draw[thick] (\StraightStart, -\YInner) -- (\StraightStart + 0.5, -\YInner);
    \draw[thick] (\StraightEnd - 0.5, -\YInner) -- (\StraightEnd, -\YInner);
    \draw[thick] (\StraightEnd, -\YInner) to[out=0, in=180] (E);

    \draw[thick] (S) to[out=0, in=180] (\StraightStart, -\YTop);
    \draw[thick] (\StraightStart, -\YTop) -- (\StraightEnd, -\YTop) node[edge stalk h, pos=0.5] {};
    \draw[thick] (\StraightEnd, -\YTop) to[out=0, in=180] (E);

    \pgfmathsetmacro{\MidX}{(\SplitStart + \SplitEnd)/2}
    \fill (\MidX, 0.15) circle (1pt);
    \fill (\MidX, 0) circle (1pt);
    \fill (\MidX, -0.15) circle (1pt);


    \draw[arrow] (-0.3, 1.85) -- (-0.3, 1.15);
    \draw[arrow] (-0.3, 0.15) -- (-0.3, 0.85);
    \node[left, font=\scriptsize] at (-0.35, 1.5) {$\overline{W}^{\scriptscriptstyle(1)}$};
    \node[left, font=\scriptsize] at (-0.35, 0.5) {$I_{n_1\times M}$};

    
    \draw[bendy arrow] (\SplitStart, 0.4) to[out=0, in=180] (\StraightStart, \YTop + 0.4);
    \node[above, font=\scriptsize] at (\StraightStart, \YTop + 0.4) {$R^{z^{\scriptscriptstyle(1),1}}$};

    \draw[bendy arrow] (\SplitStart, -0.4) to[out=0, in=180] (\StraightStart, -\YTop - 0.4);
    \node[below, font=\scriptsize] at (\StraightStart, -\YTop - 0.4) {$R^{z^{\scriptscriptstyle(1),M}}$};

    
    \draw[bendy arrow] (\SplitEnd, 0.4) to[out=180, in=0] (\StraightEnd, \YTop + 0.4);
    \node[above, font=\scriptsize] at (\StraightEnd, \YTop + 0.4) {$P_{n_1,1}$};

    \draw[bendy arrow] (\SplitEnd, -0.4) to[out=180, in=0] (\StraightEnd, -\YTop - 0.4);
    \node[below, font=\scriptsize] at (\StraightEnd, -\YTop - 0.4) {$P_{n_1,M}$};

    \draw[arrow] (7.4, -0.3) -- (8.8, -0.3);
    \draw[arrow] (10.8, -0.3) -- (9.2, -0.3);
    \node[below, font=\scriptsize] at (8.1, -0.35) {$\overline{W}^{\scriptscriptstyle(2)}$};
    \node[below, font=\scriptsize] at (10.0, -0.35) {$I_{n_2\times M}$};

    \draw[arrow] (11.3, 1.85) -- (11.3, 1.15);
    \draw[arrow] (11.3, 0.15) -- (11.3, 0.85);
    \node[right, font=\scriptsize] at (11.35, 1.5) {$I_{{n_2}\times M}$};
    \node[right, font=\scriptsize] at (11.35, 0.5) {$\phi$};

\end{scope}

\end{tikzpicture}
    \caption{Batch processing for a one-hidden-layer network with batch
      size~$M$. \textbf{Top}: the batch sheaf with matrix-valued stalks, where
      ReLU acts via coordinatewise masking. \textbf{Bottom}: the equivalent
      representation as $M$ parallel activation edges, each carrying
      its own diagonal ReLU matrix $R^{z^{(1),m}}$ and projection
      $P_{n_1,m}$. The weight edges are shared across all inputs.}
    \label{fig:batch_processing}
\end{figure}
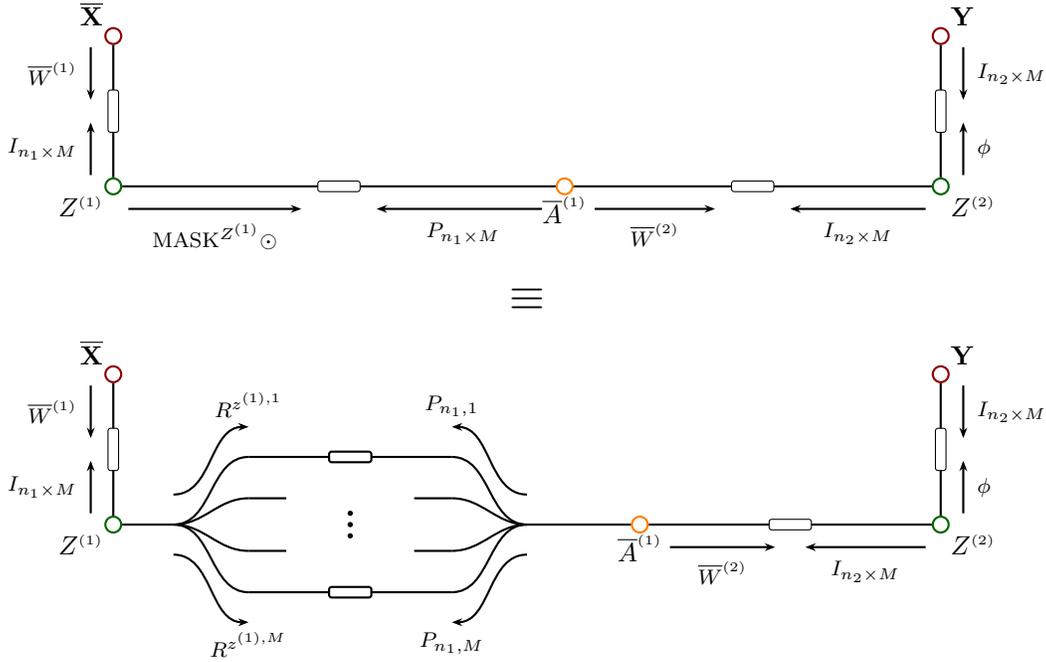

\subsection{Convergence and Timescale Considerations}\label{subsec:timescale}

The joint dynamics~\eqref{eq:joint-dynamics} sit at the intersection of
several existing frameworks, but are not fully covered by any one of them.
The analysis in~\cite{bosca2026selective} establishes convergence for joint
dynamics in the affine setting, without the piecewise linear switching
introduced by ReLU activations.
Gould~\cite{gould2025thesis} extends convergence theory to CPWA sheaves,
but treats only the case of evolving 0-cochains with fixed restriction maps.
Our setting requires both simultaneously: the coboundary operator $\delta$
and the 0-cochain (hence the polyhedral partition) change together.
We now examine what is known about convergence in this combined regime and
identify the open problems that remain.

\paragraph{Two obstructions to convergence.}
Extending convergence theory to the full
joint CPWA setting faces two obstructions. The first is \emph{norm blowup}:
the bilinear coupling $\delta\,\overline{\omega}$ can remain small while
$\|\omega\| + \|\delta\|_F \to \infty$, so the unregularized energy
$\Psi(\overline{\omega}, \delta) = \frac{1}{2}\|\delta\,\overline{\omega}\|^2$ has no
minimizer and trajectories may escape to infinity. The regularized
energy~\eqref{eq:regularized-energy-training} is coercive and rules out
this pathology, although we never saw it in practice with the unregularized potential.

The second obstruction is more subtle: \emph{persistent noncritical sliding}
along ReLU switching surfaces. At a boundary between activation regions, the
Filippov differential inclusion admits convex combinations of the one-sided
gradient flows. If these convex combinations include velocities tangent to
the boundary that neither leave the surface nor strictly decrease the energy,
trajectories can slide indefinitely without reaching a critical point.

In the fixed-weight setting of~\cref{sec:convergence}, sliding along
switching surfaces does not obstruct convergence: we show in the proof of
\cref{thm:final} that the energy decreases even during Filippov sliding. Phase-plane analysis
(\cref{subsec:convergence-experiments}) shows that ReLU boundary
crossings are transient; transient Filippov sliding is observed but does
not impede convergence (\cref{ap:ext-convergence}).
When weights evolve simultaneously, the switching surfaces themselves move,
and whether the energy remains nonincreasing during sliding in this coupled
setting is an open question.

\paragraph{Timescale separation.}
The ratio $\alpha/\beta$ controls the relative speeds of 0-cochain and weight
updates, and the timescale analysis in~\cite{bosca2026selective} provides
quantitative guidance. When $\alpha \gg \beta$, the 0-cochain equilibrates
rapidly while weights change slowly; this resembles standard training where
inference is essentially instantaneous relative to parameter updates. The
structural stagnation bound~\cite[Prop.~5.1]{bosca2026selective} quantifies
how much the weights can drift during this equilibration. For a trajectory
$(\omega(t), \delta(t))$ evolving over a time interval $[0,T]$, the total
weight displacement satisfies
\begin{equation}\label{eq:structural-stagnation}
  \|\delta(T) - \delta_0\|_F
    \;\leq\;  \frac{\beta\, B_\omega\, \|\delta_0\, \overline{\omega}_0\|}
                   {\alpha\, \lambda_{\mathrm{eff}}},
\end{equation}
where $B_\omega$ bounds the 0-cochain norm along the trajectory,
$\|\delta_0\,\overline{\omega}_0\|$ is the initial disagreement, and
$\lambda_{\mathrm{eff}}$ is an effective spectral gap measuring how
efficiently projected diffusion dissipates energy. The bound has a direct
interpretation: weight drift is proportional to the rate ratio
$\beta/\alpha$, scales with the 0-cochain magnitude, and is suppressed by
efficient fast dynamics (large $\lambda_{\mathrm{eff}}$). Since
$\lambda_{\mathrm{eff}}$ depends on the structure of the sheaf graph,
deeper networks with longer path graphs can be expected to have smaller
effective spectral gaps and hence slower equilibration per unit time. A
symmetric bound~\cite[Prop.~5.3]{bosca2026selective} governs 0-cochain
displacement when $\beta \gg \alpha$. These bounds are derived for affine
sheaves, so they apply within each activation region of our piecewise
linear setting, but a priori there is no guarantee that they compose
across region boundaries.

\paragraph{Experimental guidance from the stagnation bound.}
The bound~\eqref{eq:structural-stagnation} provides concrete guidance for
parameter selection. We fix $\alpha \sim 1$ and treat
$\lambda_{\mathrm{eff}}$ as independent of the number of samples in the batch $M$ (empirically, we observe it depends on network
depth, but not on batch size), leaving $\beta$ as the free parameter.
The remaining quantities depend on the initialization and $M$.

We initialize the 0-cochain with normalized values at each vertex stalk, so
that each of the $M$ columns in the batch 0-cochain contributes $O(1)$ to
the Frobenius norm \cite{He2015Delving}. This gives $B_\omega \sim \sqrt{M}$. The initial
disagreement aggregates $M$ columns of per-sample disagreement, each of
order $O(1)$ for randomly initialized weights, so
$\|\delta_0\,\overline{\omega}_0\| \sim \sqrt{M}$. Substituting into the
stagnation bound~\eqref{eq:structural-stagnation}:
\begin{equation}\label{eq:beta-scaling}
  \frac{\beta  B_\omega 
    \|\delta_0\,\overline{\omega}_0\|}
       {\alpha  \lambda_{\mathrm{eff}}}
  \;\sim\; \beta\, M.
\end{equation}
For the weight displacement to remain $O(1)$ during each equilibration
cycle, we need $\beta \sim 1/M$. This scaling ensures that weights do
not drift faster than the stalks can track through the sheaf graph.
\Cref{subsec:training-experiments} confirms this prediction: setting
$\beta = 1/n_{\mathrm{train}}$ works across all tasks and depths
without per-task tuning, while deviations in either direction degrade
performance.


\section{Experiments}\label{sec:experiments}

We implement the sheaf construction and the joint training dynamics from
\cref{sec:sheafyNN,sec:extensions} in a PyTorch library\footnote{\url{https://github.com/vicenbosca/neural-sheaf}} that encodes
feedforward ReLU networks as \texttt{NeuralSheaf} objects and integrates the
restricted heat equation via forward Euler. All computations use
\texttt{float64} precision on networks with at most 30 neurons per
hidden layer and two-dimensional synthetic data. We begin with the
inference dynamics (\cref{subsec:convergence-experiments}), visualizing how
the heat equation evolves the stalks toward the forward pass output.
\Cref{subsec:training-experiments} tests whether the joint dynamics from
\cref{subsec:training} can learn from data.
\Cref{subsec:diagnostics} illustrates two kinds of structural analysis that
the sheaf embedding makes accessible: spectral decomposition of the
restricted Laplacian, and per-edge discord as a local consistency measure.

\subsection{Convergence and Dynamics Visualization}
\label{subsec:convergence-experiments}

Given a trained network, we encode it as a sheaf, initialize all free stalks
to random values, and run the restricted heat
equation~\eqref{eq:neural-restricted}. \Cref{fig:convergence} shows the
resulting dynamics for a $[2,4,1]$ network: total discord decreases
monotonically on a log scale, with all edge types and layers converging to
machine precision. The output stalk trajectory converges from its random
starting point to the forward pass value.


\begin{figure}[ht]
  \centering
  \begin{subfigure}[t]{0.48\textwidth}
    \centering
    \includegraphics[width=\textwidth]{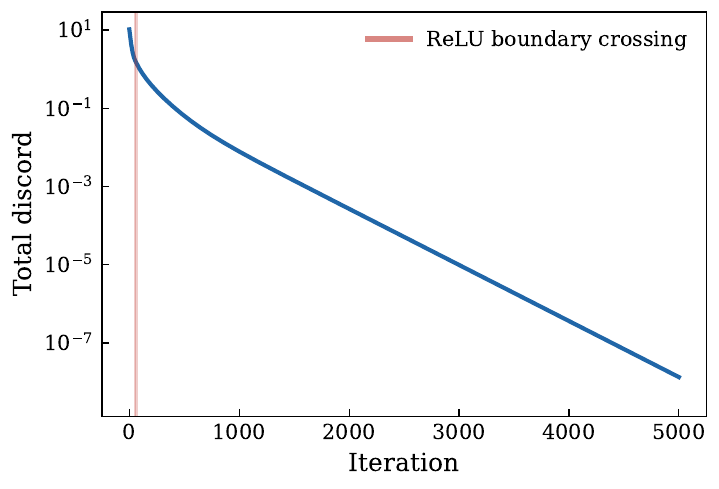}
    \caption{Total discord.}
    \label{fig:total-discord-1H}
  \end{subfigure}
  \hfill
  \begin{subfigure}[t]{0.48\textwidth}
    \centering
    \includegraphics[width=\textwidth]{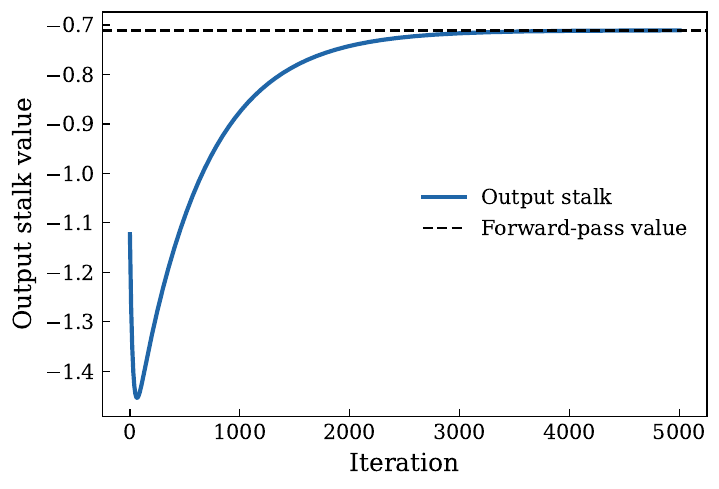}
    \caption{Output convergence.}
    \label{fig:output-1H}
  \end{subfigure}

  \caption{Convergence dynamics for a $[2,4,1]$ network with identity
  output activation, $\alpha = 1.0$, $dt = 0.01$, stalks initialized
  randomly.
    \textbf{(a)}~Total discord decreases monotonically on a log scale
  to machine precision.
    \textbf{(b)}~The output stalk converges from
  its random starting value to the forward-pass prediction (black dashed).
  This validates \cref{thm:convergence}.}
  \label{fig:convergence}
\end{figure}

Across all tested configurations, spanning different depths, output
activations, and batched inference with multiple simultaneous inputs, the
dynamics converge to the forward pass output within machine precision,
confirming \cref{thm:convergence,thm:final}. The unit determinant property
(\cref{lemma:det1}) holds to numerical precision across all networks tested.

Phase-plane plots of pre-activation stalk pairs
(\cref{ap:ext-convergence}) show trajectories navigating through
multiple activation regions before settling at the forward pass fixed point.
ReLU boundary crossings concentrate in the first few hundred iterations;
after that, activation patterns stabilize and the dynamics proceed as
linear diffusion within a single region. Most crossings are transversal,
but a fraction of stalks exhibit transient Filippov sliding, where opposing
forces from the weight and activation edges temporarily trap the trajectory
near a switching surface until downstream stalks converge and the trapping
condition dissolves. Discord decreases monotonically throughout all observed
sliding episodes, consistent with the energy decrease argument in the proof
of~\cref{thm:final}.

The library also supports the pinning mechanisms from
\cref{subsec:pinning}. Hard-pinning a hidden stalk to a non-forward-pass
value produces a different equilibrium, since the harmonic extension changes
when boundary conditions change. Soft pins with coupling strength~$\gamma$
smoothly interpolate between the free equilibrium and the target, approaching
the hard-pin limit as $\gamma \to \infty$.

\subsection{Sheaf-Based Training}\label{subsec:training-experiments}

We test the joint dynamics~\eqref{eq:joint-dynamics} as a training algorithm
on four synthetic benchmarks: paraboloid and saddle regression
($x_1^2 + x_2^2 - 2/3$ and $x_1^2 - x_2^2 + 0.5\sin(2x_1)$ on
$[-2,2]^2$ and $[0,2]^2$ respectively), circular binary classification
(disk vs.\ annulus), and four-class blob classification (anisotropic
Gaussians). Each task is run at two depths: one hidden layer
($[2, 30, 1]$ for regression, $[2, 25, 1]$ or $[2, 25, 4]$ for
classification) and two hidden layers ($[2, 10, 8, \cdot]$). The sheaf
trainer uses $\alpha = 1.0$, $\beta = 1/n_{\mathrm{train}}$
($n_{\mathrm{train}} = 300$), $dt = 0.005$ (1H) or $0.01$ (2H), with
$10^5$ (1H) or $2\cdot 10^5$ (2H) Euler steps. The SGD baseline uses identical
architectures with He initialization \cite{He2015Delving}. We compare by final test loss.

Both methods converge on all eight configurations, producing qualitatively
correct regression surfaces and decision boundaries.
\Cref{fig:training-results} shows representative results at one hidden
layer. On classification, sheaf training is competitive with SGD, with
both methods reaching comparable test loss and $100\%$ accuracy on the
circular data. On regression, SGD reaches lower loss, with sheaf-to-SGD
test loss ratios of roughly $2\times$ on the paraboloid and $1.3\times$
on the saddle, at 10k SGD epochs against 100k sheaf Euler steps. With a larger SGD budget, the gap widens (see \cref{tab:full-results}).

At two hidden layers, the SGD advantage widens across all tasks.
A targeted depth$\,\times\,dt$ sweep on the paraboloid task clarifies the
mechanism: the total integration time $T = dt \times n_{\mathrm{steps}}$,
not the step size $dt$ alone, governs performance. A fixed-$T$ control
confirms that the continuous-time analysis faithfully describes the
discrete implementation. Deeper networks converge more slowly per unit
$T$, consistent with the stagnation
bound~\eqref{eq:structural-stagnation}: the spectral gap $\lambda_1$ of
$L_{\mathcal{F}_t}[\Omega,\Omega]$ decreases systematically with depth
at initialization, predicting slower equilibration through the bound's
dependence on $\lambda_{\mathrm{eff}}$.


\begin{figure}[ht]
  \centering
  \begin{subfigure}[t]{0.48\textwidth}
    \centering
    \includegraphics[width=\textwidth]{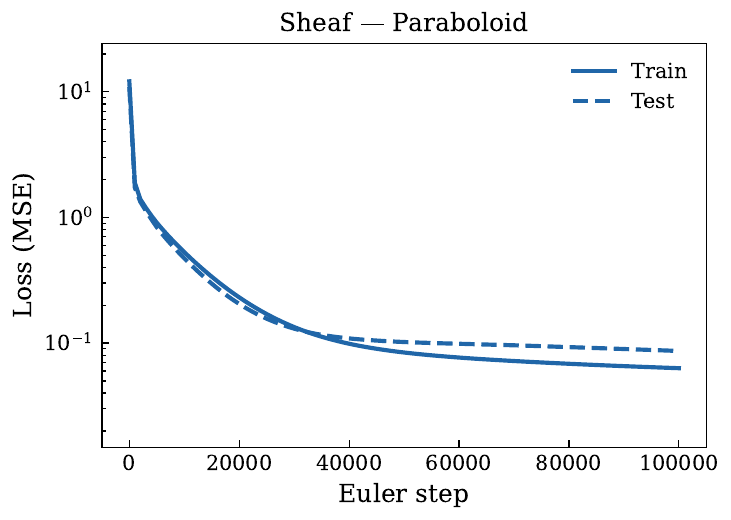}
    \caption{Sheaf training curve.}
    \label{fig:training-results:sheaf-curve}
  \end{subfigure}
  \hfill
  \begin{subfigure}[t]{0.48\textwidth}
    \centering
    \includegraphics[width=\textwidth]{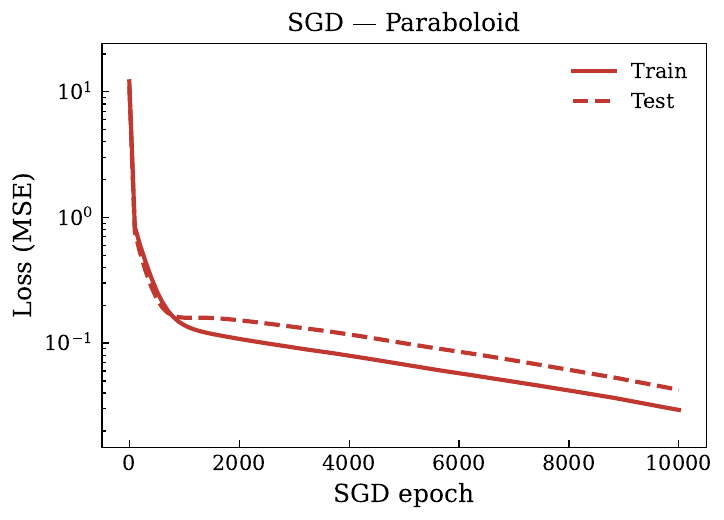}
    \caption{SGD training curve.}
    \label{fig:training-results:sgd-curve}
  \end{subfigure}

  \vspace{0.4em}

  \begin{subfigure}[t]{0.32\textwidth}
    \centering
    \includegraphics[width=\textwidth]{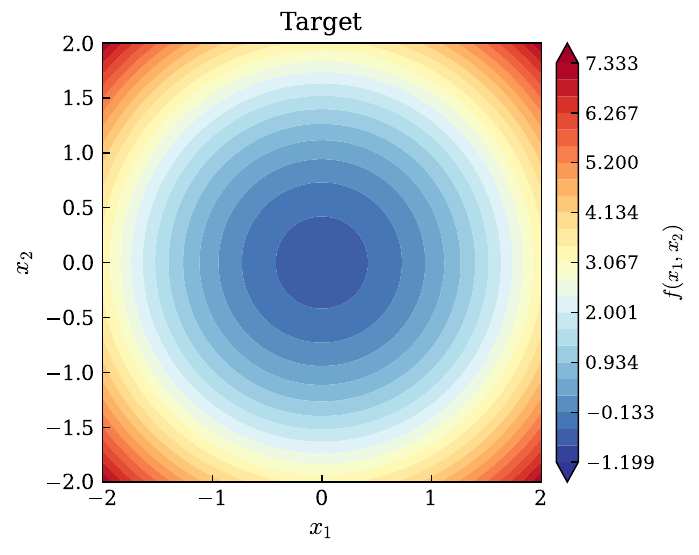}
    \caption{Target.}
    \label{fig:training-results:target}
  \end{subfigure}
  \hfill
  \begin{subfigure}[t]{0.32\textwidth}
    \centering
    \includegraphics[width=\textwidth]{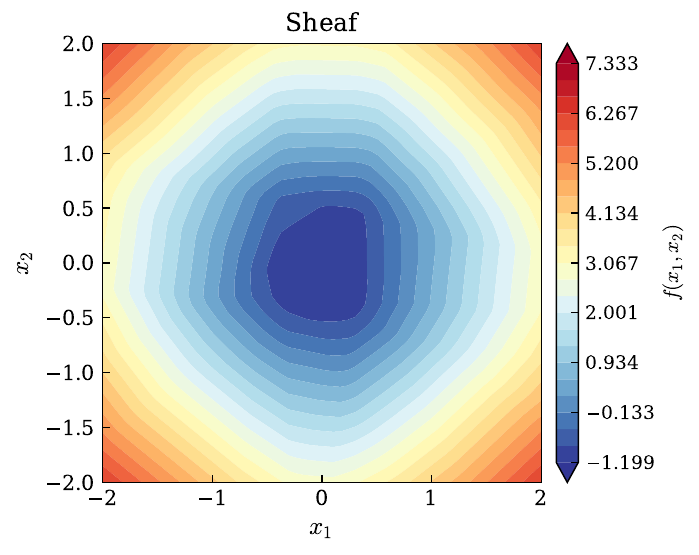}
    \caption{Sheaf.}
    \label{fig:training-results:sheaf-surface}
  \end{subfigure}
  \hfill
  \begin{subfigure}[t]{0.32\textwidth}
    \centering
    \includegraphics[width=\textwidth]{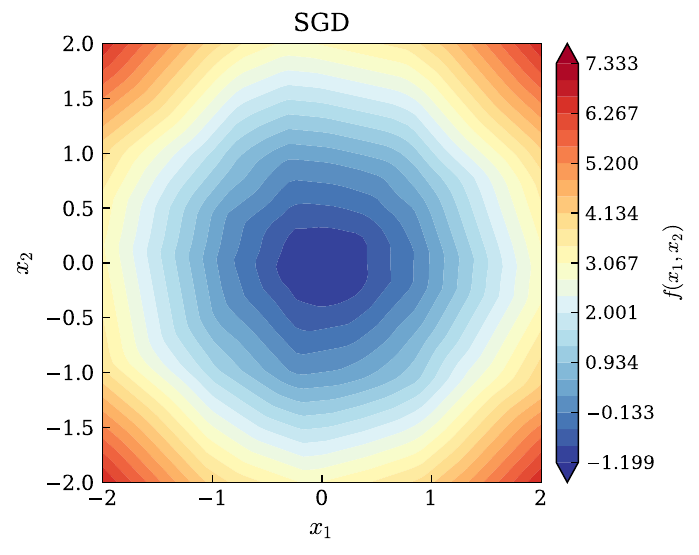}
    \caption{SGD.}
    \label{fig:training-results:sgd-surface}
  \end{subfigure}

  \caption{Sheaf training vs.\ SGD on the paraboloid regression task
    ($[2,30,1]$ architecture, identity output).
    \textbf{(a,\,b)}~Training and test loss on a log scale. Both methods
    converge smoothly; SGD reaches lower final loss.
    \textbf{(c--e)}~Learned surfaces on the $[-2,2]^2$ domain, shown with
    a shared colour scale. Both methods recover the paraboloid shape, with
    SGD producing a slightly smoother approximation.}
  \label{fig:training-results}
\end{figure}

A sweep of $\beta$ over several orders of magnitude on the paraboloid
task confirms the $1/M$ scaling predicted by the stagnation
bound~\eqref{eq:beta-scaling}. \Cref{fig:beta-scaling} plots final test
loss against the rescaled product $\beta M$ (where
$M = n_{\mathrm{train}}$ is the batch size) and reveals three regimes.
Too small a $\beta M$ starves the weight dynamics of information and the
network stagnates; too large a $\beta M$ inflates weight norms and
eventually causes divergence. Repeating the sweep at different batch
sizes confirms that $\beta M$ is the effective control parameter. The
default choice $\beta = 1/n_{\mathrm{train}}$ falls within the optimal
region without per-task tuning.


\begin{figure}[ht]
  \centering
  \begin{subfigure}[t]{0.48\textwidth}
    \centering
    \includegraphics[width=\textwidth]{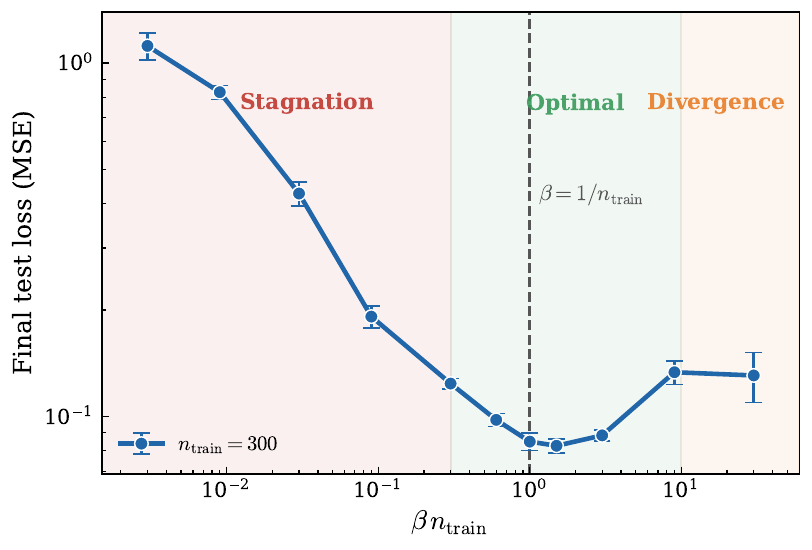}
    \caption{$\beta$ scaling.}
    \label{fig:beta-scaling}
  \end{subfigure}
  \hfill
  \begin{subfigure}[t]{0.48\textwidth}
    \centering
    \includegraphics[width=\textwidth]{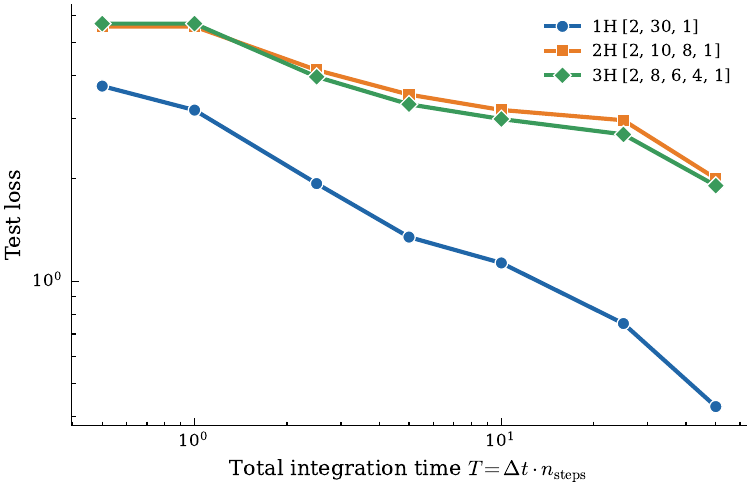}
    \caption{Integration time sweep.}
    \label{fig:depth-T-sweep}
  \end{subfigure}

  \caption{Hyperparameter investigation on the paraboloid task
    \textbf{(a)}~Final test loss
  as a function of $\beta M$ for several values of $\beta$. Experiments
  confirmed the $\beta \sim 1/M$ scaling from the stagnation bound: performance
  depends on the product $\beta M$, not on $\beta$ or $M$ individually. The
  dashed line marks the default $\beta = 1/n_{\mathrm{train}}$.
    \textbf{(b)}~Final test loss as a function of
  total integration time $T = \Delta t \times n_{\mathrm{steps}}$ for
  three architectures of increasing depth. For each depth, loss decreases
  monotonically as $T$ increases, confirming that longer equilibration
  improves the sheaf training output. Deeper networks achieve higher
  loss, requiring more integration time to reach comparable performance.
  This is consistent with the stagnation bound, which predicts slower
  convergence when $\lambda_1$ is smaller.
  \Cref{tab:spectral-gap-depth} confirms that $\lambda_1$ decreases with
  depth.}
  \label{fig:hyper-inv}
\end{figure}

Random initialization works best for regression tasks, while forward-pass
initialization suits classification. The two regimes produce visibly
different discord dynamics. With random initialization, discord starts
high across all edges and generally decreases as stalks and weights
co-equilibrate. With forward-pass initialization, internal edges begin
near a global section with near-zero discord, so the output edge carries
the full discrepancy between predictions and targets; this discord then
propagates inward as the weights adjust to accommodate the labeled data.
In both cases, discord concentrates on the nonlinear activation edges,
which mediate the ReLU constraint between pre-activation and
post-activation stalks. \Cref{fig:discord-residuals-sheaf} shows how per-edge discord residuals
shrink across three training stages, illustrating the progression from
scattered to locally consistent.

\subsection{Sheaf-Based Diagnostics}\label{subsec:diagnostics}

The sheaf embedding provides diagnostic tools that have no direct
feedforward counterpart. We illustrate two: spectral analysis of the
restricted Laplacian, and per-edge discord as a local consistency measure.

\paragraph{Spectral analysis.}
For a trained $[2,30,1]$ network on the paraboloid task, we compute the
spectrum of $L_{\mathcal{F}_t}[\Omega,\Omega]$ across 50 random inputs
at three training stages (early, intermediate, well-trained). The
eigenvalue spectra cluster tightly across inputs at each stage, indicating
that the Laplacian structure depends primarily on the weights rather than
on which particular neurons happen to be active. Eigenvector cosine
similarities exceed $0.98$ for the well-trained model, confirming this
input-independence quantitatively.

The spectral gap $\lambda_1$ stabilizes early in training while the loss
continues to decrease, indicating that the diffusion geometry is set
before fine-grained fitting occurs. Both sheaf-trained and SGD-trained
networks show spectra shifting toward smaller eigenvalues as training
progresses, with sheaf-trained networks consistently producing smaller
eigenvalues at comparable stages.

Across architectures of varying depth and width, the Fiedler eigenvector
consistently concentrates its energy on the output pre-activation
stalk~$z^{(k+1)}$, with concentration strongest for shallow networks and
decreasing as depth increases. Layer width is not the determining factor:
the pattern holds even when hidden layers are narrower than the output.
Additional spectral diagnostics, including condition number evolution and
$\lambda_{\max}$ tracking, appear in \cref{ap:ext-diagnostics}.

\paragraph{Per-edge discord.}
The discord $\|\delta \overline{\omega}\|^2$ decomposes naturally over edges, giving a
per-edge consistency score. Since the sheaf training dynamics do not
enforce exact consistency at every step, the converged state during
training may not perfectly reproduce the feedforward structure. One can
measure this gap directly: for each weight edge, plot the value that
the forward pass would produce,
$\overline{W}^{(\ell)}\overline{a}^{(\ell-1)}$, against the actual
stalk value $z^{(\ell)}$; for each ReLU edge, plot $z^{(\ell)}$ against
$a^{(\ell)}$. In a well-trained model at equilibrium, weight edge
residuals cluster near zero and ReLU edge residuals follow the
characteristic hinge curve, confirming local consistency. Poorly
trained models show scatter in both plots.

Discord localizes along edges in a structured way, governed by the
activation state of each pre-activation coordinate $z_j^{(\ell)}$.
When $z_j^{(\ell)} < 0$, the adjoint of the ReLU restriction map
$R^{z^{(\ell)}}$ has a zero diagonal entry at index $j$, so the
activation edge exerts no force on $z_j^{(\ell)}$. The weight edge
is then the sole force, and $z_j^{(\ell)}$ settles at the predicted
value $(\overline{W}^{(\ell)}\overline{a}^{(\ell-1)})_j$ with zero
weight edge discord. When $z_j^{(\ell)} \geq 0$, both edges pull on
$z_j^{(\ell)}$: the weight edge toward the predicted value and the
activation edge toward $a_j^{(\ell)}$, so at equilibrium
$z_j^{(\ell)}$ sits at the midpoint and both edges carry discord
proportional to the gap between predicted value and $a_j^{(\ell)}$.

Embedding an SGD-trained network as a sheaf, pinning the output stalk to
true labels, and running diffusion to convergence produces a per-edge
discord distribution that reveals how prediction error distributes
backward through the layers. Comparing this distribution against a
sheaf-trained network on the same task yields visibly different discord
profiles, consistent with the two methods navigating different regions of
the optimization landscape. Details appear in
\cref{ap:ext-diagnostics}.


\begin{figure}[ht]
  \centering
  \begin{subfigure}[t]{0.48\textwidth}
    \centering
    \includegraphics[width=\textwidth]{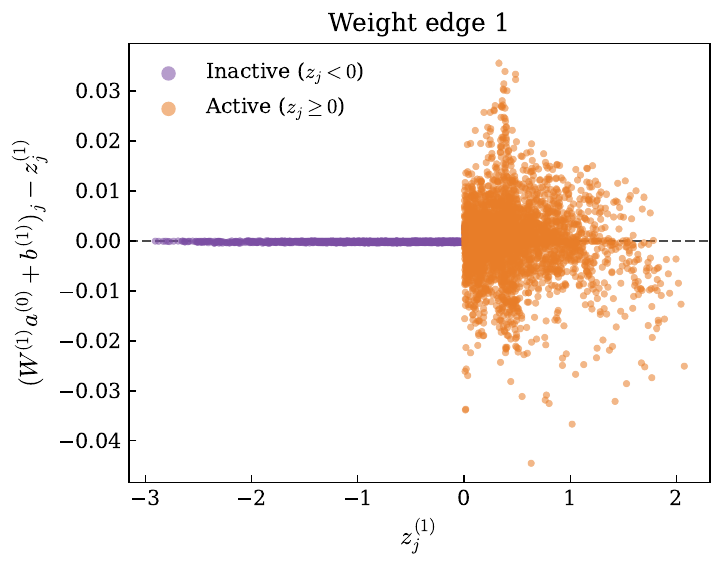}
    \caption{Weight edge discord.}
    \label{fig:discord-1H:weight}
  \end{subfigure}
  \hfill
  \begin{subfigure}[t]{0.48\textwidth}
    \centering
    \includegraphics[width=\textwidth]{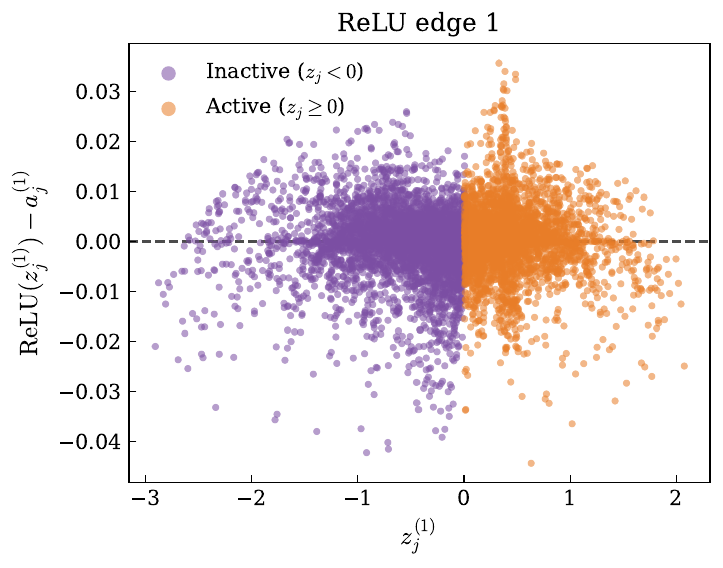}
    \caption{ReLU edge discord.}
    \label{fig:discord-1H:relu}
  \end{subfigure}

  \caption{Per-coordinate edge discord for a well-trained $[2,30,1]$ sheaf
    network on the paraboloid task ($300$ training inputs).
    \textbf{(a)}~Weight edge discord
    $(W^{(1)} a^{(0)} + b^{(1)})_j - z_j^{(1)}$ vs.\ pre-activation
    $z_j^{(1)}$.  Inactive coordinates ($z_j < 0$, purple) lie on the zero
    line: the ReLU edge exerts no competing force, so the weight edge drives
    $z_j$ to its predicted value exactly.  Active coordinates ($z_j \geq 0$,
    orange) show a spread around zero, reflecting the equilibrium compromise
    between weight and activation edge forces.
    \textbf{(b)}~ReLU edge discord $\mathrm{ReLU}(z_j^{(1)}) - a_j^{(1)}$ vs.\ $z_j^{(1)}$. For inactive coordinates ($z_j < 0$, purple), this quantity equals $-a_j^{(1)}$; the nonzero spread reflects the downstream weight edge pulling $a_j^{(1)}$ away from zero to reduce its own
discrepancy.  For active coordinates ($z_j \geq 0$, orange), both
edges compete, producing a symmetric cloud.  The magnitudes are
comparable to the weight-edge residuals, consistent with the
equilibrium force balance described in \cref{subsec:diagnostics}.
}  \label{fig:discord-1H}
\end{figure}



\section{Discussion}\label{sec:discussion}

We constructed a cellular sheaf whose unique harmonic extension
encodes the forward pass of a feedforward ReLU neural network.  The
mathematical core is the unitriangular factorization of the restricted
coboundary (\cref{lemma:det1}): each computational step in the forward
pass contributes a sub-diagonal block to a square matrix with identity
blocks on the diagonal, so the determinant is~$1$, the restricted
Laplacian is positive definite, and the harmonic extension is unique.
Forward substitution in the unitriangular system is the forward pass;
the sheaf heat equation converges exponentially to the same output
despite the state-dependent switching introduced by ReLU.

Several limitations should be stated plainly.
\begin{enumerate}[leftmargin=2em, itemsep=4pt]
  \item \emph{Path-graph dependence.}
    The unitriangular structure of $\delta_\Omega$ depends on a
    dimensional coincidence specific to the path graph: the total
    dimension of the edge stalks equals the total dimension of the free
    vertex coordinates, making $\delta_\Omega$ square.  Skip connections, as in residual architectures~\cite{he2016resnet},
add edges without adding free variables, producing a tall
    $\delta_\Omega$; the restricted Laplacian
    $\delta_\Omega^T\delta_\Omega$ remains positive definite (so
    harmonic extensions still exist and are unique), but the
    unit-determinant identity is lost.  More complex
    architectures -- encoder-decoder structures such as
    U-Net~\cite{ronneberger2015unet}, multi-modal couplings such as
    CLIP~\cite{radford2021clip} -- can be encoded as sheaves on richer
    graphs, and the heat equation distributes information
    bidirectionally on any such graph, but the convergence analysis
    requires checking positive definiteness case by case rather than
    deducing it from a universal algebraic identity.  Architectures
    with genuine cycles (recurrent networks) break the triangular
    ordering entirely.

  \item \emph{Genericity assumption.}
    The convergence theorems assume that the forward-pass equilibrium
    lies in the interior of an activation region
    (\cref{rem:zero-preactivation}).  For trained networks with dead
    neurons ($z_j^{(\ell)} = 0$ at the forward-pass solution), this
    assumption can fail, though the harmonic extension itself is
    unaffected -- only the labeling of the activation pattern is
    ambiguous, not the solution.

  \item \emph{Training performance.}
    Sheaf-based training, while principled, is not yet competitive with
    stochastic gradient descent on most of our benchmarks: at matched iteration counts (\cref{tab:full-results}) test loss ratios range from $0.2\times$ to $60\times$.  The computational cost
    of integrating the heat equation -- $10^5$ Euler steps, where SGD
    already achieves lower loss with $10^4$ epochs-- exceeds that of a backward pass.  The
    framework's present value is conceptual and diagnostic, not
    algorithmic.
\end{enumerate}

Within these limits, the sheaf perspective offers tools that the
feedforward view does not provide.  The per-edge discord decomposition
and the spectral analysis of the restricted Laplacian apply to any
trained network embedded post hoc, not only to sheaf-trained ones.
The Fiedler eigenvector's concentration on the output pre-activation
stalk, the early stabilization of the spectral gap during training,
and the active/inactive decomposition of weight-edge discord are
empirical observations with no feedforward counterpart.  

The partial clamping note (\cref{note:partial-clamping}) gives an
immediate operational mode: given any mixture of observed inputs,
desired outputs, and internal constraints on a pretrained network,
solve for a local-minimizer completion of the remaining
coordinates.  This naturally covers hidden-neuron steering,
missing-feature imputation, label-conditioned reconstruction, and
counterfactual editing of internal representations---all as instances
of harmonic extension with different boundary data on the same sheaf.
Whether the spectral and cohomological invariants of the neural sheaf
can serve as systematic tools for identifying structural bottlenecks
or failure modes in larger networks is a natural question that remains
open.

The locality of sheaf-based training -- each weight update depends only
on adjacent stalks -- connects to distributed optimization.  Gossip
algorithms~\cite{boyd2006gossip} achieve coherence through local
averaging, federated learning~\cite{mcmahan2017federated} partitions
training across clients, and the split consensus approach
of~\cite{camajori2024split} trains different layers through
peer-to-peer communication.  The sheaf framework encodes communication
structure, local objectives, and coordination mechanism in a single
mathematical object.  Whether this structural clarity can compensate
for the computational overhead of diffusion-based training at scale is
an open question.

Three specific open problems stand out.
\begin{enumerate}[leftmargin=2em, itemsep=4pt]
  \item \emph{Convergence of the joint CPWA dynamics.}
    When weights and cochains evolve simultaneously, the switching
    surfaces move with the parameters, and whether trajectories can
    become trapped in a noncritical sliding mode is unresolved.  The
    fixed-weight case (\cref{thm:convergence,thm:final}) and the affine
    joint dynamics~\cite{bosca2026selective} are both settled; the
    coupled nonlinear setting may require tools from CPWA
    theory~\cite{gould2025thesis,cortes2008discontinuous} beyond those
    used here.

  \item \emph{Extension beyond path graphs.}
    For architectures with skip connections or shared parameters, the
    restricted coboundary is no longer square, and the unit-determinant
    identity fails.  Positive definiteness of
    $\delta_\Omega^T\delta_\Omega$ -- equivalently, full column rank of
    $\delta_\Omega$ -- may still hold, but it must be verified for each
    architecture rather than following from a universal algebraic
    argument.  Understanding which graph topologies and restriction-map
    structures guarantee this property is a natural algebraic question.
    For example, Seely~\cite{seely2025sheaf} demonstrates that the sheaf
    framework extends naturally to recurrent predictive coding networks,
    where feedback loops create nonvanishing first cohomology and can
    cause learning to stall; adapting those cohomological diagnostics to
    nonlinear architectures is an open direction.

  \item \emph{Scaling.}
    The current implementation uses forward Euler on networks with at
    most 30 neurons per hidden layer.  Efficient integration exploiting
    the block-tridiagonal Laplacian structure, adaptive step sizes
    respecting ReLU boundary geometry, and parallelization leveraging
    the locality of the updates all need development.  The
    $\beta \sim 1/n_{\mathrm{train}}$ scaling from
    \cref{subsec:timescale} is a first quantitative step, but a full
    understanding of how diffusion-based training interacts with width
    and depth at scale remains to be developed.
\end{enumerate}

In this light, neural networks may be viewed less as layered circuits
than as geometric objects whose behavior is governed by local
consistency relations.  The sheaf framework makes these relations
explicit and places them within a broader topological setting,
suggesting that the mathematics of local-to-global phenomena may have
more to say about learning systems than we currently understand.





\bibliographystyle{abbrvnat}
\bibliography{references}  

\begin{thebibliography}{30}
\providecommand{\natexlab}[1]{#1}
\providecommand{\url}[1]{\texttt{#1}}
\expandafter\ifx\csname urlstyle\endcsname\relax
  \providecommand{\doi}[1]{doi: #1}\else
  \providecommand{\doi}{doi: \begingroup \urlstyle{rm}\Url}\fi

\bibitem[Barbero et~al.(2022)Barbero, Bodnar, S{\'a}ez~de Oc{\'a}riz~Borde, Bronstein, Veli{\v{c}}kovi{\'c}, and Li{\`o}]{barbero2022connection}
F.~Barbero, C.~Bodnar, H.~S{\'a}ez~de Oc{\'a}riz~Borde, M.~Bronstein, P.~Veli{\v{c}}kovi{\'c}, and P.~Li{\`o}.
\newblock Sheaf neural networks with connection laplacians.
\newblock In \emph{Proceedings of Topological, Algebraic, and Geometric Learning Workshops 2022}, volume 196 of \emph{Proceedings of Machine Learning Research}, pages 28--36, 2022.
\newblock URL \url{https://proceedings.mlr.press/v196/barbero22a.html}.

\bibitem[Bodnar et~al.(2022)Bodnar, Giovanni, Chamberlain, Lio, and Bronstein]{bodnar2022neural}
C.~Bodnar, F.~D. Giovanni, B.~P. Chamberlain, P.~Lio, and M.~M. Bronstein.
\newblock Neural sheaf diffusion: A topological perspective on heterophily and oversmoothing in {GNN}s.
\newblock In A.~H. Oh, A.~Agarwal, D.~Belgrave, and K.~Cho, editors, \emph{Advances in Neural Information Processing Systems}, 2022.
\newblock URL \url{https://openreview.net/forum?id=vbPsD-BhOZ}.

\bibitem[Bosca and Ghrist(2026)]{bosca2026selective}
V.~Bosca and R.~Ghrist.
\newblock Selective adaptation of beliefs and communication on cellular sheaves, 2026.
\newblock URL \url{https://arxiv.org/abs/2601.22431}.

\bibitem[Bosca et~al.(2025)Bosca, Rask, Tanweer, Tawfeek, and Stone]{bosca2025topological}
V.~Bosca, T.~Rask, S.~Tanweer, A.~R. Tawfeek, and B.~Stone.
\newblock Topological signatures of re{LU} neural network activation patterns.
\newblock In \emph{Topology, Algebra, and Geometry in Data Science}, 2025.
\newblock URL \url{https://openreview.net/forum?id=Q88w76j2Dd}.

\bibitem[Boyd et~al.(2006)Boyd, Ghosh, Prabhakar, and Shah]{boyd2006gossip}
S.~Boyd, A.~Ghosh, B.~Prabhakar, and D.~Shah.
\newblock Randomized gossip algorithms.
\newblock \emph{IEEE Transactions on Information Theory}, 52\penalty0 (6):\penalty0 2508--2530, 2006.
\newblock \doi{10.1109/TIT.2006.874516}.

\bibitem[Camajori~Tedeschini et~al.(2024)Camajori~Tedeschini, Brambilla, and Nicoli]{camajori2024split}
B.~Camajori~Tedeschini, M.~Brambilla, and M.~Nicoli.
\newblock Split consensus federated learning: An approach for distributed training and inference.
\newblock \emph{IEEE Access}, 12:\penalty0 119535–119549, 01 2024.
\newblock \doi{10.1109/ACCESS.2024.3446577}.

\bibitem[Caralt et~al.(2024)Caralt, Gil, Duta, Liò, and Cot]{hernandez2024joint}
F.~H. Caralt, G.~B. Gil, I.~Duta, P.~Liò, and E.~A. Cot.
\newblock Joint diffusion processes as an inductive bias in sheaf neural networks, 2024.
\newblock URL \url{https://arxiv.org/abs/2407.20597}.

\bibitem[Clarke(1983)]{clarke1983optimization}
F.~H. Clarke.
\newblock \emph{Optimization and Nonsmooth Analysis}.
\newblock Canadian Mathematical Society Series of Monographs and Advanced Texts. John Wiley \& Sons, New York, 1983.
\newblock ISBN 0-471-87504-X.

\bibitem[Cort{\'e}s(2008)]{cortes2008discontinuous}
J.~Cort{\'e}s.
\newblock Discontinuous dynamical systems: a tutorial on solutions, nonsmooth analysis, and stability.
\newblock \emph{IEEE Control Systems Magazine}, 28\penalty0 (3):\penalty0 36--73, 2008.
\newblock \doi{10.1109/MCS.2008.919306}.

\bibitem[Curry(2014)]{curry2014sheaves}
J.~M. Curry.
\newblock \emph{Sheaves, Cosheaves and Applications}.
\newblock PhD thesis, University of Pennsylvania, Philadelphia, PA, 2014.
\newblock URL \url{https://repository.upenn.edu/bitstreams/53914c60-81a8-4305-b8b6-693f2f42ea18/download}.
\newblock Ph.D. thesis; arXiv:1303.3255.

\bibitem[Filippov(1988)]{filippov1988differential}
A.~F. Filippov.
\newblock \emph{Differential Equations with Discontinuous Righthand Sides}.
\newblock Kluwer Academic Publishers, Dordrecht, 1988.

\bibitem[Gould(2025)]{gould2025thesis}
J.~J. Gould.
\newblock \emph{Cellular Sheaves of Hilbert Spaces}.
\newblock Ph.d. thesis, University of Pennsylvania, Philadelphia, PA, 2025.
\newblock URL \url{https://repository.upenn.edu/entities/publication/a40a44d9-f4e2-4c56-abcb-53936caeffad}.
\newblock Department of Mathematics.

\bibitem[Hansen and Gebhart(2020)]{hansen2020sheaf}
J.~Hansen and T.~Gebhart.
\newblock Sheaf neural networks.
\newblock \emph{CoRR}, abs/2012.06333, 2020.
\newblock URL \url{https://arxiv.org/abs/2012.06333}.

\bibitem[Hansen and Ghrist(2019)]{hansen2019toward}
J.~Hansen and R.~Ghrist.
\newblock Toward a spectral theory of cellular sheaves.
\newblock \emph{Journal of Applied and Computational Topology}, 3\penalty0 (4):\penalty0 315--358, 2019.
\newblock ISSN 2367-1734.
\newblock \doi{10.1007/s41468-019-00038-7}.
\newblock URL \url{https://doi.org/10.1007/s41468-019-00038-7}.

\bibitem[Hansen and Ghrist(2021)]{hansen2021opinion}
J.~Hansen and R.~Ghrist.
\newblock Opinion dynamics on discourse sheaves.
\newblock \emph{SIAM Journal on Applied Mathematics}, 81\penalty0 (5):\penalty0 2033--2060, 2021.
\newblock \doi{10.1137/20M1341088}.
\newblock URL \url{https://doi.org/10.1137/20M1341088}.

\bibitem[He et~al.(2015)He, Zhang, Ren, and Sun]{He2015Delving}
K.~He, X.~Zhang, S.~Ren, and J.~Sun.
\newblock Delving deep into rectifiers: Surpassing human-level performance on imagenet classification.
\newblock In \emph{Proceedings of the IEEE International Conference on Computer Vision (ICCV)}, pages 1026--1034, 2015.
\newblock \doi{10.1109/ICCV.2015.123}.

\bibitem[He et~al.(2016)He, Zhang, Ren, and Sun]{he2016resnet}
K.~He, X.~Zhang, S.~Ren, and J.~Sun.
\newblock Deep residual learning for image recognition.
\newblock In \emph{Proceedings of the IEEE Conference on Computer Vision and Pattern Recognition}, pages 770--778, 2016.
\newblock \doi{10.1109/CVPR.2016.90}.

\bibitem[Liu et~al.(2023)Liu, Cole, Peterson, and Kirby]{liu2023relu}
Y.~Liu, C.~M. Cole, C.~Peterson, and M.~Kirby.
\newblock Relu neural networks, polyhedral decompositions, and persistent homology, 2023.
\newblock URL \url{https://arxiv.org/abs/2306.17418}.

\bibitem[McMahan et~al.(2017)McMahan, Moore, Ramage, Hampson, and Arcas]{mcmahan2017federated}
B.~McMahan, E.~Moore, D.~Ramage, S.~Hampson, and B.~A.~y. Arcas.
\newblock {Communication-Efficient Learning of Deep Networks from Decentralized Data}.
\newblock In A.~Singh and J.~Zhu, editors, \emph{Proceedings of the 20th International Conference on Artificial Intelligence and Statistics}, volume~54 of \emph{Proceedings of Machine Learning Research}, pages 1273--1282. PMLR, 20--22 Apr 2017.
\newblock URL \url{https://proceedings.mlr.press/v54/mcmahan17a.html}.

\bibitem[Millidge et~al.(2022)Millidge, Tschantz, and Buckley]{millidge2022predictive}
B.~Millidge, A.~Tschantz, and C.~L. Buckley.
\newblock Predictive coding approximates backprop along arbitrary computation graphs.
\newblock \emph{Neural Computation}, 34\penalty0 (6):\penalty0 1329--1368, 2022.
\newblock \doi{10.1162/neco_a_01497}.

\bibitem[Mont{\'u}far et~al.(2014)Mont{\'u}far, Pascanu, Cho, and Bengio]{montufar2014linearregions}
G.~F. Mont{\'u}far, R.~Pascanu, K.~Cho, and Y.~Bengio.
\newblock On the number of linear regions of deep neural networks.
\newblock In \emph{Advances in Neural Information Processing Systems 27}, 2014.
\newblock URL \url{https://proceedings.neurips.cc/paper/2014/hash/fa6f2a469cc4d61a92d96e74617c3d2a-Abstract.html}.

\bibitem[Parada-Mayorga et~al.(2020)Parada-Mayorga, Riess, Ribeiro, and Ghrist]{parada2020quiver}
A.~Parada-Mayorga, H.~Riess, A.~Ribeiro, and R.~Ghrist.
\newblock Quiver signal processing (qsp), 2020.
\newblock URL \url{https://arxiv.org/abs/2010.11525}.

\bibitem[Radford et~al.(2021)Radford, Kim, Hallacy, Ramesh, Goh, Agarwal, Sastry, Askell, Mishkin, Clark, Krueger, and Sutskever]{radford2021clip}
A.~Radford, J.~W. Kim, C.~Hallacy, A.~Ramesh, G.~Goh, S.~Agarwal, G.~Sastry, A.~Askell, P.~Mishkin, J.~Clark, G.~Krueger, and I.~Sutskever.
\newblock Learning transferable visual models from natural language supervision, 2021.
\newblock URL \url{https://arxiv.org/abs/2103.00020}.

\bibitem[Ronneberger et~al.(2015)Ronneberger, Fischer, and Brox]{ronneberger2015unet}
O.~Ronneberger, P.~Fischer, and T.~Brox.
\newblock U-net: Convolutional networks for biomedical image segmentation.
\newblock In N.~Navab, J.~Hornegger, W.~M. Wells, and A.~F. Frangi, editors, \emph{Medical Image Computing and Computer-Assisted Intervention -- MICCAI 2015}, pages 234--241, Cham, 2015. Springer International Publishing.
\newblock ISBN 978-3-319-24574-4.

\bibitem[Scellier and Bengio(2017)]{scellier2017equilibrium}
B.~Scellier and Y.~Bengio.
\newblock Equilibrium propagation: Bridging the gap between energy-based models and backpropagation.
\newblock \emph{Frontiers in Computational Neuroscience}, 11:\penalty0 24, 2017.
\newblock \doi{10.3389/fncom.2017.00024}.

\bibitem[Seely(2025)]{seely2025sheaf}
J.~Seely.
\newblock Sheaf cohomology of linear predictive coding networks, 2025.
\newblock URL \url{https://arxiv.org/abs/2511.11092}.

\bibitem[Seigal et~al.(2023)Seigal, Harrington, and Nanda]{seigal2023principal}
A.~Seigal, H.~A. Harrington, and V.~Nanda.
\newblock Principal components along quiver representations.
\newblock \emph{Foundations of Computational Mathematics}, 23\penalty0 (4):\penalty0 1129--1165, 2023.
\newblock ISSN 1615-3383.
\newblock \doi{10.1007/s10208-022-09563-x}.
\newblock URL \url{https://doi.org/10.1007/s10208-022-09563-x}.

\bibitem[Shevitz and Paden(1994)]{shevitz1994lyapunov}
D.~Shevitz and B.~Paden.
\newblock Lyapunov stability theory of nonsmooth systems.
\newblock \emph{IEEE Transactions on Automatic Control}, 39\penalty0 (9):\penalty0 1910--1914, 1994.
\newblock \doi{10.1109/9.317122}.

\bibitem[Srivastava et~al.(2014)Srivastava, Hinton, Krizhevsky, Sutskever, and Salakhutdinov]{srivastava2014dropout}
N.~Srivastava, G.~Hinton, A.~Krizhevsky, I.~Sutskever, and R.~Salakhutdinov.
\newblock Dropout: a simple way to prevent neural networks from overfitting.
\newblock \emph{Journal of Machine Learning Research}, 15\penalty0 (56):\penalty0 1929--1958, 2014.

\bibitem[Whittington and Bogacz(2017)]{whittington2017predictive}
J.~C.~R. Whittington and R.~Bogacz.
\newblock An approximation of the error backpropagation algorithm in a predictive coding network with local hebbian synaptic plasticity.
\newblock \emph{Neural Computation}, 29\penalty0 (5):\penalty0 1229--1262, 2017.
\newblock \doi{10.1162/NECO_a_00949}.

\end{thebibliography}


\appendix

\section{Proofs}\label{ap:proofs}

This appendix collects the explicit block structure of the restricted
Laplacian $L_{\mathcal{F}_t}[\Omega,\Omega]$ and the convergence proof for networks with nonlinear
final activations \cref{thm:final}.

\subsection{Block structure of the restricted Laplacian}
\label{ap:block-structure}

Recall the restricted dynamics~\eqref{eq:neural-restricted}. We derive the
block structure of $L_{\mathcal{F}_t}[\Omega,\Omega]$ from the
construction in \cref{subsec:networks-to-sheaves}. For the identity output
case ($\phi = I$), the output edge enforces $\hat{y} = z^{(k+1)}$ at
equilibrium, so we work with the free coordinates
\begin{equation}\label{eq:free-coords}
  \omega
    = \bigl(z^{(1)},\; a^{(1)},\; z^{(2)},\; a^{(2)},\;
      \ldots,\; a^{(k)},\; z^{(k+1)}\bigr).
\end{equation}
The fixed boundary data consists of the input
$\overline{\mathbf{x}} \in \R^{n_0 + n_1}$ at $v_x$ together with the ones
blocks $\mathbf{1}_{n_{\ell+1}}$ held constant within each extended
activation stalk at $v_{a^{(\ell)}}$ for $\ell = 1, \ldots, k$.

Each weight edge $e_{z^{(\ell)}}$ couples
$\overline{a}^{(\ell-1)} = (a^{(\ell-1)T},\, \mathbf{1}_{n_\ell}^T)^T$
to $z^{(\ell)}$ through the extended weight matrix
$\overline{W}^{(\ell)} = (W^{(\ell)} \mid B^{(\ell)})$. Since the ones
block is fixed, only $W^{(\ell)}$ appears in the free-free coupling
$L[\Omega,\Omega]$; the bias contribution
$B^{(\ell)}\mathbf{1}_{n_\ell} = b^{(\ell)}$ enters through $L[\Omega,U]u$.

\medskip
\noindent\textbf{Single hidden layer.}
For $k = 1$, the restricted Laplacian on the free coordinates
$\omega = (z^{(1)}, a^{(1)}, z^{(2)})$ is
\begin{equation}\label{eq:lap-1layer}
  L_{\mathcal{F}_t}[\Omega,\Omega]
  = \begin{pmatrix}
      I_{n_1} + R^{z^{(1)}}
        & -R^{z^{(1)}}
        &  \\[4pt]
      -R^{z^{(1)}}
        & I_{n_1} + {W^{(2)}}^T W^{(2)}
        & -{W^{(2)}}^T \\[4pt]
      
        & -W^{(2)}
        & I_{n_2} 
    \end{pmatrix}.
\end{equation}
The corresponding coordinate dynamics are
\begin{equation}\label{eq:coord-dynamics-1layer}
  \frac{d}{dt}\begin{pmatrix}
    z^{(1)} \\[3pt] a^{(1)} \\[3pt] z^{(2)} 
  \end{pmatrix}
  = -\alpha
  \begin{pmatrix}
    z^{(1)} - \overline{W}^{(1)}\overline{\mathbf{x}}
      + R^{z^{(1)}}(z^{(1)} - a^{(1)}) \\[3pt]
    a^{(1)} - R^{z^{(1)}} z^{(1)}
      + {W^{(2)}}^T(\overline{W}^{(2)}\overline{a}^{(1)} - z^{(2)}) \\[3pt]
    z^{(2)} - \overline{W}^{(2)}\overline{a}^{(1)} 
  \end{pmatrix},
\end{equation}
where the terms $-\overline{W}^{(1)}\overline{\mathbf{x}}$,
${W^{(2)}}^T b^{(2)}$, and $-b^{(2)}$ are the boundary contributions from
$L_{\mathcal{F}_t}[\Omega,U]u$.

\medskip
\noindent\textbf{General $k$-layer network.}
Group the free coordinates into layer blocks
$(z^{(\ell)}, a^{(\ell)})$ for $\ell = 1, \ldots, k$, with a final block
$z^{(k+1)}$. Define the block components for
$\ell \in \{1, \ldots, k\}$:
\begin{equation}\label{eq:block-defs}
\begin{aligned}
  A_\ell &\coloneqq \begin{pmatrix}
    I_{n_\ell} + R^{z^{(\ell)}}
      & -R^{z^{(\ell)}} \\[3pt]
    -R^{z^{(\ell)}}
      & I_{n_\ell} + {W^{(\ell+1)}}^T W^{(\ell+1)}
  \end{pmatrix}
  \in \R^{2n_\ell \times 2n_\ell}, \\[6pt]
  C_\ell &\coloneqq \begin{pmatrix}
    0 & -W^{(\ell)} \\[3pt]
    0 & 0
  \end{pmatrix}, \qquad
  D_\ell \coloneqq \begin{pmatrix}
    0 & -W^{(\ell)}
  \end{pmatrix}.
\end{aligned}
\end{equation}
The diagonal block $A_\ell$ captures the activation edge (coupling
$z^{(\ell)}$ and $a^{(\ell)}$) and the weight edge to the next layer
(contributing ${W^{(\ell+1)}}^T W^{(\ell+1)}$ to the $a^{(\ell)}$ diagonal).
The off-diagonal blocks $C_\ell$ and $D_\ell$ couple consecutive layers
through weight edges. The restricted Laplacian has the block-tridiagonal form
\begin{equation}\label{eq:lap-klayer}
  L_{\mathcal{F}_t}[\Omega,\Omega]
  = \begin{pmatrix}
      A_1       & C_2^T       &             &             &                 \\
      C_2       & A_2         & C_3^T       &             &                 \\
                & C_3         & \ddots      & \ddots      &                 \\
                &             & \ddots      & A_k         & D_{k+1}^T       \\
                &             &             & D_{k+1}     & I_{n_{k+1}} 
    \end{pmatrix}.
\end{equation}

\subsection{Convergence with final activation
  (\texorpdfstring{\cref{thm:final}}{Theorem~\ref*{thm:final}})}
\label{ap:convergence-final}

We restate the theorem for convenience.

\noindent\textbf{\cref{thm:final}}
(\textit{Convergence with final activation}).
\textit{Consider a $k$-hidden-layer ReLU network with $C^1$ final activation
$\phi\colon \mathbb{R}^{n_{k+1}} \to \mathbb{R}^{n_{k+1}}$ that is either
bounded or has at most linear growth.  Assume the forward-pass equilibrium
$\omega^*$ lies in the interior of an activation region (a generic
condition; see \cref{rem:zero-preactivation}). For any initial condition
$(\omega_0, \hat{y}_0)$, the solution of the restricted
dynamics~\eqref{eq:neural-restricted} augmented
by~\eqref{eq:final-dynamics} converges to the unique fixed point
$(\omega^*, \phi(z^{(k+1)*}))$ encoding the standard neural network output.}

\begin{proof}
Recall that $L_{\mathcal{F}_i}[\Omega,\Omega]$ is positive definite for every
activation pattern by \cref{lemma:det1}. Write
$x = (\omega, \hat{y})$, where $\omega$ collects all free coordinates
up to the final preactivation $z^{(k+1)}$ (included), and $\hat{y}$ is the output
vertex state. Define the energy
\begin{equation}\label{eq:lyapunov-final}
  E(\omega, \hat{y})
    \;=\; \tfrac{1}{2}\,\omega^T
      L_{\mathcal{F}_i}[\Omega,\Omega]\,\omega
      + \omega^T L_{\mathcal{F}_i}[\Omega,U]\,u
    \;+\; \tfrac{1}{2}\bigl\|\phi(z^{(k+1)}) - \hat{y}\bigr\|^2.
\end{equation}
Each term is continuous and piecewise $C^1$, so $E$ is continuous and locally
Lipschitz.  The positive-definite quadratic term
$\frac{1}{2}\omega^T L_{\mathcal{F}_i}[\Omega,\Omega]\,\omega$ dominates at
large $\|\omega\|$, and the growth hypothesis on~$\phi$ ensures the output
potential does not destroy this dominance, so $E$ is coercive and sublevel
sets are bounded.

The right-hand side
of~\eqref{eq:neural-restricted}--\eqref{eq:final-dynamics} equals
$-\alpha\,\nabla E$ in the interior of every ReLU region. At switching
surfaces, the Filippov set equals the convex hull of the limiting gradient
vectors, so the dynamics satisfy
$\frac{d}{dt} x(t) \in -\alpha\,\partial_C E(x(t))$,
where $\partial_C E$ denotes the Clarke
subdifferential~\cite{clarke1983optimization}. Existence of absolutely
continuous Filippov solutions follows from standard
theory~\cite{filippov1988differential,cortes2008discontinuous}, since the
right-hand side is measurable and locally bounded.

In the interior of any activation region, the chain rule gives
$\frac{d}{dt}E(x(t)) = -\alpha\,\|\nabla E(x(t))\|^2 \leq 0$.
At a switching surface $\{z_j^{(\ell)} = 0\}$, the activation edge
contributes $\frac{1}{2}(R_{jj}^{z^{(\ell)}}\, z_j^{(\ell)} -
a_j^{(\ell)})^2$ to the energy. Since $\operatorname{ReLU}(0) = 0$
regardless of convention, this evaluates to
$\frac{1}{2}(a_j^{(\ell)})^2$ from both sides, so $E^+ = E^-$ on the
surface. Both $E^+$ and $E^-$ are $C^1$ functions that coincide on
$\{z_j^{(\ell)} = 0\}$, so differentiating along any tangent direction $\tau$
gives $\nabla_\tau E^+ = \nabla_\tau E^-$; only the normal derivatives
differ. Any Filippov sliding velocity $v_s$ is tangent to the surface,
and its tangential component equals $-\alpha\,\nabla_\tau E$ (since
both one-sided velocities $v^\pm = -\alpha\,\nabla E^\pm$ share this
tangential part). Hence $v_s = -\alpha\,\nabla_\tau E$ and
\[
  \frac{d}{dt}E(x(t))
    = \bigl\langle \nabla_\tau E,\, v_s \bigr\rangle
    = -\alpha\,\bigl\|\nabla_\tau E\bigr\|^2
    \leq 0.
\]
Combining both cases, $E \circ x$ is nonincreasing. Coercivity then implies
that trajectories remain bounded, so the $\omega$-limit set of any solution
is nonempty, compact, and invariant. By LaSalle's invariance
principle~\cite{filippov1988differential,shevitz1994lyapunov}, every
trajectory converges to the largest invariant subset of
$\{x \colon 0 \in \partial_C E(x)\}$.

In the interior of any activation region, $\nabla E = 0$ gives
$\hat{y} = \phi(z^{(k+1)})$ and
$L_{\mathcal{F}_i}[\Omega,\Omega]\,\omega
+ L_{\mathcal{F}_i}[\Omega,U]\,u = 0$.
Positive definiteness of the restricted Laplacian and self-consistency
identify $\omega = \omega^*$ as the unique interior critical point, as in
\cref{thm:convergence}.

It remains to exclude Clarke critical points on switching surfaces. Consider
the surface $\{z_j^{(\ell)} = 0\}$ and suppose $\nabla_\tau E = 0$ there.
The stationarity equations cascade from the output: $\partial E/\partial
\hat{y} = 0$ gives $\hat{y} = \phi(z^{(k+1)})$, then $\partial
E/\partial z^{(k+1)} = 0$ forces the weight-edge discord at layer $k+1$ to
vanish, and propagating layer by layer, the weight-edge discord vanishes at
every layer above~$\ell$. The stationarity equation for $a_j^{(\ell)}$ then
reduces to $a_j^{(\ell)} = 0$ (the activation edge is the only remaining
contribution, since $z_j^{(\ell)} = 0$ and the discord from layer $\ell+1$
has been eliminated). With $a_j^{(\ell)} = 0$, the activation edge
contributes $R_{jj}^{z^{(\ell)}}(R_{jj}^{z^{(\ell)}} \cdot 0 - 0) = 0$ to both
one-sided normal derivatives, so $\partial_{z_j} E^+ = \partial_{z_j} E^- =
-t_j$, where $t_j = (\overline{W}^{(\ell)}\overline{a}^{(\ell-1)})_j$. Since both one-sided normal derivatives equal $-t_j$, the Clarke condition
requires $t_j = 0$. If $t_j = 0$, the cascade extends through all layers:
the discord vanishes at every edge, making the cochain a global section of
the sheaf. Positive definiteness of the restricted Laplacian implies the
unique global section consistent with the pinned input is $\omega^*$. But
$z_j^{(\ell)*} \neq 0$ by general position, contradicting
$z_j^{(\ell)} = 0$.

Since $\omega^*$ is the unique element of the LaSalle invariant set,
all trajectories converge to $(\omega^*, \phi(z^{(k+1)*}))$.
\end{proof}


\section{Supplementary Diagrams and Formulas}\label{ap:diagrams}

This appendix develops three topics deferred from \cref{sssec:joint-dynamics,sssec:loss-functions,sssec:regularization}: the parameter sheaf that governs weight dynamics, the sheaf-theoretic realization of regularization, and the explicit edge potential formulas for losses in the discrepancy framework.

\subsection{The parameter sheaf and coupled dynamics}
\label{ap:structure-sheaf}

The joint dynamics~\eqref{eq:joint-dynamics} describe coupled evolution of the 0-cochain $\overline{\omega}$ and the coboundary operator $\delta$. The cochain equation is sheaf diffusion on the \emph{state sheaf} $\mathcal{F}^{\Omega}$ with restriction maps determined by the current weights. The weight equation has an analogous interpretation: it is sheaf diffusion on an auxiliary \emph{parameter sheaf} $\mathcal{H}^{W}$, introduced in~\cite{bosca2026selective} for opinion dynamics and independently by Gould~\cite{gould2025thesis} under the name \emph{restriction map diffusion}. The same duality is implicit in the ``dual diffusion'' formulation of Hernandez Caralt et al.~\cite[Eq.~(8), Fig.~2]{hernandez2024joint}, who observe that the learning-to-lie ODE can be rewritten as sheaf diffusion with features playing the role of restriction maps.

We specialize the construction of~\cite[Def.~4.5]{bosca2026selective} to the neural setting. Fix a 0-cochain $\overline{\omega}$ on the neural sheaf $\mathcal{F}$. The parameter sheaf $\mathcal{H}^{W}$ is defined on the same graph as $\mathcal{F}$, with the following data.

The vertex stalk at each vertex $v$ is the direct sum of spaces of candidate restriction maps from $v$:
\begin{equation}\label{eq:structure-stalk}
  \mathcal{H}^{W}(v)
    = \bigoplus_{e \,:\, v \trianglelefteqslant e}
      \mathrm{Hom}\bigl(\mathcal{F}(v),\, \mathcal{F}(e)\bigr).
\end{equation}
A 0-cochain $\rho \in C^0(\mathcal{H}^{W})$ assigns to each vertex a tuple of matrices, one per incident edge. In the neural network, only the weight matrices $\overline{W}^{(\ell)}$ are trainable; the remaining restriction maps (identity, ReLU diagonal, projection) are fixed. The trainable-map projection $\Pi_W$ restricts attention to the weight components.

The edge stalk at each edge $e$ coincides with the edge stalk of $\mathcal{F}$ on weight edges, and it is trivial on activation and output edges. The restriction maps of $\mathcal{H}^{W}$ are \emph{evaluation maps} at the current cochain: for each incidence $v \trianglelefteqslant e$,
\begin{equation}\label{eq:structure-restriction}
  (\mathcal{H}^{W})_{v \trianglelefteqslant e}(\rho_v)
    = \rho_{v \trianglelefteqslant e}(\overline{\omega}_v),
\end{equation}
where $\rho_{v \trianglelefteqslant e}$ is the $e$-component of $\rho_v$ and $\overline{\omega}_v$ is the current extended activation at vertex $v$. The coboundary of $\mathcal{H}^{W}$ measures \emph{structural discrepancy}: $(\delta^{\mathcal{H}} \rho)_e = 0$ precisely when the candidate maps encoded in $\rho$ make the fixed cochain $\overline{\omega}$ locally consistent on edge $e$. Global sections of $\mathcal{H}^{W}$ are sheaf structures under which $\overline{\omega}$ is a global section of $\mathcal{F}$.

The Laplacian of $\mathcal{H}^{W}$ governs weight updates. Diffusion on $\mathcal{H}^{W}$, projected onto the trainable components via $\Pi_W$, recovers the second line of~\eqref{eq:joint-dynamics}:
\begin{equation}\label{eq:structure-diffusion}
  \frac{d\delta}{dt}
    = -\beta\, \Pi_W\bigl(L_{\mathcal{H}^{W}}\, \rho\bigr)
    = -\beta\, \Pi_W\bigl(\delta\, \overline{\omega}\, \overline{\omega}^T\bigr).
\end{equation}
The two sheaves are coupled: $\mathcal{F}^{\Omega}$ depends on the current weights through its restriction maps, while $\mathcal{H}^{W}$ depends on the current activations through its evaluation maps. Training is simultaneous diffusion on both sheaves, each using the other's current state to define its own Laplacian. \Cref{fig:coupled-sheaves} illustrates this coupling for a 1-hidden layer network.

\begin{figure}[htbp]
\centering

\begin{tikzpicture}[
    vertex/.style={circle, fill=white, draw=black, thick, minimum size=5pt, inner sep=0pt},
    edge stalk h/.style={rectangle, rounded corners=1pt, fill=white, draw=black, minimum width=16pt, minimum height=4pt, inner sep=0pt},
    edge stalk v/.style={rectangle, rounded corners=1pt, fill=white, draw=black, minimum width=4pt, minimum height=16pt, inner sep=0pt},
    arrow/.style={-{Stealth[length=4pt]}, thick},
    green arrow/.style={-{Stealth[length=4pt]}, thick, color=DarkGreen},
    link to/.style={-{Stealth[length=3pt]}, color=DarkBlue, thin, densely dashed},
    link plain/.style={color=DarkBlue, thin, densely dashed},
    every node/.style={font=\small},
]

\begin{scope}[shift={(0.7,1.2)}]

    \node[vertex, draw=DarkRed] (Tvx) at (0, 2.8) {};
    \node[vertex, draw=DarkGreen] (Tvz1) at (0, 0) {};
    \node[vertex, draw=DarkOrange1] (Tva1) at (2.8, 0) {};
    \node[vertex, draw=DarkGreen] (Tvz2) at (5.4, 0) {};
    \node[vertex, draw=DarkRed] (Tvy) at (5.4, 2.8) {};

    \node[above left] at (Tvx) {$\overline{\mathbf{x}}$};
    \node[below left, color=DarkGreen] (Tlz1) at (Tvz1) {$z^{\scriptscriptstyle(1)}$};
    \node[below=2pt, color=DarkGreen] (Tla1) at (Tva1) {$\overline{a}^{\scriptscriptstyle(1)}$};
    \node[below right, color=DarkGreen] (Tlz2) at (Tvz2) {$z^{\scriptscriptstyle(2)}$};
    \node[above right] at (Tvy) {$\mathbf{y}$};

    \draw[thick] (Tvx) -- (Tvz1);
    \draw[thick] (Tvz1) -- (Tva1);
    \draw[thick] (Tva1) -- (Tvz2);
    \draw[thick] (Tvz2) -- (Tvy);

    \node[edge stalk v] (Tes1) at (0, 1.4) {};
    \node[edge stalk h] (Tes2) at (1.4, 0) {};
    \node[edge stalk h] (Tes3) at (4.1, 0) {};
    \node[edge stalk v] (Tes4) at (5.4, 1.4) {};

    \draw[green arrow] (-0.3, 2.65) -- (-0.3, 1.55);
    \draw[arrow] (-0.3, 0.15) -- (-0.3, 1.25);
    \node[left, font=\scriptsize, color=DarkGreen] (TlW1) at (-0.35, 2.1) {$\overline{W}^{\scriptscriptstyle(1)}$};
    \node[left, font=\scriptsize] at (-0.35, 0.7) {$I_{n_1}$};

    \draw[arrow] (0.15, -0.4) -- (1.25, -0.4);
    \draw[arrow] (2.4, -0.4) -- (1.55, -0.4);
    \node[below, font=\scriptsize] at (0.65, -0.45) {$R^{z^{(1)}}$};
    \node[below, font=\scriptsize] at (1.95, -0.45) {$P_{n_1}$};

    \draw[green arrow] (3.1, -0.4) -- (4.0, -0.4);
    \draw[arrow] (5.25, -0.4) -- (4.2, -0.4);
    \node[below, font=\scriptsize, color=DarkGreen] (TlW2) at (3.6, -0.45) {$\overline{W}^{\scriptscriptstyle(2)}$};
    \node[below, font=\scriptsize] at (4.8, -0.45) {$I_{n_2}$};

    \draw[arrow] (5.7, 2.65) -- (5.7, 1.55);
    \node[right, font=\scriptsize] at (5.75, 2.1) {$I_{n_2}$};
    \draw[arrow] (5.7, 0.15) -- (5.7, 1.25);
    \node[right, font=\scriptsize] at (5.75, 0.7) {$\phi$};

\end{scope}

\node[anchor=west, font=\normalsize] at (8.5, 4.0) {State sheaf $\mathcal{F}^{\Omega}$};
\node[anchor=west, font=\small] at (8.5, 2.5) {$\displaystyle\frac{d\omega}{dt}
    = -\alpha \, P_{\Omega}
      \bigl(\delta^T \delta\, \overline{\omega}\bigr),$};

\begin{scope}[shift={(0,-5.0)}]

    \node[vertex, draw=DarkGreen] (BvW1) at (0, 3.2) {};
    \node[vertex, draw=DarkRed] (BvIR) at (0, 0) {};
    \node[vertex, draw=DarkOrange1] (BvPW2) at (3.5, 0) {};
    \node[vertex, draw=DarkRed] (BvII) at (6.8, 0) {};
    \node[vertex, draw=DarkRed] (BvI) at (6.8, 3.2) {};

    \node[above left, color=DarkGreen] (BlW1) at (BvW1) {$\overline{W}^{\scriptscriptstyle(1)}$};
    \node[below left, font=\scriptsize] at (BvIR) {$(I_{n_1} \;\; R^{z^{(1)}})$};
    \node[below=1pt, font=\scriptsize] (BlPW2) at (BvPW2) {$(P_{n_1} \;\; \textcolor{DarkGreen}{\overline{W}^{\scriptscriptstyle(2)}})$};
    \node[below right, font=\scriptsize] at (BvII) {$(I_{n_2} \;\; \phi)$};
    \node[above right] at (BvI) {$I_{n_2}$};

    \draw[thick] (BvW1) -- (BvIR);
    \draw[thick] (BvIR) -- (BvPW2);
    \draw[thick] (BvPW2) -- (BvII);
    \draw[thick] (BvII) -- (BvI);

    \node[edge stalk v] (Bes1) at (0, 1.6) {};
    \node[edge stalk h] (Bes2) at (1.75, 0) {};
    \node[edge stalk h] (Bes3) at (5.15, 0) {};
    \node[edge stalk v] (Bes4) at (6.8, 1.6) {};

    \draw[arrow] (-0.3, 3.05) -- (-0.3, 1.75);
    \node[left, font=\scriptsize] at (-0.35, 2.4) {$\overline{\mathbf{x}}$};
    \draw[green arrow] (-0.3, 0.15) -- (-0.3, 1.45);
    \node[left, font=\tiny] (Blz1e1) at (-0.35, 0.8) {$\begin{bmatrix} \textcolor{DarkGreen}{z^{\scriptscriptstyle(1)}} \\ 0 \end{bmatrix}$};

    \draw[green arrow] (0.15, -0.4) -- (1.55, -0.4);
    \node[below, font=\tiny] (Blz1e2) at (0.75, -0.45) {$\begin{bmatrix} 0 \\ \textcolor{DarkGreen}{z^{\scriptscriptstyle(1)}} \end{bmatrix}$};
    \draw[green arrow] (2.65, -0.4) -- (1.95, -0.4);
    \node[below, font=\tiny] (Bla1e2) at (2.4, -0.45) {$\begin{bmatrix} \textcolor{DarkGreen}{\overline{a}^{\scriptscriptstyle(1)}}  \\ 0 \end{bmatrix}$};

    \draw[green arrow] (4.35, -0.4) -- (4.95, -0.4);
    \node[below, font=\tiny] (Bla1e3) at (4.5, -0.45) {$\begin{bmatrix} 0 \\ \textcolor{DarkGreen}{\overline{a}^{\scriptscriptstyle(1)}}  \end{bmatrix}$};
    \draw[green arrow] (6.65, -0.4) -- (5.35, -0.4);
    \node[below, font=\tiny] (Blz2e3) at (6.1, -0.45) {$\begin{bmatrix} \textcolor{DarkGreen}{z^{\scriptscriptstyle(2)}} \\ 0 \end{bmatrix}$};

    \draw[arrow] (7.1, 3.05) -- (7.1, 1.75);
    \node[right, font=\scriptsize] at (7.15, 2.4) {$\mathbf{y}$};
    \draw[green arrow] (7.1, 0.15) -- (7.1, 1.45);
    \node[right, font=\tiny] (Blz2e4) at (7.15, 0.8) {$\begin{bmatrix} 0 \\ \textcolor{DarkGreen}{z^{\scriptscriptstyle(2)}} \end{bmatrix}$};

\end{scope}

\node[anchor=west, font=\normalsize] at (8.5, -1.8) {Parameter sheaf $\mathcal{H}^{W}$};
\node[anchor=west, font=\small] at (8.5, -3.3) {$\displaystyle
\frac{d\delta}{dt} =  -\beta\, \Pi_W\bigl(\delta\, \overline{\omega}\, \overline{\omega}^T\bigr)$};



\draw[link to]
    (BlW1.north) to[out=100, in=-110] (TlW1.south);

\coordinate (z1fork) at (0.45, -3.2);
\draw[link plain] (Tlz1.south) -- (z1fork);
\draw[link to] (z1fork) to[out=-150, in=70] (Blz1e1.north);
\draw[link to] (z1fork) to[out=-40, in=120] (Blz1e2.north);

\draw[link to]
    ($(BlPW2.north)!0.6!(BlPW2.north east)$) to[out=80, in=-80] (TlW2.south);

\coordinate (a1fork) at (3.5, -3.5);
\draw[link plain] (Tla1.south) -- (a1fork);
\draw[link to] (a1fork) to[out=-130, in=90] (Bla1e2.north);
\draw[link to] (a1fork) to[out=-50, in=90] (Bla1e3.north);

\coordinate (z2fork) at (6.35, -3.2);
\draw[link plain] (Tlz2.south) -- (z2fork);
\draw[link to] (z2fork) to[out=-140, in=60] (Blz2e3.north);
\draw[link to] (z2fork) to[out=-30, in=110] (Blz2e4.north);

\end{tikzpicture}

\caption{Joint dynamics on coupled sheaves. The upper diagram depicts the state sheaf $\mathcal{F}^{\Omega}$, where vertices represent network layers (pinned vertices $\overline{\mathbf{x}}$ and $\mathbf{y}$ in red, hidden states in green) and edge stalks carry restriction maps built from weights and activations. The lower diagram shows the parameters sheaf $\mathcal{H}^{W}$, where vertices store weight matrices and edge stalks encode gradients with respect to activations. Dashed blue arrows link state variables (pointing downward) and parameters (pointing upward) between the two sheaves, illustrating how the heat diffusion on both sheaves is coupled.}
\label{fig:coupled-sheaves}
\end{figure}
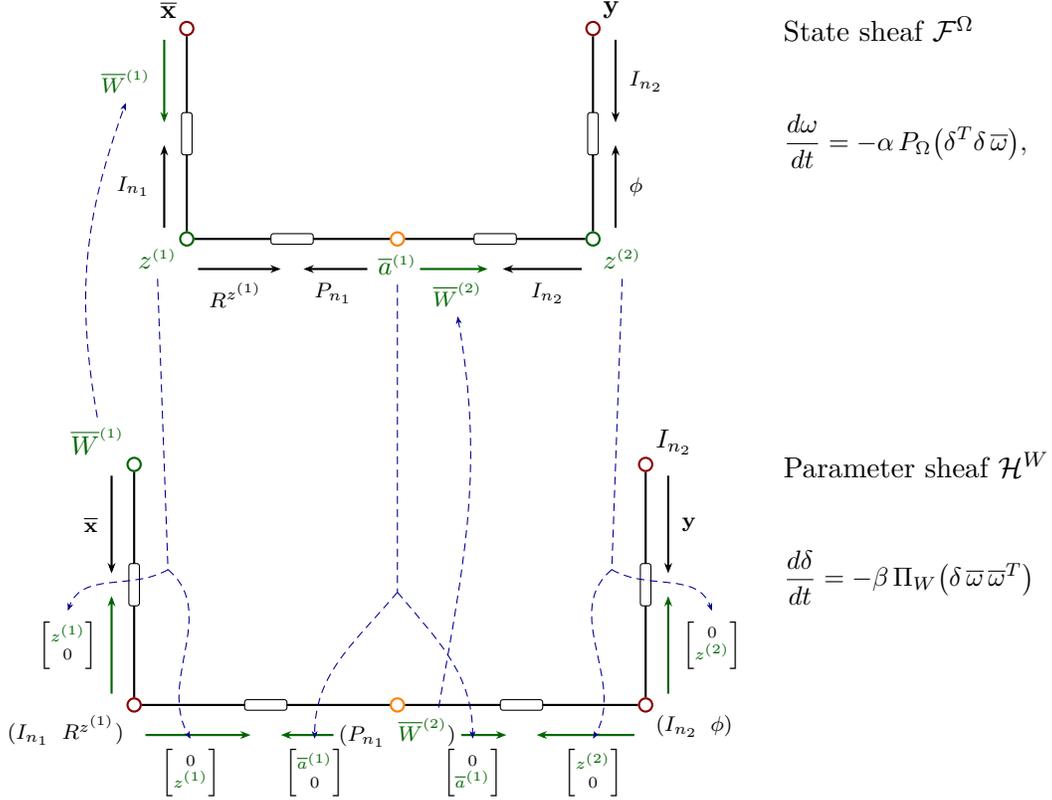

\subsection{Regularization as sheaf augmentation}\label{ap:regularization}

The regularized energy~\eqref{eq:regularized-energy-training} penalizes deviation from zero: $(\lambda/2)\|\omega\|^2 + (\mu/2)\|\delta\|_F^2$. The general form anchors to an arbitrary reference point $(\omega_0, \delta_0)$, typically the initial condition:
\begin{equation}\label{eq:regularized-general}
  \mathcal{L}_{\lambda,\mu}(\overline{\omega}, \delta)
    = \tfrac{1}{2}\|\delta\,\overline{\omega}\|^2
      + \tfrac{\lambda}{2}\|\omega - \omega_0\|^2
      + \tfrac{\mu}{2}\|\delta - \delta_0\|_F^2.
\end{equation}
The corresponding gradient descent dynamics are
\begin{equation}\label{eq:regularized-general-dynamics}
\begin{aligned}
  \frac{d\omega}{dt}
    &= -\alpha  P_\Omega\Bigl(
     \delta^T \delta\,\overline{\omega}
        + \lambda\,(\overline{\omega} - \overline{\omega}_0) \Bigr), \\[4pt]
  \frac{d\delta}{dt}
    &= -\beta\, \Pi_W \Bigl(
      \delta\, \overline{\omega}\, \overline{\omega}^T
        + \mu\, (\delta - \delta_0) \Bigr).
\end{aligned}
\end{equation}

Both regularization terms admit sheaf-theoretic interpretations as \emph{weighted reluctance}~\cite{hansen2021opinion,bosca2026selective}. The idea is to augment the graph with additional edges that penalize departure from a reference value, converting a penalty term in the energy into an ordinary discrepancy term on an enlarged sheaf.

For the cochain term, add a \emph{reluctance edge} from each dynamic vertex $v$ to a new \emph{anchor vertex} $v_0$ holding the reference value $\omega_{0,v}$. The restriction maps on this edge are both $\sqrt{\lambda}\, I$, so the edge discrepancy is $\sqrt{\lambda}(\omega_v - \omega_{0,v})$ and the edge potential is $(\lambda/2)\|\omega_v - \omega_{0,v}\|^2$. Summing over all dynamic vertices recovers the cochain penalty.

For the weight term, an analogous construction operates on the parameter sheaf $\mathcal{H}^{W}$. A reluctance edge connects each weight vertex to an anchor holding $\delta_0$, with restriction maps $\sqrt{\mu}\, I$. Diffusion on the augmented parameter sheaf produces the $\mu(\delta - \delta_0)$ term in the weight dynamics.

The augmented graph thus has the same topology as the original neural graph, plus one reluctance edge per dynamic vertex (for cochain regularization) and one per weight edge (for structure regularization). On this enlarged graph, the regularized dynamics~\eqref{eq:regularized-general-dynamics} are unregularized joint dynamics: gradient descent on the total discrepancy of the augmented sheaf, with no additional penalty terms. This is the diagrammatic realization referred to in \cref{sssec:regularization}.

\begin{figure}[ht]
\centering
\begin{subfigure}[b]{0.43\textwidth}
\centering
\raisebox{10pt}{%
\resizebox{\linewidth}{!}{%
\begin{tikzpicture}[
    vertex/.style={circle, fill=white, draw=black, thick, minimum size=5pt, inner sep=0pt},
    edge stalk h/.style={rectangle, rounded corners=1pt, fill=white, draw=black, minimum width=16pt, minimum height=4pt, inner sep=0pt},
    edge stalk v/.style={rectangle, rounded corners=1pt, fill=white, draw=black, minimum width=4pt, minimum height=16pt, inner sep=0pt},
    arrow/.style={-{Stealth[length=4pt]}, thick},
    green arrow/.style={-{Stealth[length=4pt]}, thick, color=DarkGreen},
    every node/.style={font=\small},
]

\node[vertex, draw=DarkRed] (vx) at (0, 2.8) {};
\node[vertex, draw=DarkGreen] (vz1) at (0, 0) {};
\node[vertex, draw=DarkOrange1] (va1) at (2.8, 0) {};
\node[vertex, draw=DarkGreen] (vz2) at (5.4, 0) {};
\node[vertex, draw=DarkRed] (vy) at (5.4, 2.8) {};

\node[vertex, draw=DarkRed] (v0top) at (2.8, 2.8) {};      
\node[vertex, draw=DarkRed] (v0z1) at (0, -2.8) {};         
\node[vertex, draw=DarkRed] (v0z2) at (5.4, -2.8) {};       

\node[above left] at (vx) {$\overline{\mathbf{x}}$};
\node[left, color=DarkGreen] at (vz1) {$z^{\scriptscriptstyle(1)}$};
\node[below=1pt, font=\tiny] at (va1) {$\begin{bmatrix} \textcolor{DarkGreen}{a^{\scriptscriptstyle(1)}} \\ \mathbf{1} \end{bmatrix}$};
\node[right, color=DarkGreen] at (vz2) {$z^{\scriptscriptstyle(2)}$};
\node[above right] at (vy) {$\mathbf{y}$};
\node[above] at (2.55, 2.8) {$\mathbf{0}$};
\node[below left] at (v0z1) {$\mathbf{0}$};
\node[below right] at (v0z2) {$\mathbf{0}$};

\draw[thick] (vx) -- (vz1);
\draw[thick] (vz1) -- (va1);
\draw[thick] (va1) -- (vz2);
\draw[thick] (vz2) -- (vy);

\node[edge stalk v] at (0, 1.4) {};
\node[edge stalk h] at (1.4, 0) {};
\node[edge stalk h] at (4.1, 0) {};
\node[edge stalk v] at (5.4, 1.4) {};

\draw[green arrow] (-0.3, 2.65) -- (-0.3, 1.55);
\draw[arrow] (-0.3, 0.15) -- (-0.3, 1.25);
\node[left, font=\scriptsize, color=DarkGreen] at (-0.35, 2.1) {$\overline{W}^{\scriptscriptstyle(1)}$};
\node[left, font=\scriptsize] at (-0.35, 0.7) {$I_{n_1}$};

\draw[arrow] (0.15, -0.4) -- (1.25, -0.4);
\draw[arrow] (2.4, -0.4) -- (1.55, -0.4);
\node[below, font=\scriptsize] at (0.65, -0.45) {$R^{z^{(1)}}$};
\node[below, font=\scriptsize] at (1.95, -0.45) {$P_{n_1}$};

\draw[green arrow] (3.2, -0.4) -- (4.0, -0.4);
\draw[arrow] (5.25, -0.4) -- (4.2, -0.4);
\node[below, font=\scriptsize, color=DarkGreen] at (3.6, -0.45) {$\overline{W}^{\scriptscriptstyle(2)}$};
\node[below, font=\scriptsize] at (4.8, -0.45) {$I_{n_2}$};

\draw[arrow] (5.7, 2.65) -- (5.7, 1.55);
\node[right, font=\scriptsize] at (5.75, 2.1) {$I_{n_2}$};
\node[right, font=\scriptsize] at (5.5, 1.4) {$\nabla U_{\text{out}}$};
\draw[arrow] (5.7, 0.15) -- (5.7, 1.25);
\node[right, font=\scriptsize] at (5.75, 0.7) {$I_{n_2}$};


\draw[thick] (v0top) -- (va1);
\node[edge stalk v] at (2.8, 1.4) {};

\draw[arrow] (2.6, 2.6) -- (2.6, 1.7);
\draw[arrow] (2.6, 0.25) -- (2.6, 1.1);
\node[left, font=\scriptsize] at (2.6, 2.15) {$I$};
\node[left, font=\scriptsize] at (2.6, 0.65) {$\sqrt{\lambda}\, P$};
\node[left, font=\scriptsize] at (2.8, 1.4) {$\nabla U_{\text{reg}}$};

\draw[thick] (vz1) -- (v0z1);
\node[edge stalk v] at (0, -1.4) {};

\draw[arrow] (-0.3, -0.15) -- (-0.3, -1.25);
\draw[arrow] (-0.3, -2.65) -- (-0.3, -1.55);
\node[left, font=\scriptsize] at (-0.35, -0.7) {$\sqrt{\lambda}\, I$};
\node[left, font=\scriptsize] at (-0.35, -2.1) {$I$};
\node[left, font=\scriptsize] at (0, -1.4) {$\nabla U_{\text{reg}}$};

\draw[thick] (vz2) -- (v0z2);
\node[edge stalk v] at (5.4, -1.4) {};

\draw[arrow] (5.7, -0.15) -- (5.7, -1.25);
\draw[arrow] (5.7, -2.65) -- (5.7, -1.55);
\node[right, font=\scriptsize] at (5.75, -0.7) {$\sqrt{\lambda}\, I$};
\node[right, font=\scriptsize] at (5.75, -2.1) {$I$};
\node[right, font=\scriptsize] at (5.6, -1.4) {$\nabla U_{\text{reg}}$};

\end{tikzpicture}%
}
}
\caption{State sheaf $\mathcal{F}^{\Omega}$}
\label{fig:reg-opinion}
\end{subfigure}
\hfill
\begin{subfigure}[b]{0.53\textwidth}
\centering
\resizebox{\linewidth}{!}{%
\begin{tikzpicture}[
    vertex/.style={circle, fill=white, draw=black, thick, minimum size=5pt, inner sep=0pt},
    edge stalk h/.style={rectangle, rounded corners=1pt, fill=white, draw=black, minimum width=16pt, minimum height=4pt, inner sep=0pt},
    edge stalk v/.style={rectangle, rounded corners=1pt, fill=white, draw=black, minimum width=4pt, minimum height=16pt, inner sep=0pt},
    arrow/.style={-{Stealth[length=4pt]}, thick},
    green arrow/.style={-{Stealth[length=4pt]}, thick, color=DarkGreen},
    every node/.style={font=\small},
    orange arrow/.style={-{Stealth[length=4pt]}, thick, color=DarkOrange1},
]

\node[vertex, draw=DarkGreen] (vW1) at (0, 3.2) {};
\node[vertex, draw=DarkRed] (vIR) at (0, 0) {};
\node[vertex, draw=DarkOrange1] (vPW2) at (3.5, 0) {};
\node[vertex, draw=DarkRed] (vII) at (6.8, 0) {};
\node[vertex, draw=DarkRed] (vI) at (6.8, 3.2) {};

\node[vertex, draw=DarkRed] (v0top) at (0, 5.8) {};       
\node[vertex, draw=DarkRed] (v0mid) at (3.5, 3.2) {};     

\node[below left, color=DarkGreen, font=\scriptsize, yshift=7pt,] at (vW1) {$\overline{W}^{\scriptscriptstyle(1)}$};
\node[below left, font=\scriptsize] at (vIR) {$(I_{n_1} \;\; R^{z^{(1)}})$};
\node[below=1pt, font=\scriptsize] at (vPW2) {$(P_{n_1} \;\; \textcolor{DarkGreen}{\overline{W}^{\scriptscriptstyle(2)}})$};
\node[below right, font=\scriptsize] at (vII) {$(I_{n_2} \;\; I_{n_2})$};
\node[above right] at (vI) {$I_{n_2}$};
\node[above] at (-.3, 5.8) {$\mathbf{0}$};
\node[above] at (3.25, 3.2) {$\mathbf{0}$};

\draw[thick] (vW1) -- (vIR);
\draw[thick] (vIR) -- (vPW2);
\draw[thick] (vPW2) -- (vII);
\draw[thick] (vII) -- (vI);

\node[edge stalk v] at (0, 1.6) {};
\node[edge stalk h] at (1.75, 0) {};
\node[edge stalk h] at (5.15, 0) {};
\node[edge stalk v] at (6.8, 1.6) {};


\draw[arrow] (-0.3, 3.) -- (-0.3, 1.75);
\node[left, font=\scriptsize] at (-0.35, 2.4) {$\overline{\mathbf{x}}$};
\draw[green arrow] (-0.3, 0.15) -- (-0.3, 1.45);
\node[left, font=\tiny] at (-0.35, 0.8) {$\begin{bmatrix} \textcolor{DarkGreen}{z^{\scriptscriptstyle(1)}} \\ 0 \end{bmatrix}$};

\draw[green arrow] (0.15, -0.4) -- (1.55, -0.4);
\node[below, font=\tiny] at (0.75, -0.45) {$\begin{bmatrix} 0 \\ \textcolor{DarkGreen}{z^{\scriptscriptstyle(1)}} \end{bmatrix}$};
\draw[green arrow] (2.65, -0.4) -- (1.95, -0.4);
\node[below, font=\tiny] at (2.4, -0.45) {$\begin{bmatrix}  \textcolor{DarkGreen}{\overline{a}^{\scriptscriptstyle(1)}}\\ 0\end{bmatrix}$};

\draw[green arrow] (4.35, -0.4) -- (4.95, -0.4);
\node[below, font=\tiny] at (4.5, -0.45) {$\begin{bmatrix} 0\\ \textcolor{DarkGreen}{\overline{a}^{\scriptscriptstyle(1)}}  \end{bmatrix}$};
\draw[green arrow] (6.65, -0.4) -- (5.35, -0.4);
\node[below, font=\tiny] at (6.1, -0.45) {$\begin{bmatrix} \textcolor{DarkGreen}{z^{\scriptscriptstyle(2)}} \\ 0 \end{bmatrix}$};

\draw[arrow] (7.1, 3.05) -- (7.1, 1.75);
\node[right, font=\scriptsize] at (7.15, 2.4) {$\mathbf{y}$};
\node[right, font=\scriptsize] at (6.95, 1.6) {$\nabla U_{\text{out}}$};
\draw[green arrow] (7.1, 0.15) -- (7.1, 1.45);
\node[right, font=\tiny] at (7.15, 0.8) {$\begin{bmatrix} 0 \\ \textcolor{DarkGreen}{z^{\scriptscriptstyle(2)} }\end{bmatrix}$};


\draw[thick] (v0top) -- (vW1);
\node[edge stalk v] at (0, 4.5) {};

\draw[arrow] (-0.3, 5.65) -- (-0.3, 4.65);
\draw[arrow] (-0.3, 3.4) -- (-0.3, 4.35);
\node[left, font=\scriptsize] at (-0.35, 5.15) {$I$};
\node[left, font=\scriptsize] at (-0.35, 3.85) {$\sqrt{\mu}\, I$};
\node[left, font=\scriptsize] at (0, 4.5) {$\nabla U_{\text{reg}}$};

\draw[thick] (v0mid) -- (vPW2);
\node[edge stalk v] at (3.5, 1.6) {};

\draw[arrow] (3.3, 3.) -- (3.3, 1.95);
\draw[arrow] (3.3, 0.3) -- (3.3, 1.25);
\node[left, font=\scriptsize] at (3.3, 2.45) {$I$};
\node[left, font=\scriptsize] at (3.3, 0.75) {$\sqrt{\mu}\, P$};
\node[left, font=\scriptsize] at (3.5, 1.6) {$\nabla U_{\text{reg}}$};

\end{tikzpicture}%
}
\caption{Parameter sheaf $\mathcal{H}^{W}$}
\label{fig:reg-structure}
\end{subfigure}

\caption{Regularization as sheaf augmentation for a 1-hidden-layer network. Each free variable is anchored to its own dummy vertex pinned at $\mathbf{0}$, with restriction maps $\sqrt{\lambda}\,I$ (or $\sqrt{\mu}\,I$) penalizing deviation from zero. Separate parent vertices allow for different types of regularization. \textbf{(a)} In the state sheaf $\mathcal{F}^\Omega$, three regularization edges anchor the free cochains $z^{(1)}$, $\overline{a}^{(1)}$, and $z^{(2)}$ toward zero, adding $\lambda I$ to the diagonal of the Laplacian. \textbf{(b)} In the parameter sheaf $\mathcal{H}^{W}$, regularization edges anchor the trainable restriction maps $\overline{W}^{(1)}$ and $\overline{W}^{(2)}$ toward zero, contributing $\mu I$ to the weight dynamics. Green elements indicate the trainable variables.}
\label{fig:regularization}
\end{figure}
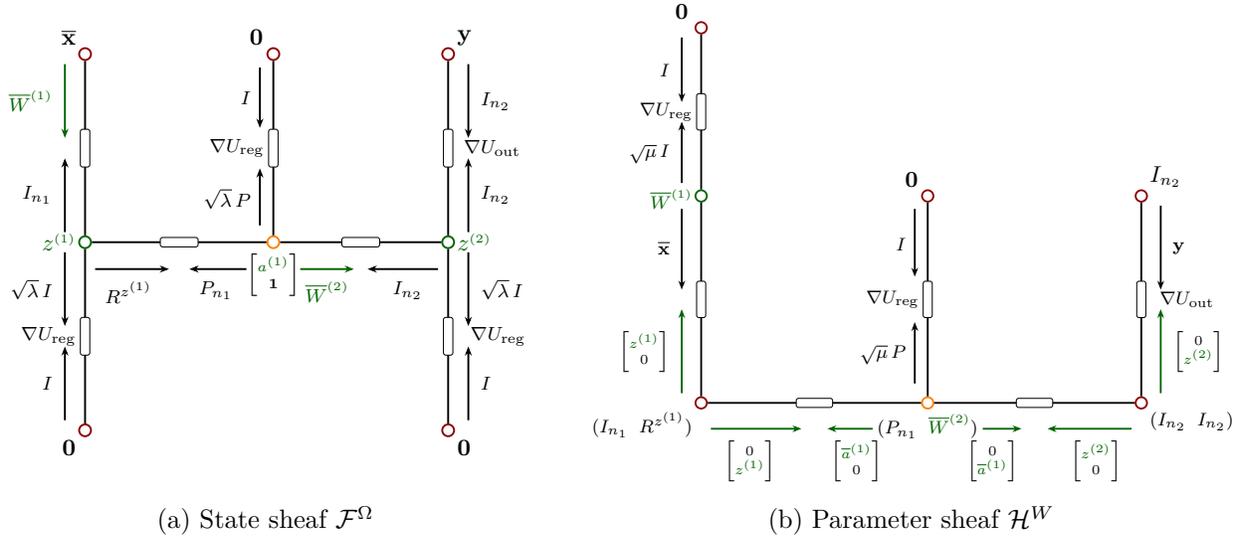

Setting $\omega_0 = 0$ and $\delta_0 = 0$ recovers~\eqref{eq:regularized-energy-training}: the anchors hold zero values, and the reluctance edges penalize the norms $\|\omega\|$ and $\|\delta\|_F$ directly. This is the standard weight decay interpretation. Setting $(\omega_0, \delta_0)$ to the initial condition produces the form used in~\cite[\S{}C.1]{bosca2026selective}, where reluctance models agents' resistance to changing their initial beliefs or communication patterns. 


\subsection{Edge potential formulas}\label{ap:edge-potentials}

\Cref{sssec:loss-functions} introduced the discrepancy framework: a loss fits this framework when $U_{\mathrm{out}}(\hat{y}, \mathbf{y}) = f(\hat{y} - \mathbf{y})$ for a convex function $f$. Each such $f$ produces a nonlinear sheaf Laplacian $L_\mathcal{F}^{\nabla f} = \delta^T \circ \nabla f \circ \delta$ in the sense of Hansen and Ghrist~\cite[\S10]{hansen2021opinion}. We collect the explicit formulas here.

Throughout, $d = \hat{y} - \mathbf{y} \in \mathbb{R}^{n_{k+1}}$ denotes the prediction error, and the gradient (or subgradient) is taken componentwise.

\paragraph{Squared error.}
The potential $f(d) = \frac{1}{2}\|d\|^2$ has gradient $(\nabla f(d))_i = d_i$, recovering the standard (linear) sheaf Laplacian. This is the loss implicit in the convergence theorems of \cref{sec:convergence}.

\paragraph{L1 loss.}
The potential $f(d) = \|d\|_1 = \sum_i |d_i|$ has subgradient $(\partial f(d))_i = \mathrm{sign}(d_i)$, with $\mathrm{sign}(0) \in [-1, 1]$. The resulting Laplacian penalizes all nonzero discrepancies equally, regardless of magnitude.

\paragraph{$p$-norm loss.}
For $1 < p < \infty$, the potential $f(d) = \frac{1}{p}\sum_i |d_i|^p$ has gradient
\begin{equation}\label{eq:p-norm-gradient}
  (\nabla f(d))_i = |d_i|^{p-2}\, d_i.
\end{equation}
The case $p = 2$ reduces to squared error. For $1 < p < 2$, the gradient grows sublinearly in $|d_i|$, placing relatively more weight on small discrepancies than squared error does. For $p > 2$, the gradient grows superlinearly, emphasizing large discrepancies.

\paragraph{Huber loss.}
The Huber potential with threshold $\tau > 0$ interpolates between squared error (for small discrepancies) and L1 (for large ones):
\begin{equation}\label{eq:huber-potential}
  f_\tau(d)_i
    = \begin{cases}
        \frac{1}{2}\, d_i^2
          & \text{if } |d_i| \leq \tau, \\[4pt]
        \tau\bigl(|d_i| - \tfrac{1}{2}\tau\bigr)
          & \text{if } |d_i| > \tau,
      \end{cases}
\end{equation}
with gradient
\begin{equation}\label{eq:huber-gradient}
  (\nabla f_\tau(d))_i
    = \mathrm{clip}(d_i, -\tau, \tau)
    = \begin{cases}
        d_i & \text{if } |d_i| \leq \tau, \\
        \tau\, \mathrm{sign}(d_i) & \text{if } |d_i| > \tau.
      \end{cases}
\end{equation}
The Huber Laplacian behaves like the standard sheaf Laplacian near consensus and like the L1 Laplacian far from it, combining sensitivity to small errors with robustness to outliers.

All four losses are convex and componentwise separable, so the resulting nonlinear Laplacians retain the local update structure: each vertex update depends only on the discrepancies with its neighbors.
The same family of potentials applies to the regularization edges discussed in \cref{ap:regularization}. Replacing the quadratic reluctance potential with an L1 or elastic net potential produces sparse regularization within the sheaf framework, as noted in \cref{sssec:regularization}.


\section{Extended experiments}\label{ap:extended-experiments}

This appendix collects the extended experimental results referenced in
\cref{sec:experiments}. \Cref{ap:ext-convergence} provides convergence
diagnostics for deeper architectures and phase-plane visualizations.
\Cref{ap:ext-training} reports full numerical results across all
benchmark configurations. \Cref{ap:ext-diagnostics} presents the
complete spectral and discord analyses.

\subsection{Convergence and dynamics}\label{ap:ext-convergence}

\Cref{fig:convergence-deep-sigmoid} extends the convergence
demonstration from \cref{subsec:convergence-experiments} to a
$[2,6,4,4,1]$ network with sigmoid output activation, exercising both
\cref{thm:convergence,thm:final}. \Cref{fig:phase-planes} displays the
phase-plane dynamics described in the main text, showing ReLU boundary
crossings concentrated in the first few hundred iterations across three
architectures of increasing depth.


\begin{figure}[ht]
  \centering
  \includegraphics[width=\textwidth]{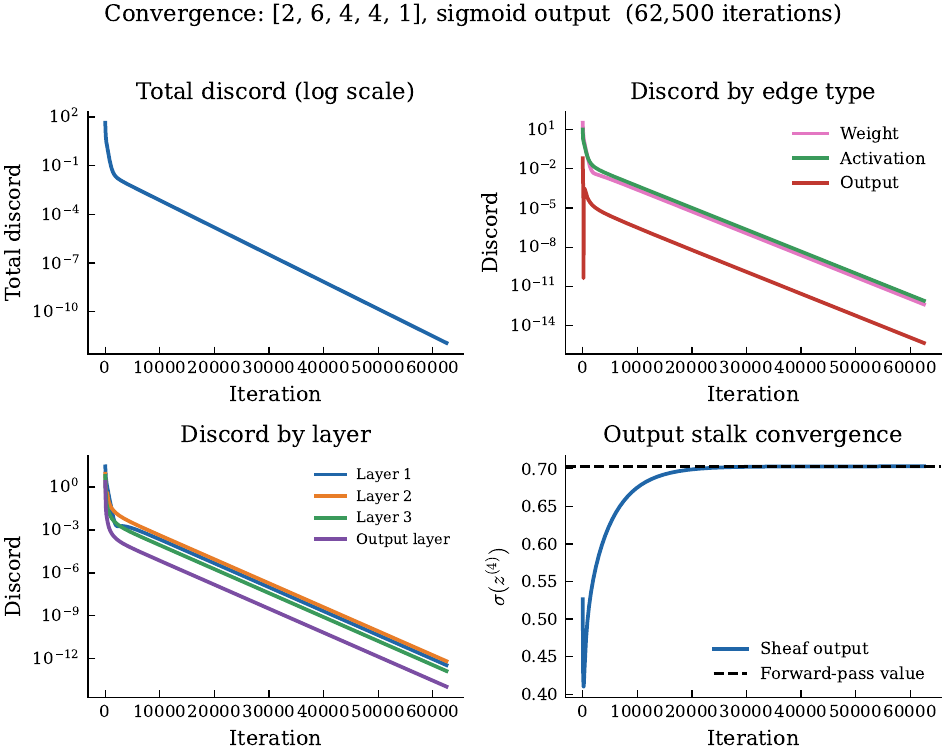}
  \caption{Convergence dynamics for a $[2,6,4,4,1]$ network with sigmoid
  output activation, $\alpha = 1.0$, $dt = 0.01$, stalks initialized
  randomly.
  \textbf{Top left:}~Total discord decreases monotonically on a log scale
  to machine precision.
  \textbf{Top right:}~Discord decomposed by edge type; the output
  activation edge (enforcing the sigmoid nonlinearity) is visible as a
  separate component.
  \textbf{Bottom left:}~Per-layer discord shows a transient phase in the first hundred iterations followed by exponential convergence.
  \textbf{Bottom right:}~The output stalk $\sigma(z^{(4)})$ converges from
  its random starting value to the forward-pass prediction (black dashed).
  This validates \cref{thm:convergence,thm:final} for deeper networks with
  nonlinear final activations.}
  \label{fig:convergence-deep-sigmoid}
\end{figure}


\begin{figure}[ht]
  \centering
  \includegraphics[width=\textwidth]{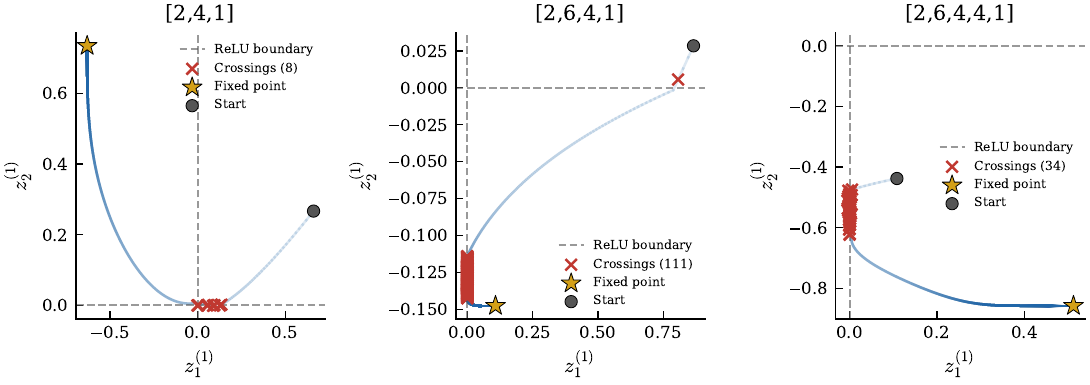}
  \caption{Phase-plane trajectories of two pre-activation components
  $(z_1^{(1)}, z_2^{(1)})$ during sheaf diffusion, for three
  architectures of increasing depth. Trajectories fade from light (early)
  to dark (late). Dashed grey lines mark the ReLU boundaries $z_j = 0$,
  red crosses mark boundary crossings, and the gold star marks the
  forward-pass fixed point. Crossings are concentrated in the first few
  hundred iterations; the trajectory then settles into a single activation
  region and converges monotonically. Deeper networks produce longer
  transients but the same qualitative pattern: ReLU boundary crossings are
  a transient phenomenon that does not impede convergence.}
  \label{fig:phase-planes}
\end{figure}

\clearpage
\subsection{Training results}\label{ap:ext-training}

\Cref{tab:full-results} reports test loss for both training methods
across all eight task$\,\times\,$depth configurations for matched iteration counts in SGD and the sheaf method ($10^5$ for 1H and $2\cdot 10^5$ for 2H).
\Cref{fig:all-training-outputs} shows the learned regression surfaces
and decision boundaries for the single-hidden-layer experiments.


\begin{table}[ht]
\centering
\caption{Training results across all task/depth configurations.}
\label{tab:full-results}
\begin{tabular}{llccccc}
\toprule
Task & Depth & Sheaf Train & SGD Train & Sheaf Test & SGD Test & Ratio \\
\midrule
  Paraboloid & 1H & 0.0756 & 0.0030 & 0.0884 & 0.0044 & 19.9$\times$ \\
  Saddle & 1H & 0.0108 & 0.0010 & 0.0106 & 0.0013 & 8.4$\times$ \\
  Circular & 1H & 0.0652 & 0.0021 & 0.0679 & 0.0031 & 22.0$\times$ \\
  Blobs & 1H & 0.1368 & 0.0657 & 0.1212 & 0.2613 & 0.5$\times$ \\
  Paraboloid & 2H & 0.0764 & 0.0065 & 0.0880 & 0.0095 & 9.2$\times$ \\
  Saddle & 2H & 0.0110 & 0.0009 & 0.0126 & 0.0011 & 11.0$\times$ \\
  Circular & 2H & 0.0180 & 0.0001 & 0.0188 & 0.0003 & 58.7$\times$ \\
  Blobs & 2H & 0.1319 & 0.0249 & 0.1323 & 0.6758 & 0.2$\times$ \\
\bottomrule
\end{tabular}
\end{table}


\begin{figure}[p]
  \centering
  \includegraphics[width=0.8\textwidth]{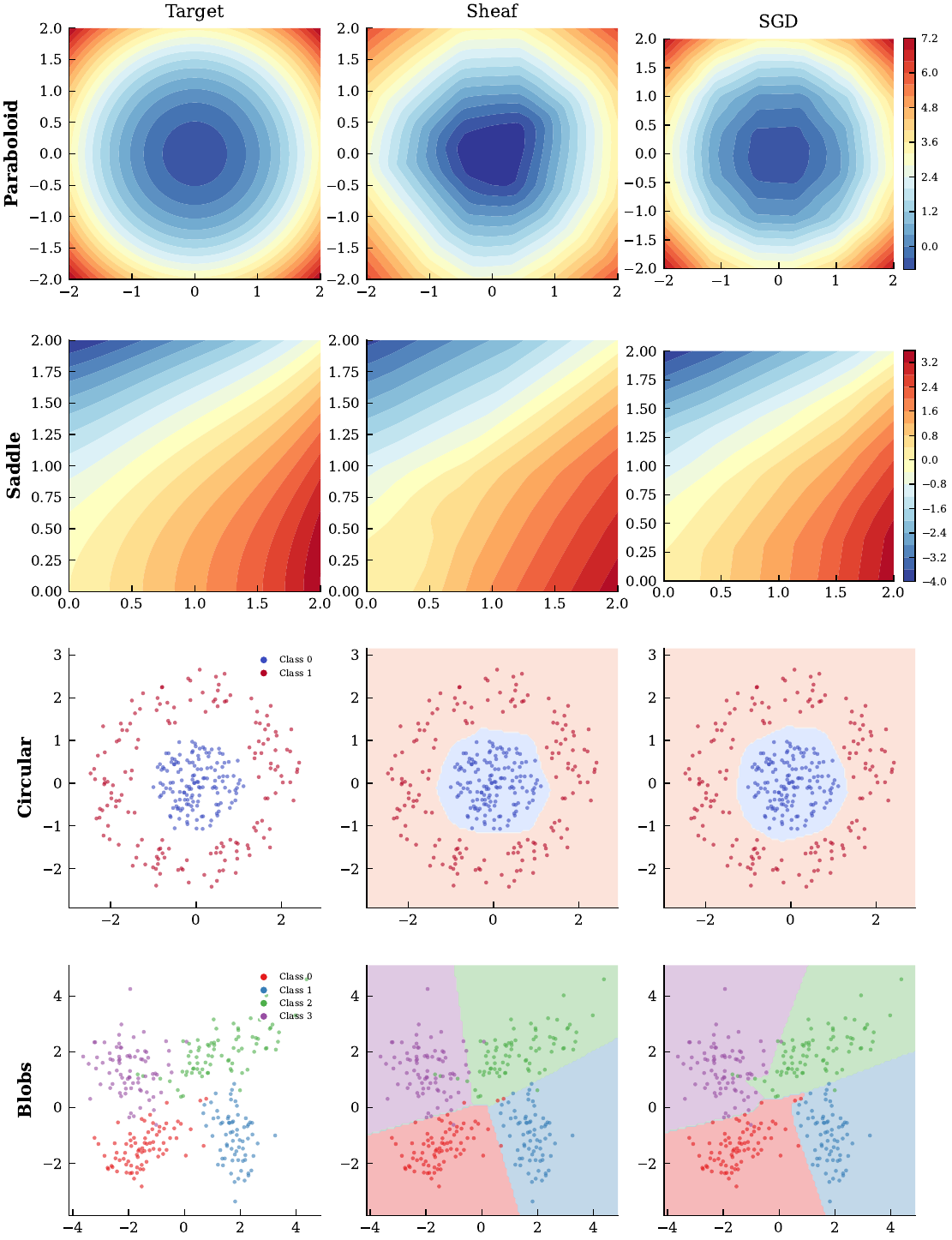}
  \caption{Learned outputs for all four benchmark tasks (single hidden
  layer). Each row shows the target function or data distribution (left),
  the sheaf-trained output (center), and the SGD-trained output (right).
  \textbf{Paraboloid and saddle}~(top two rows): filled contour plots of
  the learned regression surface on the input domain, with shared color
  scales within each row.
  \textbf{Circular and blobs}~(bottom two rows): training data colored by
  class label, with the learned decision boundary shown as a shaded
  background. Both methods produce qualitatively correct outputs across all
  tasks. The sheaf method captures the global shape but with less
  fine-grained detail in regression, consistent with the test loss ratios
  reported in \cref{tab:full-results}.}
  \label{fig:all-training-outputs}
\end{figure}

\clearpage
\subsection{Spectral and discord diagnostics}\label{ap:ext-diagnostics}

We first report the spectral analysis referenced in
\cref{subsec:diagnostics}. \Cref{tab:spectral-gap-depth} shows that
$\lambda_1$ decreases systematically with network depth at
initialization, supporting the stagnation bound discussion in
\cref{subsec:training-experiments}. \Cref{fig:eigvec-energy} displays
per-block eigenvector energy for two architectures, confirming the
Fiedler concentration on the output pre-activation stalk described in
the main text. \Cref{fig:spectral-gap-tracking,fig:spectral-comparisons,%
fig:spectral-histograms} track spectral quantities during training,
comparing sheaf-trained and SGD-trained networks.

The remaining figures present per-edge discord analyses.
\Cref{fig:discord-3H} extends the one-hidden-layer discord plots
from \cref{fig:discord-1H} to a $[2,12,8,6,1]$ network, showing that the
inactive/active decomposition persists across layers.
\Cref{fig:discord-residuals-sheaf} displays discord residuals at three
training stages, illustrating how local consistency improves as training
progresses. \Cref{tab:pinned-discord} reports aggregate discord
statistics under output-pinned inference for both training methods.

\begin{table}[ht]
\centering
\caption{Spectral gap $\lambda_1$ vs.\ depth at initialization (median over 50 inputs).}
\label{tab:spectral-gap-depth}
\begin{tabular}{lcc}
\toprule
Architecture & $\lambda_1$ (median) & $\lambda_1$ (std) \\
\midrule
  \texttt{[2,30,1]} & 0.2248 & 0.0200 \\
  \texttt{[2,10,8,1]} & 0.0602 & 0.0047 \\
  \texttt{[2,8,6,4,1]} & 0.0281 & 0.0060 \\
\bottomrule
\end{tabular}
\end{table}


\begin{figure}[ht]
  \centering
  \begin{subfigure}[t]{0.40\textwidth}
    \centering
    \includegraphics[width=\textwidth]{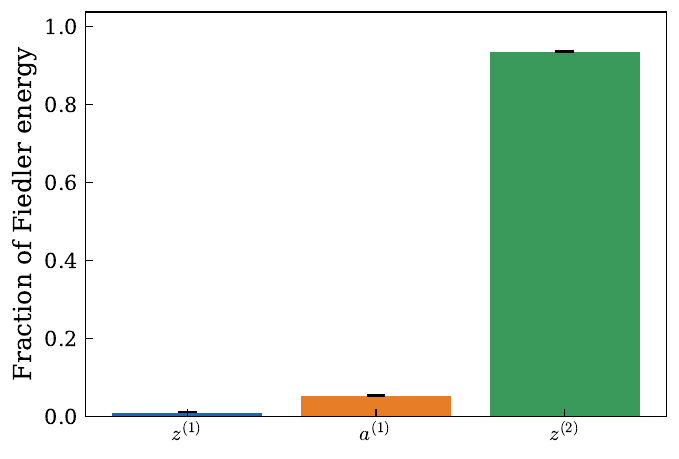}
    \caption{Fiedler energy, $[2,30,1]$.}
    \label{fig:eigvec-energy:1H-fiedler}
  \end{subfigure}
  \hfill
  \begin{subfigure}[t]{0.58\textwidth}
    \centering
    \includegraphics[width=\textwidth]{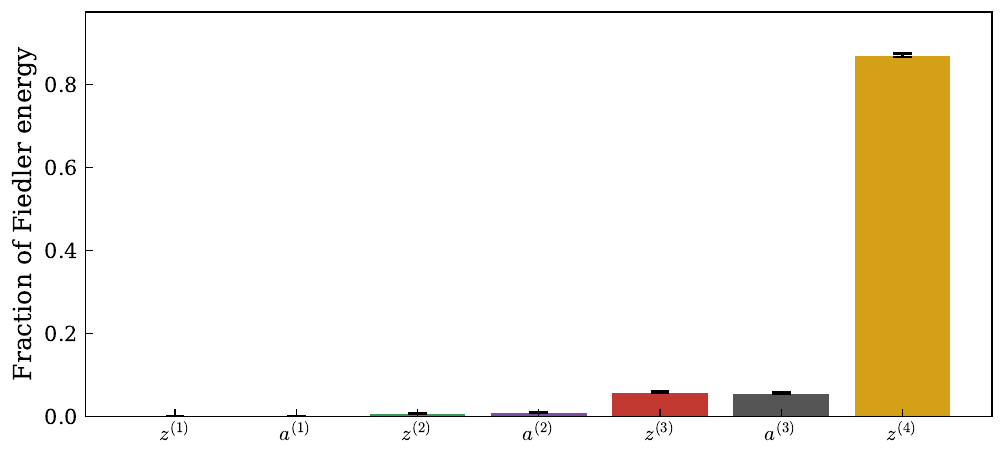}
    \caption{Fiedler energy, $[2,12,8,6,1]$.}
    \label{fig:eigvec-energy:3H-fiedler}
  \end{subfigure}

  \vspace{0.4em}

  \begin{subfigure}[t]{0.40\textwidth}
    \centering
    \includegraphics[width=\textwidth]{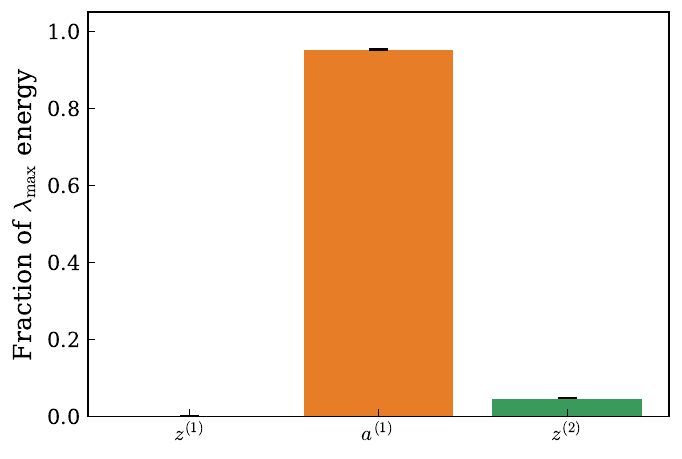}
    \caption{$\lambda_{\max}$ energy, $[2,30,1]$.}
    \label{fig:eigvec-energy:1H-lambdamax}
  \end{subfigure}
  \hfill
  \begin{subfigure}[t]{0.58\textwidth}
    \centering
    \includegraphics[width=\textwidth]{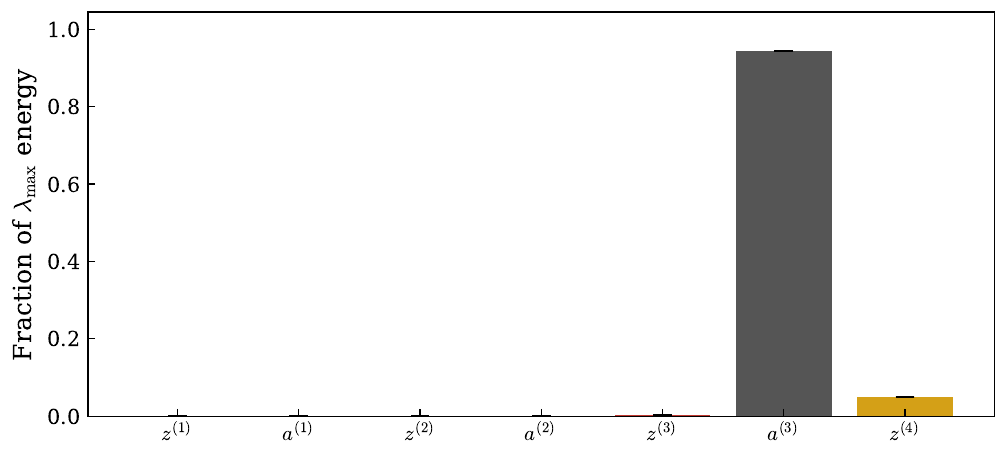}
    \caption{$\lambda_{\max}$ energy, $[2,12,8,6,1]$.}
    \label{fig:eigvec-energy:3H-lambdamax}
  \end{subfigure}

  \caption{Per-block eigenvector energy of $L_{\mathcal{F}}[\Omega,\Omega]$
    for two architectures (columns) and two eigenvectors (rows), averaged
    over $50$ random inputs with error bars showing one standard deviation.
    \textbf{Top row:}~Fiedler eigenvector ($\lambda_1$).
    In both architectures, the Fiedler energy concentrates on the output
    pre-activation stalk $z^{(k+1)}$, identifying it as the
    information-flow bottleneck of the sheaf graph.
    \textbf{Bottom row:}~$\lambda_{\max}$ eigenvector.
    The dominant eigenvector energy coincides with the post-activation block $a^{(\ell)}$ adjacent to the weight matrix with the largest operator norm.  Neither layer position nor layer width is the determining factor.}
  \label{fig:eigvec-energy}
\end{figure}


\begin{figure}[ht]
  \centering
  \begin{subfigure}[t]{0.48\textwidth}
    \centering
    \includegraphics[width=\textwidth]{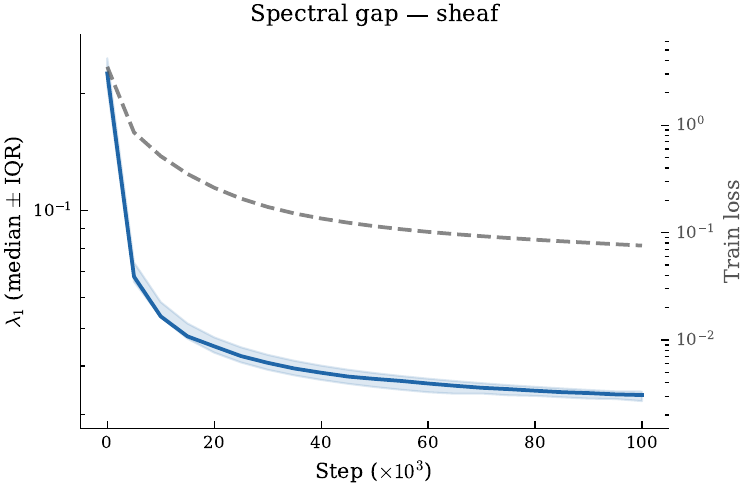}
    \caption{Sheaf training}
    \label{fig:spectral-gap-sheaf}
  \end{subfigure}
  \hfill
  \begin{subfigure}[t]{0.48\textwidth}
    \centering
    \includegraphics[width=\textwidth]{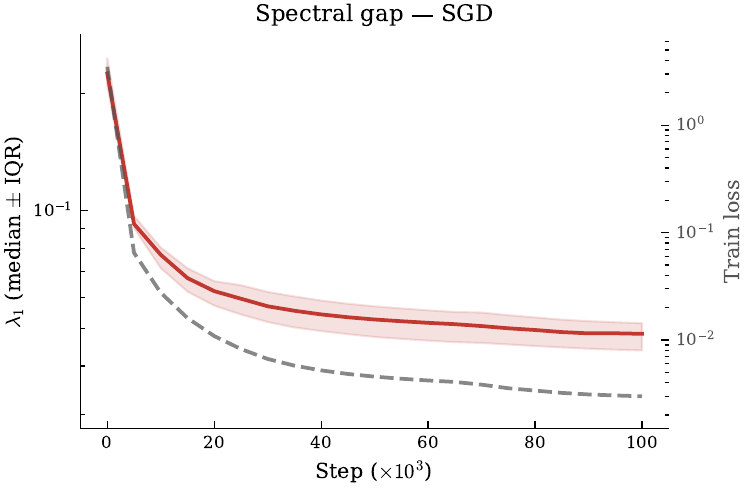}
    \caption{SGD training}
    \label{fig:spectral-gap-sgd}
  \end{subfigure}
  \caption{Spectral gap $\lambda_1$ (median $\pm$ IQR over 50 inputs) and
  training loss during training of a $[2,30,1]$ network on the paraboloid
  task. Both axes share identical limits across panels for direct
  comparison. In both methods, $\lambda_1$ drops rapidly in the first
  ${\sim}10$k steps.
  The sheaf method produces consistently smaller $\lambda_1$ throughout
  training.}
  \label{fig:spectral-gap-tracking}
\end{figure}


\begin{figure}[ht]
  \centering
  \begin{subfigure}[t]{0.32\textwidth}
    \centering
    \includegraphics[width=\textwidth]{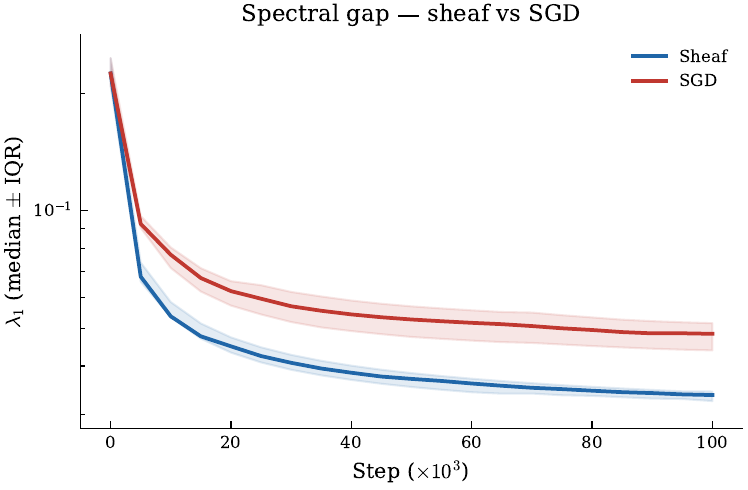}
    \caption{$\lambda_1$}
    \label{fig:spectral-gap-comparison}
  \end{subfigure}
  \hfill
  \begin{subfigure}[t]{0.32\textwidth}
    \centering
    \includegraphics[width=\textwidth]{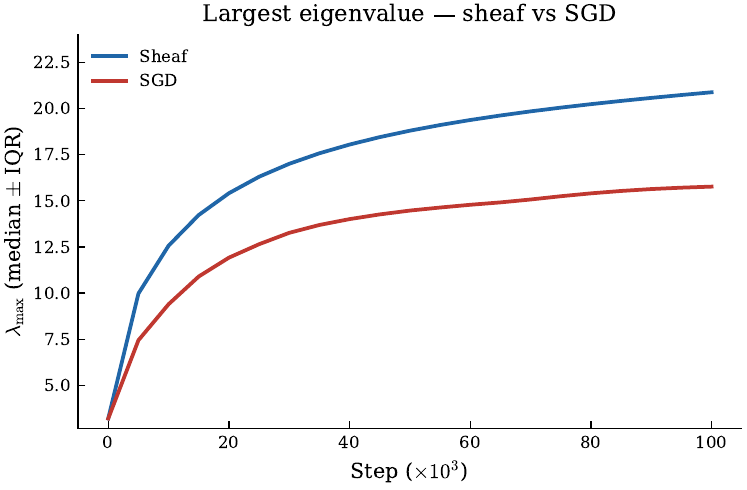}
    \caption{$\lambda_{\max}$}
    \label{fig:spectral-lmax-comparison}
  \end{subfigure}
  \hfill
  \begin{subfigure}[t]{0.32\textwidth}
    \centering
    \includegraphics[width=\textwidth]{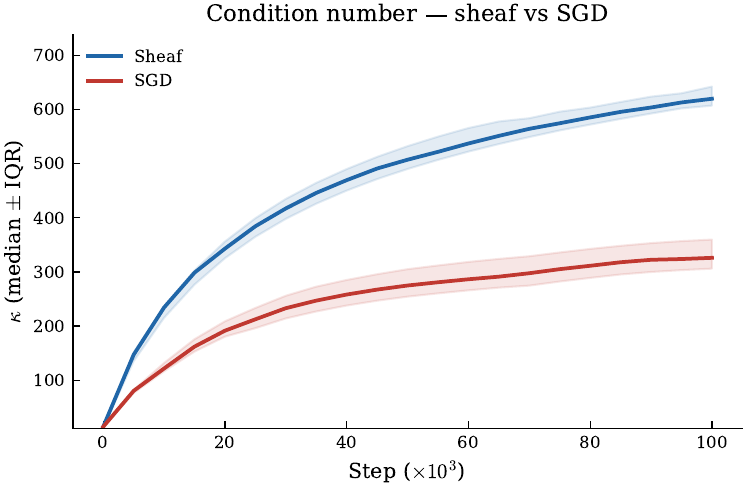}
    \caption{$\kappa = \lambda_{\max} / \lambda_1$}
    \label{fig:spectral-kappa-comparison}
  \end{subfigure}
  \caption{Spectral statistics of $L_{\mathcal{F}_t}[\Omega,\Omega]$
  during training (median $\pm$ IQR, $[2,30,1]$ on paraboloid). Both
  methods shift the spectrum toward smaller $\lambda_1$ and larger
  $\lambda_{\max}$ as training progresses, but the sheaf method produces
  consistently smaller spectral gaps \textbf{(a)} and larger condition
  numbers \textbf{(c)}. The largest eigenvalue \textbf{(b)} grows in both
  methods, with sheaf-trained
  networks producing moderately larger $\lambda_{\max}$.}
  \label{fig:spectral-comparisons}
\end{figure}


\begin{figure}[ht]
  \centering
  \begin{subfigure}[t]{0.48\textwidth}
    \centering
    \includegraphics[width=\textwidth]{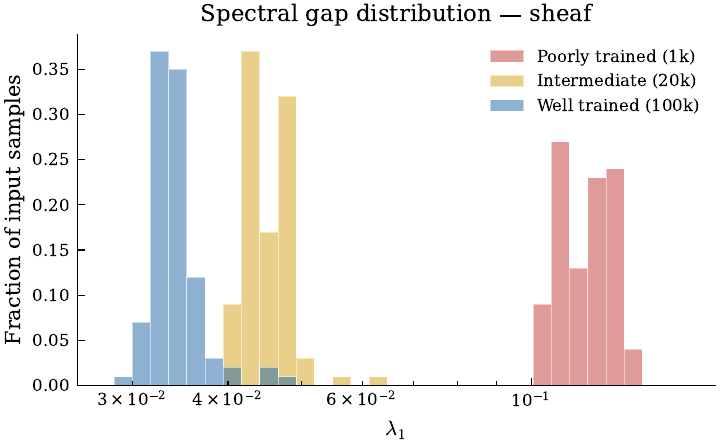}
    \caption{Sheaf training}
    \label{fig:spectral-hist-sheaf}
  \end{subfigure}
  \hfill
  \begin{subfigure}[t]{0.48\textwidth}
    \centering
    \includegraphics[width=\textwidth]{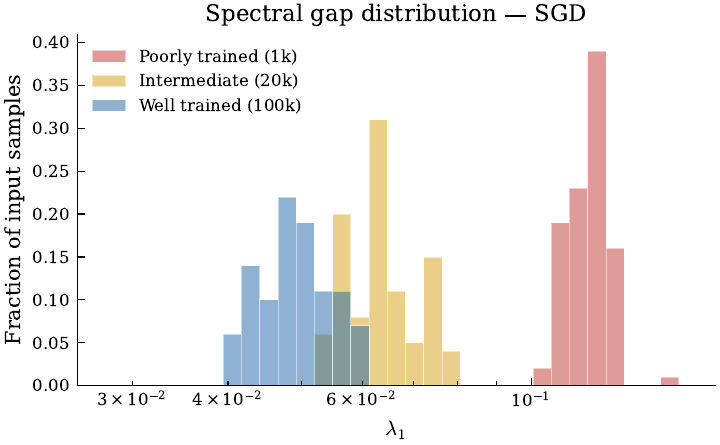}
    \caption{SGD training}
    \label{fig:spectral-hist-sgd}
  \end{subfigure}
  \caption{Distribution of the spectral gap $\lambda_1$ across 100 random
  inputs at three training stages (poorly trained at 1k steps,
  intermediate at 20k, well trained at 100k). Log-spaced bins; shared
  $x$-axis range across both panels. In both methods, the distribution
  shifts leftward (toward smaller $\lambda_1$) as training progresses,
  tightening around a stable value. The sheaf-trained distributions are
  shifted further left at all stages, consistent with the tracking plots
  in \cref{fig:spectral-gap-tracking}.}
  \label{fig:spectral-histograms}
\end{figure}




\begin{figure}[p]
  \centering
  \begin{subfigure}[t]{0.24\textwidth}
    \centering
    \includegraphics[width=\textwidth]{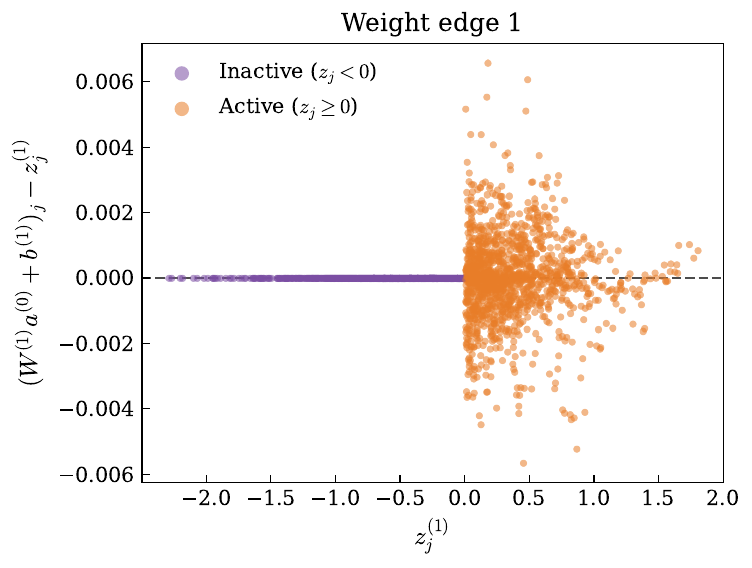}
    \caption{Weight edge~1.}
    \label{fig:discord-3H:w1}
  \end{subfigure}
  \hfill
  \begin{subfigure}[t]{0.24\textwidth}
    \centering
    \includegraphics[width=\textwidth]{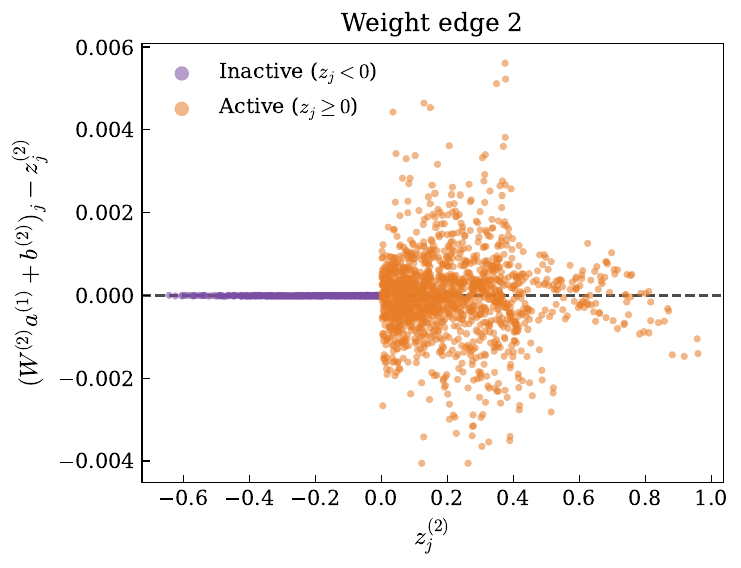}
    \caption{Weight edge~2.}
    \label{fig:discord-3H:w2}
  \end{subfigure}
  \hfill
  \begin{subfigure}[t]{0.24\textwidth}
    \centering
    \includegraphics[width=\textwidth]{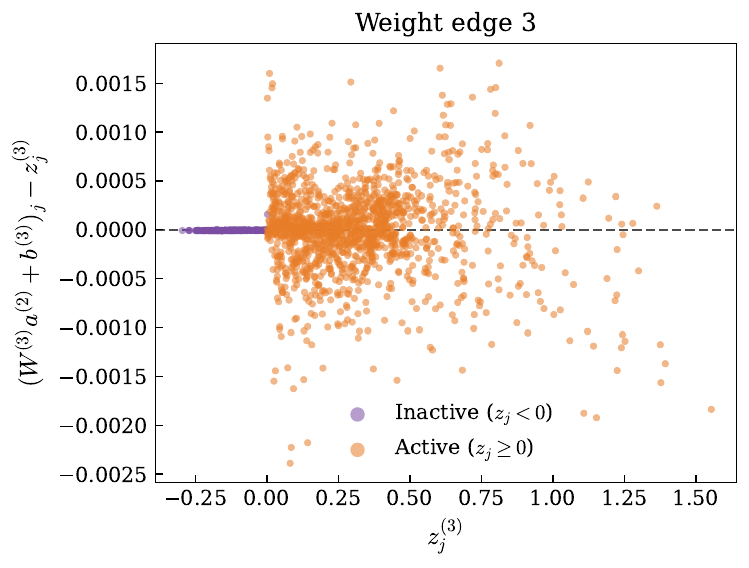}
    \caption{Weight edge~3.}
    \label{fig:discord-3H:w3}
  \end{subfigure}
  \hfill
  \begin{subfigure}[t]{0.24\textwidth}
    \centering
    \includegraphics[width=\textwidth]{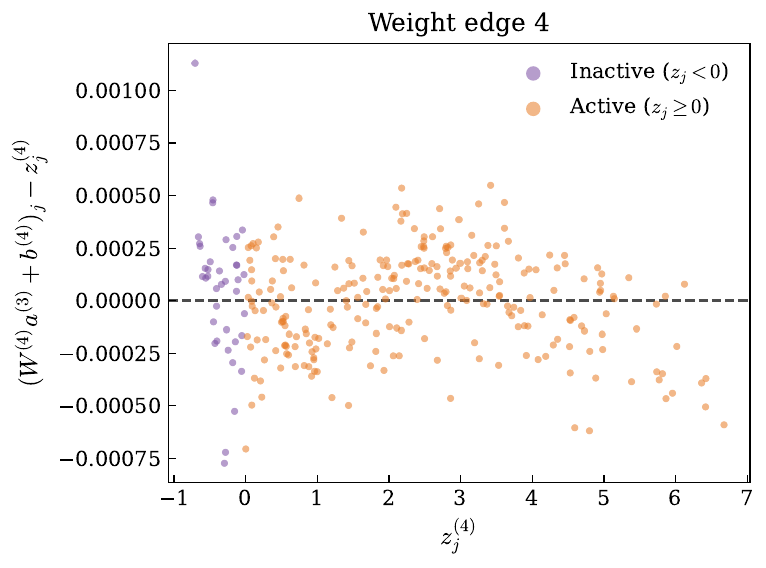}
    \caption{Weight edge~4.}
    \label{fig:discord-3H:w4}
  \end{subfigure}

  \vspace{0.6em}

  \begin{subfigure}[t]{0.32\textwidth}
    \centering
    \includegraphics[width=\textwidth]{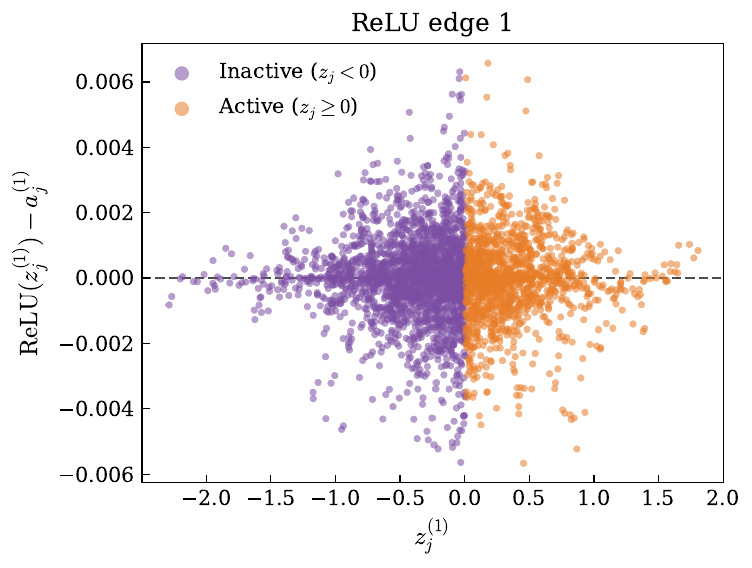}
    \caption{ReLU edge~1.}
    \label{fig:discord-3H:r1}
  \end{subfigure}
  \hfill
  \begin{subfigure}[t]{0.32\textwidth}
    \centering
    \includegraphics[width=\textwidth]{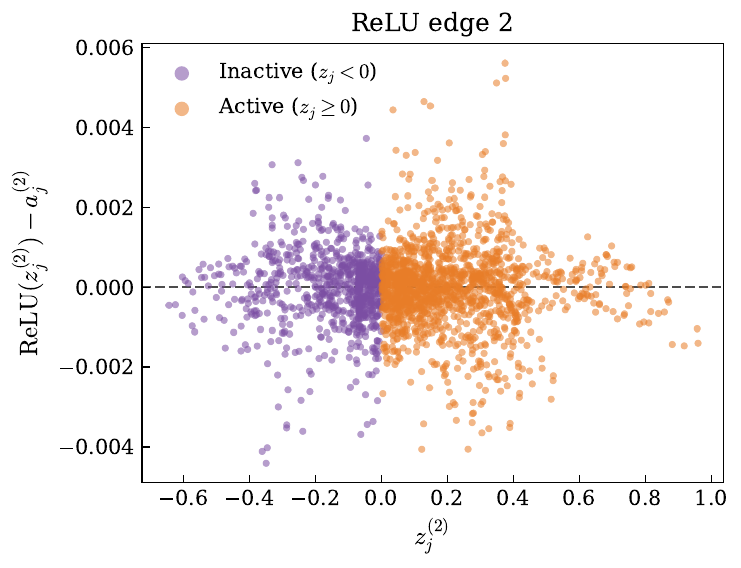}
    \caption{ReLU edge~2.}
    \label{fig:discord-3H:r2}
  \end{subfigure}
  \hfill
  \begin{subfigure}[t]{0.32\textwidth}
    \centering
    \includegraphics[width=\textwidth]{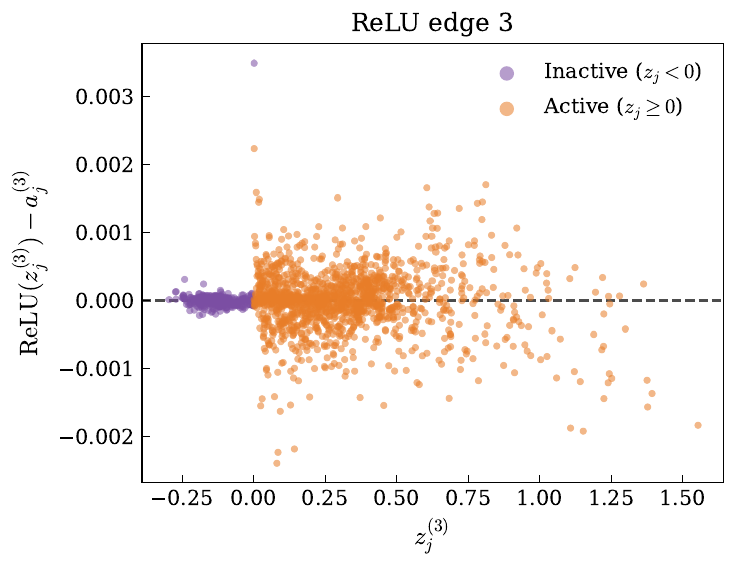}
    \caption{ReLU edge~3.}
    \label{fig:discord-3H:r3}
  \end{subfigure}

  \caption{Per-coordinate edge discord for a well-trained $[2,12,8,6,1]$
    sheaf network (three hidden layers) on the paraboloid task.
    \textbf{(a--d)}~Weight edge discord
    $(W^{(\ell)} a^{(\ell-1)} + b^{(\ell)})_j - z_j^{(\ell)}$ vs.\
    pre-activation $z_j^{(\ell)}$ for each of the four weight edges.
    Across the three first weight edges, inactive coordinates (purple) cluster tightly on the
    zero line, confirming that the ReLU edge exerts no force on dead neurons
    and the weight edge achieves zero discord. The last weight edge does not follow this pattern because the final activation in this case is the identity, so this edge acts as the output edge.
    Active coordinates (orange) show a symmetric spread around zero whose
    magnitude varies by layer, reflecting the local force balance between
    weight and activation edges at equilibrium.
    \textbf{(e--g)}~ReLU edge discord
    $\mathrm{ReLU}(z_j^{(\ell)}) - a_j^{(\ell)}$ vs.\
    $z_j^{(\ell)}$ for the three activation edges. We found no clear structure. Across all ReLU edges, we observe a decay in the number of inactive neurons (negative values of $z_j$).}
  \label{fig:discord-3H}
\end{figure}


\begin{figure}[p]
  \centering
  \includegraphics[width=0.9\textwidth]{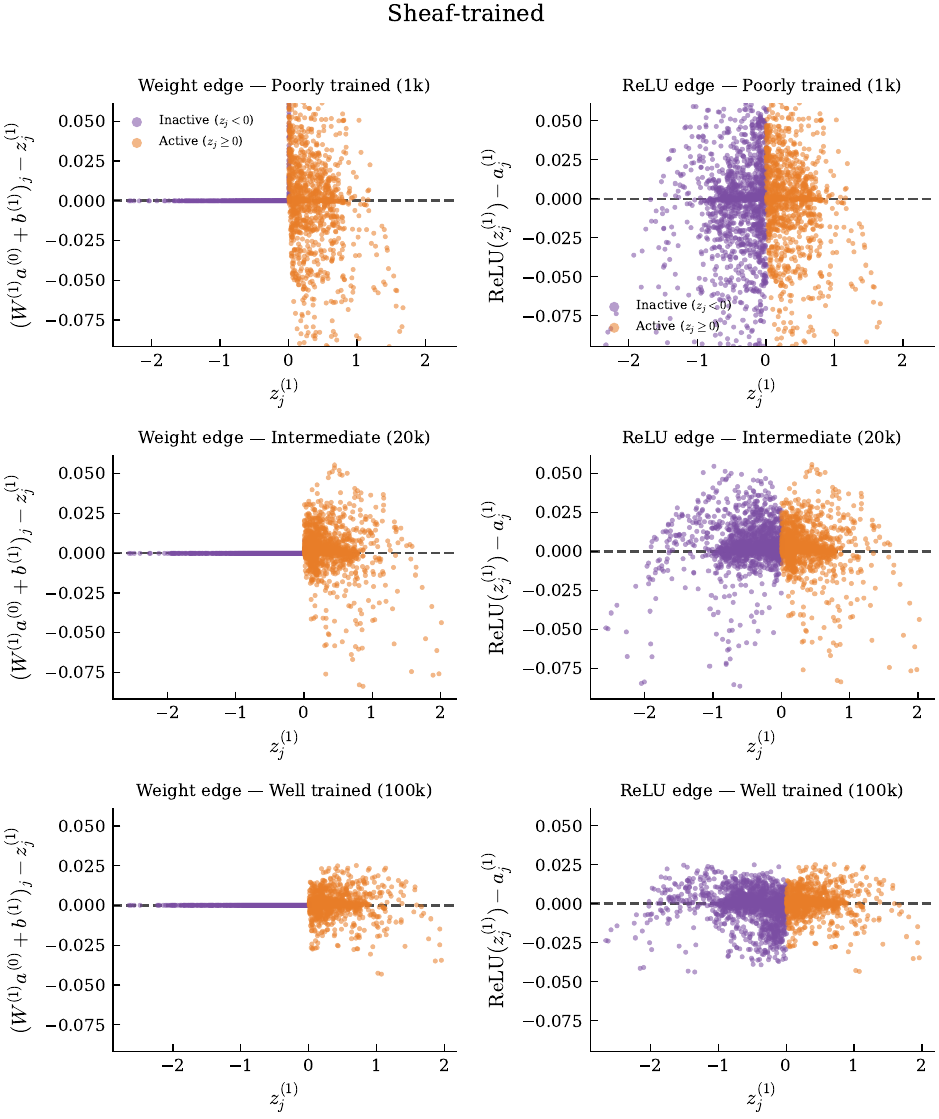}
  \caption{Per-edge discord residuals at three training stages for a
  \emph{sheaf-trained} $[2,30,1]$ network on the paraboloid task. Each row
  corresponds to a training stage (1k, 20k, 100k steps); output stalks are
  hard-pinned to true labels. $y$-axis limits are fixed across all rows
  (set by the intermediate stage) to show the decrease in discord magnitude
  with training. Points are colored by activation state: active
  ($z_j \geq 0$, orange) or inactive ($z_j < 0$, purple).
  \textbf{Left:}~Weight edge discord residual
  $(W^{(1)} a^{(0)} + b^{(1)})_j - z_j^{(1)}$.
  \textbf{Right:}~ReLU edge discord residual
  $\mathrm{ReLU}(z_j^{(1)}) - a_j^{(1)}$.
  In the well-trained model (bottom), inactive coordinates lie on the zero
  line and active coordinates cluster near zero, confirming local
  consistency at equilibrium.}
  \label{fig:discord-residuals-sheaf}
\end{figure}

\begin{figure}[p]
  \centering
  \includegraphics[width=0.9\textwidth]{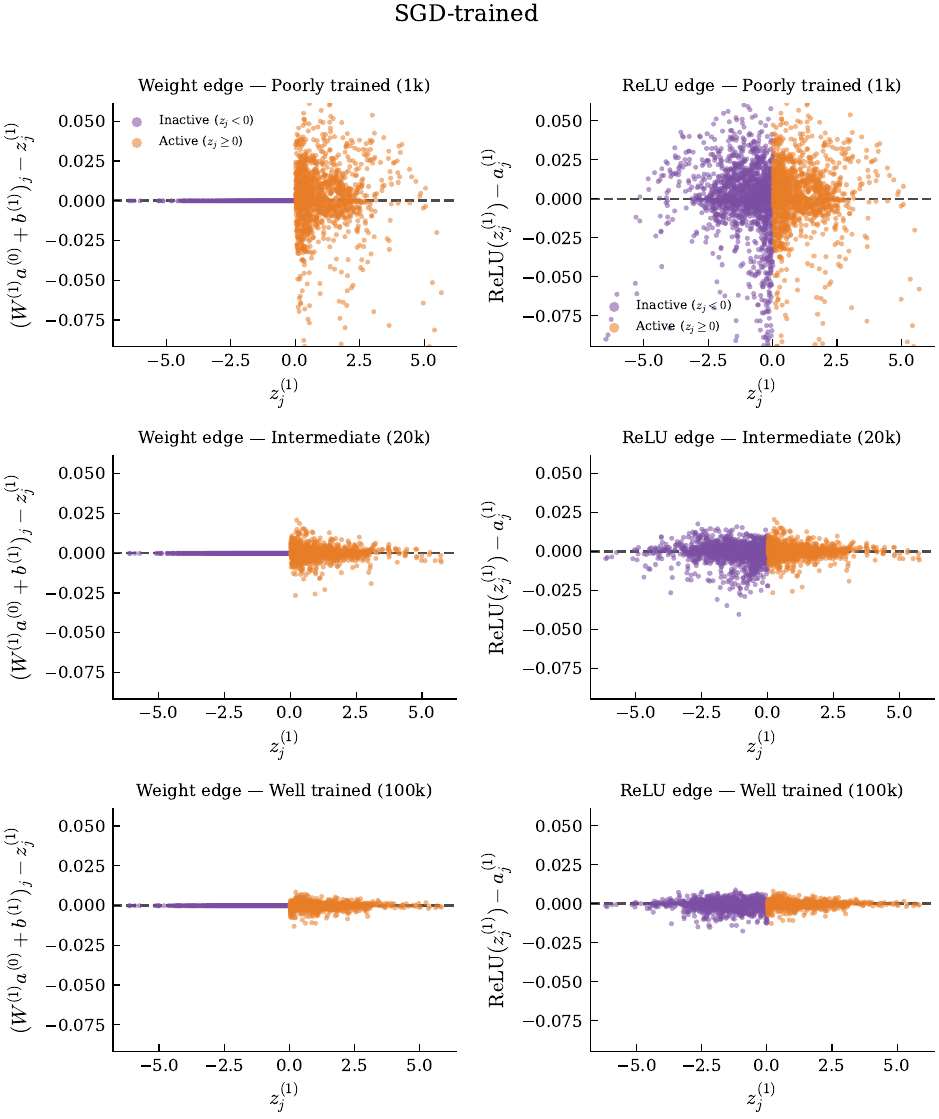}
  \caption{Same as \cref{fig:discord-residuals-sheaf} but for an
  \emph{SGD-trained} network with the same architecture and training
  budget. $y$-axis limits are shared with the sheaf figure for direct
  comparison. The SGD-trained network's weights are embedded in a
  \texttt{NeuralSheaf} and analysed via pinned diffusion under the same
  protocol. Qualitative differences in the residual patterns reflect the
  different optimization landscapes explored by the two methods.}
  \label{fig:discord-residuals-sgd}
\end{figure}


\begin{table}[ht]
\centering
\caption{Per-edge pinned discord (mean $\pm$ std over 50 test inputs). Output stalk hard-pinned to true labels.}
\label{tab:pinned-discord}
\begin{tabular}{lcc}
\toprule
Edge & Sheaf-trained & SGD-trained \\
\midrule
  W1 & 0.0007 $\pm$ 0.0010 & 0.0000 $\pm$ 0.0001 \\
  W2 & 0.0001 $\pm$ 0.0002 & 0.0000 $\pm$ 0.0000 \\
  R1 & 0.0027 $\pm$ 0.0034 & 0.0002 $\pm$ 0.0003 \\
\midrule
  Total & 0.0035 & 0.0003 \\
\bottomrule
\end{tabular}
\end{table}

\end{document}